\documentclass[10pt]{iopart}

\pdfoutput=1
\usepackage{dsfont}
\usepackage{url}
\usepackage{algpseudocode,algorithm,algorithmicx}
\usepackage{subfigure}
\usepackage{upgreek}

\usepackage{etex} 
\usepackage{morefloats} 

\usepackage{mystyle}
\usepackage{notations} 
\graphicspath{{./figures/}}

\begin{document}

\title[The Sliding Frank-Wolfe Algorithm]{The Sliding Frank-Wolfe Algorithm \\ and its Application to Super-Resolution Microscopy}

\author{Quentin Denoyelle$^1$, Vincent Duval$^2$, Gabriel Peyr{\'e}$^3$, and Emmanuel Soubies$^4$}

\address{$^1$CEREMADE, Univ. Paris-Dauphine \\ $^2$INRIA Paris, MOKAPLAN \\  $^3$CNRS \& ENS \\ $^4$Biomedical Imaging Group, EPFL }
\ead{denoyelle@ceremade.dauphine.fr, vincent.duval@inria.fr, gabriel.peyre@ens.fr, emmanuel.soubies@epfl.ch }
\vspace{10pt}

%
%
%
%
%

%

\begin{abstract}
This paper showcases the theoretical and numerical performance of the Sliding Frank-Wolfe, which is a novel optimization algorithm to solve the BLASSO sparse spikes super-resolution problem. 
The BLASSO is a continuous (\ie off-the-grid or grid-less) counterpart to the well-known $\ell^1$ sparse regularisation method (also known as LASSO or Basis Pursuit). Our algorithm is a variation on the classical Frank-Wolfe (also known as conditional gradient) which follows a recent trend of interleaving convex optimization updates (corresponding to adding new spikes) with non-convex optimization steps (corresponding to moving the spikes).
Our main theoretical result is that this algorithm terminates in a finite number of steps under a mild non-degeneracy hypothesis. 
We then target applications of this method to several instances of single molecule fluorescence imaging modalities, among which certain approaches rely heavily on the inversion of a Laplace transform.
Our second theoretical contribution is the proof of the exact support recovery property of the BLASSO to invert the 1-D Laplace transform in the case of positive spikes.
On the numerical side, we conclude this paper with an extensive study of the practical performance of the Sliding Frank-Wolfe on different instantiations of single molecule fluorescence imaging, including convolutive and non-convolutive (Laplace-like) operators. This shows the versatility and superiority of this method with respect to alternative sparse recovery technics. 
\end{abstract}

%

\section{Introduction}

\subsection{Super-Resolution using the BLASSO} \label{sec-intro-blasso-pw}

Super-resolution consists in retrieving the fine scale details of a possibly noisy signal from coarse scale information. The  importance of recovering the high frequencies of a signal comes from the fact that there is often a physical blur in the acquisition process, such as diffraction in optical systems, wave reflection in seismic imaging or spikes recording from neuronal activity. 

In resolution theory~\cite{den-resolution1997}, the two-point resolution criterion defines the ability of a system to resolve two points of equal intensities. 
 It is defined as a distance, namely  the Rayleigh criterion, which only depends on the system. In the case of the ideal low-pass filter (\ie convolution with the Dirichlet kernel) with cutoff frequency $f_c$, the Rayleigh criterion is $1/f_c$. Then, super-resolution in signal processing consists in developing techniques which enable to retrieve information below the Rayleigh criterion.

Let us introduce in a more formal way the problem which will be the framework of this article. Let $\Pos$ be a connected subset of $\RR^d$ with non-empty interior or the $d$-dimensional torus $\TT^d$ ($d\in\NN^*$) and $\radon$ the Banach space of bounded Radon measures on $\Pos$. The latter can be seen as the topological dual of $\ContX(X,\RR)$, the space  of continuous functions on $\Pos$ that vanish at infinity. We consider a given integral operator $\Phi:\radon\to\Obs$, where $\Obs$ is a separable Hilbert space, whose kernel $\phi$ is supposed to be a smooth function (see Definition~\ref{sec:intro-def:admkernel} for the technical assumptions made on $\phi$), \ie 
\begin{align}\label{def-Phi}
	\forall m\in\radon,\quad\Phi m\eqdef\int_\Pos \phi(\pos)\d m(\pos).
\end{align}
The operator $\Phi$ models the acquisition process. It includes translation-invariant operators such as convolutions (\ie $\phi(x)=\tilde{\phi}(\cdot-x)$) as well as non-translation invariant operators such as the Laplace transform ($X=\RR_+^*$ and $\phi(x) = (t\mapsto e^{-tx} ) \in L^2(\RR_+)$) considered in the present paper.

%

%
%
The sparse spikes super-resolution problem aims at recovering an approximation of an unknown input discrete measure $\measO\eqdef\sum_{i=1}^{N} \ampOi \dirac{\posOi}$ from noisy measurements $\obsw\eqdef\obsO+w$ where $\obsO \eqdef \Phi\measO$ and $w \in \Hh$ models the acquisition noise. Here $\ampOi\in\RR$ are the amplitudes of the Dirac masses at positions $\posOi\in\Pos$.
This is an ill-posed inverse problem and the BLASSO is a way to solve it in a stable way by introducing a sparsity-enforcing convex regularization.

\subsubsection{From the LASSO to the BLASSO}

%
%
%

The common practice in sparse spike recovery relies on $\lun$ regularization which is known as LASSO  in statistic~\cite{tibshirani-regression1994} or  basis pursuit  in the signal processing community~\cite{chen-atomic1998}.  Given a grid of possible positions, the reconstruction problem is addressed as the minimization of a quadratic error subject to an $\lun$ penalization. The $\lun$ prior  provides solutions with few nonzero coefficients and can be computed efficiently with convex optimization methods. Moreover, recovery guarantees have been proved under certain assumptions~\cite{donoho-superresolution1992}.


Following recent works (see for instance \cite{bhaskar-atomic2011,bredies-inverse2013,candes-towards2013,deCastro-exact2012,duval-exact2013}), we consider instead  sparse spike estimation methods  which operate over a continuous domain, \ie without resorting to some sort of discretization on a grid. The inverse problem is solved over the space of Radon measures which is a non-reflexive Banach space. This continuous \guillemet{grid-free} setting makes the mathematical analysis easier and allows us to make precise statement about the location of the recovered spikes. 

The technique that we are considering in this paper consists in solving a convex optimization problem that uses the total variation norm, which is  the counterpart of the $\lun$ norm for measures. It  favors the emergence of spikes in the solution and is defined by
\begin{align}\label{sec:intro-eq:TV}
  	\forall m \in \radon, \quad
	\normTVX{m}	\eqdef	\usup{\psi \in \ContX} \enscond{ \int_\Pos \psi \d m }{ \normLi{\psi} \leq 1 }.
\end{align}
In particular, for $\measO\eqdef\sum_{i=1}^{N} \ampOi \dirac{\posOi}$,
\eq{
	\normTVX{m_{\ampO,\posO}}=\normu{\ampO},
}
which shows in a way that the total variation norm generalizes the $\lun$ norm to the continuous setting of measures (\ie no discretization grid is required).

When no noise is contaminating the data, one considers the classical basis pursuit, defined originally in~\cite{chen-atomic1998} in a finite dimensional setting, written here over the space of Radon measures
\begin{align}\label{eq-blasso-noiseless}
  \umin{m \in \radon} 
\normTVX{m} \quad \mbox{s.t.}\quad \Phi m=\obsO \tag{$\bpursuit$}.
\end{align}
This problem is studied in~\cite{candes-towards2013}, in the case where $\Phi$ is an ideal low-pass filter on the torus $X=\TT$.

When the signal is noisy, \ie when one observes $y=\obsO+w$, with $w\in \Obs$, we may rather consider the problem
\begin{align}\label{eq-blasso-noisy}
	\umin{m \in \radon} \frac{1}{2} \normObs{\Phi m - y}^2 + \lambda \normTVX{m} \tag{$\blasso$}.
\end{align}
Here $\lambda>0$ is a parameter that should be adapted to the noise level $\normObs{w}$. This problem is coined as BLASSO~\cite{deCastro-exact2012}.

\subsubsection{BLASSO performance analysis}

In order to quantify the recovery performance of the methods $\bpursuit$ and $\blasso$, the following two questions arise:
\begin{enumerate}
 	\item Does the solutions of $\bpursuit$ recover the input measure $\measO$ ?
 	\item How close is the solution of $\blasso$ to the solution of $\bpursuit$ ?
\end{enumerate}

When the amplitudes of the spikes are arbitrary complex numbers, the answers to the above questions require a large enough minimum separation distance $\De(\measO)$ between the spikes, where 
\begin{align}\label{min-sep-dist}
	\De(\measO) \eqdef \umin{i \neq j} \dX(\posOi,\posOj).
\end{align}
When $\Pos=\TT$, $\dX$ is the geodesic distance on the circle
\begin{align}\label{dist-T}
	\forall x,y\in\RR, \quad \dX(x+\ZZ,y+\ZZ)=\min_{k\in\ZZ} |x-y+k|.
\end{align}
In~\cite{candes-towards2013}, the authors shows that for the ideal low-pass filter,  $\measO$ is the unique solution of $\bpursuit$ provided that $\De(\measO)\geq \frac{C}{f_c}$ where $C>0$ is a universal constant and $f_c$ the cutoff frequency of the ideal low-pass filter. In the same paper, it is shown that $C\leq 2$ when $\ampO\in\CC^N$ and $C\leq 1.87$ when $\ampO\in \RR^N$. In~\cite{carlos-super2015}, the constant $C$ is further refined to  $C\leq 1.26$ when $\ampO\in \RR^N$. Suboptimal lower bounds on $C$ were given in~\cite{duval-exact2013,tang2015resolution}. Moreover, it was recently shown in~\cite{ferreira-2018tight} that necessarily $C\geq 1$ in the sense that for all $\varepsilon>0$, and for $f_c$ large enough, there exist measures with $\De(\measO)\geq (1-\varepsilon)/f_c$ which are not identifiable using~\eqref{eq-blasso-noiseless}.

The second question receives partial answers in~\cite{azais-spike2014,bredies-inverse2013,candes-superresolution2013,fernandez-support2013}. In~\cite{bredies-inverse2013}, it is shown that if the solution of $\bpursuit$ is unique then the measures recovered by $\blasso$ converge in the 
weak-* sense to the solution of $\bpursuit$ when $\lambda \rightarrow 0$ and $\normObs{w}/\lambda \rightarrow 0$. In~\cite{candes-superresolution2013}, the authors measure the reconstruction error using the $\Ldeux$ norm of an ideal low-pass filtered version of the recovered measures. 
In~\cite{azais-spike2014,fernandez-support2013}, error bounds are given on the locations of the recovered spikes with respect to those of the input measure $\measO$.
However, those works provide little information about the geometrical structure of the measures recovered by $\blasso$. That point is addressed in~\cite{duval-exact2013} where the authors show that under the \emph{Non Degenerate Source Condition}, there exists a unique solution to $\blasso$ with the exact same number of spikes as the original measure provided that $\lambda$ and $\normObs{w}/\lambda$ are small enough. Moreover in that regime, this solution converges to the original measure when the noise drops to zero.

\paragraph{BLASSO for positive spikes.}

For positive spikes (\ie $\ampOi>0$), the picture is radically different. Exact recovery of $\measO$ without noise (\ie $(w,\la)=(0,0)$) holds whatever the distance between the spikes~\cite{deCastro-exact2012}, but stability constants explode as $\De(\measO) \to 0$.
However, the authors in~\cite{candes-stable2014}  show  that stable recovery is obtained if the signal-to-noise ratio grows faster than $O(1/\De^{2N})$. This  closely matches the optimal lower bounds of $O(1/\De^{2N-1})$ obtained by combinatorial methods~\cite{demanet-recoverability2014}.

Finally, provided a certain nondegeneracy condition, it was recently shown in~\cite{Denoyelle2017}  that support recovery is guaranteed in the presence of noise if  the signal-to-noise ratio grows faster than $O(1/\De^{2N-1})$.

\subsection{Solving the BLASSO}
As the BLASSO is an optimization problem over the infinite dimensional space of Radon measures $\radon$, its resolution is challenging.  We review in this section the existing approaches to tackle this problem. They can be roughly divided into three main families although there exists a flurry of generalizations and extensions that must be considered separately.

\myparagraph{Fixed spatial discretization} A common approach consists in constraining the measure to be supported on a grid. This  leads  to a finite dimensional convex optimization problem---known as LASSO~\cite{tibshirani-regression1994} or basis pursuit~\cite{chen-atomic1998}---for which there exist numerous solvers. These include the block-coordinate descent (BCD) algorithm~\cite{tseng-convergence2001,wu-coordinate2008}, the homotopy/LARS algorithm~\cite{efron-lars2004,soussen-homotopy2015}, or proximal forward-backward splitting algorithms \cite{combettes-fb2005} such as  the Iterative Soft Thresholding (IST)~\cite{daubechies-ist}.  Although simples to implement, the latters are in general slow to converge (the error in the objective function is typically of the order of $O(1/k)$, where $k$ is the number of iterations)~\cite{daubechies-ist,donoho-adapting1999,figueiredo-em2003}. However, there exist accelerated versions such as FISTA~\cite{beck-fista2009}, which benefit from a better non-asymptotic rate of convergence ($O(1/k^2)$). Finally, it is noteworthy that  these proximal methods enjoy a linear asymptotic rate (see for instance~\cite{LiangLinearFB}), but this  regime can be slow to reach.

The main limitation of these grid-based methods is that, in order to go below the Rayleigh limit and perform super-resolution, the grid must be thin enough. This leads to theoretical and practical issues. Indeed, refining the grid not only increases the computational cost of each iteration, but it also deteriorates the conditioning of the linear operator to invert.  Hence, in practice, these methods provide solutions which are composed of small clusters of non-zero coefficents around  each ``true'' spike.  A way to mitigate this issue is to perform a post processing by replacing each cluster of spikes by its center of mass, as proposed in~\cite{tang-justdiscretize2013,flinthweiss2018}. This drastically reduces the number of false positive spikes although it is hard to analyze theoretically and can be unstable. Instead, one can also consider methods based on safe rules~\cite{elghaoui-safe2010} which perform a progressive pruning of the grid and keep only active sets of weights~\cite{salmon-screening2017}. Finally, it has been shown in~\cite{duval-thingridsI2017,duval-thingridsII2017}  that  the solution of the LASSO, in a small noise regime and when the step size tends to zero, contains pairs of spikes around the true ones.

\myparagraph{Fixed spectral discretization and  semidefinite programming (SDP) formulation}  In~\cite{candes-towards2013}, the authors propose a reformulation of  the Basis Pursuit for measures into an equivalent finite dimensional SDP for which  solvers exist. Similarly, one can get an SDP formulation of the BLASSO. However, these equivalences are only true in a $1$-dimensional setting. In higher dimensions ($d\geq2$), one needs to use the so-called Lasserre's hierarchy~\cite{lasserre-moments2009,lasserre-global2004}. This principle has been used for the super-resolution problem in~\cite{decastro-semi2015}.

  The resolution of SDPs can be tackled through proximal splitting methods~\cite{toh-fistasvd2010} as well as interior point methods~\cite{boyd2004convex}. However, the overall complexity of the latter is polynomial in $O(f_c^{2d})$, where $d$ is the dimension of the domain $\Pos$, which restricts its application to small dimensional problems.
This limitation has led to recent developments~\cite{catala-rank2017} where the authors proposed a relaxed low rank SDP formulation of the BLASSO in order to use a Frank-Wolfe-type method (see below). The resulting method  enjoys the better overall complexity of $O(f_c^d\log(f_c))$ per iteration.

   Finally, note that these SDP-based approaches are restricted to  certain type of forward operators (typically Fourier measurements). In contrast,   grid-based proximal methods as well as Frank-Wolfe  (directly on the BLASSO, see below) can be used for a larger class of operators $\Phi$.

\myparagraph{Optimization over the space of measures} In order to directly solve the BLASSO, one needs to design algorithms that do not use any Hilbertian structure and can instead deal with measures. The benefit is the fact that one can exploit advantageously the continuous setting of the problem (typically moving continuously spikes over the domain). In contrast to fixed spatial or spectral discretization methods,  these algorithms proceed by iteratively adding new spikes, \ie Dirac masses, to the recovered measure.   

The Frank-Wolfe (FW) algorithm~\cite{frank-fw1956} (see Section~\ref{sec:sfw}), also called the Conditional Gradient Method (CGM)~\cite{levitin-constrained1966}, solves  optimization problems of the form $\min_{m \in C}\ f(m)$, where $C$ is a weakly compact convex set of a topological vector space and $f$ is a differentiable convex function (in the case of the BLASSO, $m$ is a Radon measure). It proceeds by iteratively minimizing a linearized version of $f$. No Hilbertian structure is used which makes it well suited to work on the space of Radon measures. It has been proven under a curvature condition on $f$ (which holds on a Banach space for smooth functions having a Lipschitz gradient) that the rate of convergence of this algorithm in the objective function is $O(1/k)$ (see for instance~\cite{demyanov-1970approximate}).
However, it is possible to improve the convergence speed of FW  by replacing the current iterate by any \guillemet{better} candidate $m \in C$ that further decreases the objective function $f$. This simple idea has led to several successful variations of the standard FW algorithm. 
For instance, the authors of~\cite{bredies-inverse2013} proposed a modified Frank-Wolfe algorithm for the BLASSO where the final step updates the amplitudes and positions of spikes by a gradient descent on a non-convex optimization problem. Moving the spikes positions takes advantage of the continuous framework of the problem (the domain $\Pos$ is not discretized) which is the main ingredient that leads to a typical $N$-step convergence observed empirically.  Finally, this approach has later been used in~\cite{boyd-adcg2015} and provides state of the art results in many sparse inverse problems such as  matrix completion  or Single Molecule Localization Microscopy (SMLM)~\cite{SMLM,sage-quantitative2015}. %

\subsection{Other methods for super-resolution}

The Prony method~\cite{prony} and its successors such as MUSIC (MUltiple SIgnal Classification) \cite{schmidt-multiple1986}, ESPRIT (Estimation of Signal Parameters by Rotational Invariance Techniques) \cite{Kailath_1990},   or Matrix Pencil \cite{hua-pencils1990}, are spectral methods which perform spikes localization from low frequency measurements. They do not need any discretization and enable to recover exactly the initial signal in the noiseless case as long as there are enough observations compared to the number of distinct frequencies \cite{liao-music2014}. Extensions to deal with noise have been developped in \cite{cadzow-signal1988,condat-cadzow2015} and stability is known under a minimum separation distance \cite{liao-music2014}. Greedy algorithms constitute another class of popular methods for sparse super-resolution.
The Matching Pursuit (MP)~\cite{mallat-mp1994} adds new spikes by finding the ones that best correlate with the residual. The Orthogonal Matching Pursuit (OMP)~\cite{tropp-omp2008,soussen-omp2013,herzet-ols2014} is similar to MP but imposes that the current estimate of the observations, \ie
$\Phi(\sum_{i=1}^k \amp_i \dirac{\pos_i})$,
is always orthogonal to the residual. Hence, the amplitudes of the Dirac masses are updated  by an orthogonal projection after every support update (\ie addition of a new spike). 
It is noteworthy that there exist many generalizations/variants of OMP. For instance, the results of OMP can be improved with a backtracking step at each iteration, allowing to remove  non reliable spikes  from the support of the reconstructed measure~\cite{huang-backtracking2011}.

These greedy pursuit algorithms can be applied without grid discretization~\cite{jacques2008geometrical} which enables the use of local optimizations over the spikes' positions~\cite{eftekhari2015greed}.
Finally, let us mention the class of nonconvex optimization methods which include the well known 
Iterative Hard Thresholding (IHT)~\cite{davies-it2008,davies-iht2009}

\subsection{Contributions}

Our first set of contributions, detailed in Section~\ref{sec:etaW-laplace}, studies the BLASSO performance in the special case of several types of Laplace transforms. This theoretical study is motivated by the use of these Laplace transform for certain types of fluorescence microscopy imaging devices.  Our main finding is that for positive spikes, these operators can be stably inverted without minimum separation distance. This study makes use of the theoretical tools developed in our previous work~\cite{Denoyelle2017}. 

Our algorithmic contributions are detailed in Section~\ref{sec:sfw}, where we introduce the \ADCG, which is an extension of the initial FW solver proposed in~\cite{bredies-inverse2013}. 
Proposition~\ref{sec:sfw-prop:cvfaible} shows that this algorithm, used to solve the BLASSO, enjoys the same convergence property as the classical Frank-Wolfe algorithm (weak-* convergence with a rate in the objective function of $O(1/k)$).
Our main theoretical contribution is Theorem~\ref{sec:sfw-thm:cvksteps} which proves that our algorithm converges towards the unique solution of the BLASSO in a finite number of iterations.

Section~\ref{sec:microscopy} makes the connection between these two sets of contributions, by showcasing the \adcgshort~algorithm for 3-D PALM/STORM super-resolution fluorescence microscopy. We study its performance for several imaging operators, among which some relies on the inversion of a Laplace transform along the depth axis.  

The code to reproduce the numerical illustrations of this article can be found online at~\url{https://github.com/qdenoyelle}.

\subsection{Notations and Definitions}

This section gathers some useful notations and definitions.

\myparagraph{Ground space and measures}

We frame our theoretical and numerical analysis of the BLASSO on the space of Radon measure over a set $\Pos$.


\begin{defn}[Set $\Pos$ of positions of spikes]\label{def:Pos}
  The set of positions of spikes, denoted $\Pos$, is supposed to be a subset of $\RR^d$ with non-empty interior $\Poso$, or $\TT^d$ with $d\in\NN^*$. 
\end{defn}%

Definition~\ref{def:Pos} covers the particular case of $\Pos=\RR^d$, $\Pos=\TT^d$ or any compact subset with non-empty interior of $\RR^d$.

\begin{defn}[Continuous functions on $\Pos$]\label{def:ContX}
Let $(Y,\norm{\cdot}_Y)$ be a normed space. We denote by 
  $\Cder{}_c(X,Y)$ the space of $Y$-valued continuous functions with compact support, by $\ContX(\Pos,Y)$ the set of continuous functions that vanish at infinity \ie
\eq{%
	\forall \varepsilon>0, \exists K\subset\Pos \mbox{ compact}, \quad \underset{x\in\Pos\setminus K}{\sup} \ \norm{\phi(x)}_Y \leq \varepsilon,
}
and by $\Contk{k}(\Pos,Y)$ the set of $k$-times differentiable functions on $\Pos$.   Note that when $\Pos$ is compact,  $\Cder{}_c(X,Y)$ and $\ContX(\Pos,Y)$ are simply the set $\Cont(\Pos,Y)$ of continuous functions on $\Pos$.
\end{defn}

Now we can define rigorously the space of real bounded Radon measures on $\Pos$.

\begin{defn}[Set $\radon$ of Radon measures]\label{def:radon}
We denote by $\radon$ the set of real bounded Radon measures on $\Pos$ which is the topological dual of $\ContX(\Pos,\RR)$ endowed with $\normLi{\cdot}$ (the supremum norm for functions defined on $\Pos$).  
\end{defn}
By the Riesz representation theorem, $\radon$ is also the set of regular real Borel measures with finite total mass on $\Pos$.
See~\cite{rudin-real1987} for more details on Radon measures.

\myparagraph{Kernels}
This paragraph details the assumptions that we use in the following on the kernel $\phi$. We recall that the operator  $\Phi: \radon\rightarrow \Obs$, which models the acquisition process of the source signal, has the form:
\begin{align}\label{eq-defnPhi}
	\forall m\in\radon, \quad
	\Phi m &\eqdef \int_\Pos \varphi(\pos)\d m(\pos).
\end{align}
The above quantity is well-defined (as a Bochner integral) as soon as $\phi$ is continuous and bounded. In order to apply some results of~\cite{Denoyelle2017}, we add the hypotheses that are summarized below.  

\begin{defn}[Admissible kernels $\phi$]\label{sec:intro-def:admkernel}
We denote by $\kernel{k}$\index{kernel}, the set of admissible kernels of order $k$. A function $\phi:\Pos\to\Obs$ belongs to $\kernel{k}$ if:
\begin{itemize}
\item $\phi\in\Contk{k}(\Pos,\Obs)$,
\item For all $p\in\Obs$, $x\in\Pos\mapsto\dotObs{\phi(\pos)}{p}$ vanishes at infinity,
\item for all $0\leq i\leq k$, $\underset{\pos\in\Pos}{\sup} \normObs{D^i\phi(x)} <+\infty$.
\end{itemize}
where $D^i\phi$ is the $i$-th differential of $\phi$.
\end{defn}

\myparagraph{Operators}

Given $\pos=(\pos_1,\ldots,\pos_N)\in\Poso{}^N$, we denote by $\Phi_{\pos}:\RR^N\rightarrow \Obs$ the linear operator such that:
\begin{align}\label{sec:intro-def:Phix}
	\forall a\in \RR^N, \quad
	\Phi_{\pos}(a) \eqdef \sum_{i=1}^{N} a_i \phi(\pos_i),
\end{align}
and by $\Ga_{\pos}:(\RR^{N}\times\underbrace{\RR^N\times\cdots\times\RR^N}_{d})\rightarrow \Obs$ the linear operator defined by:
\begin{align}\label{sec:intro-def:Gax}
	\forall (a,b_1,\ldots,b_d)\in \RR^N\times(\RR^N)^d, \quad	
	\Ga_{\pos}\begin{pmatrix}a\\b_1\\\vdots\\b_d\end{pmatrix} \eqdef 
	\sum_{i=1}^{N}\left( a_i \phi(\pos_i) + \sum_{j=1}^d b_{j,i} \partial_j \varphi(\pos_i)\right).
\end{align}
We may also write $\Ga_{\pos}=\pa{\Phi_{\pos} \ \pa{\Phi_{\pos}}^{(1)}}$,
where $\pa{\Phi_{\pos}}^{(1)}$ (sometimes denoted  by $\Phi_{\pos}{}'$) stacks all the first order derivatives of $\phi$ for the different positions $\pos_i$. Similarly we define $\pa{\Phi_{\pos}}^{(k)}$ for $k\geq 1$ by stacking all the derivatives of order $k$. Finally, $\Ga_{\pos}^+$ refers to the pseudo-inverse of $\Ga_{\pos}$.

When $d=1$, given $\poscluster\in \Poso$, we denote by $\phiD{k} \in \Obs$ the $k^{th}$ derivative of $\phi$ at $\poscluster$, \textit{i.e.}
\begin{equation}\label{sec:intro-eq:phider}
\phiD{k} \eqdef \phi^{(k)}(\poscluster).
\end{equation}
In particular, $\phiD{0} = \phi(\poscluster)$. Given $k \in \NN$, we then define:
\begin{equation}\label{sec:intro-eq:Fk}
\Fk \eqdef \begin{pmatrix} \phiD{0} & \phiD{1} & \ldots & \phiD{k} \end{pmatrix}.
\end{equation}

\myparagraph{Injectivity Assumption} In order to avoid degeneracy issues we sometimes assume the following injectivity assumption of the operator when restricted to discrete spikes. 

\begin{defn}\label{sec:intro-def:injectivity}
  	Let $\varphi : \Pos\to\Obs$. For all $k\in\NN$, we say that the hypothesis $\injk$ holds at $\poscluster\in\Poso$ if and only if 
  	\begin{equation}
  			\phi\in\kernel{k}
		\text{ and }
		(\phiD{0}, \ldots,\phiD{k}) \text{ are linearly independent in }
		\Obs. 	\tag{$\injk$}
  	\end{equation}
\end{defn}

\myparagraph{Norms}

We use the $\ell^{\infty}$ norm, $\normLiVec{\cdot}$, for vectors of $\RR^N$ or $\RR^{2N}$, whereas the notation $\norm{\cdot}$ refers to an operator norm (on matrices, or bounded linear operators). $\normObs{\cdot}$ is the norm on $\Obs$ associated to the inner product $\dotObs{\cdot}{\cdot}$. $\normLi{\cdot}$ denotes the $\Linf$ norm for functions defined on $\Pos$.

%
%
\section{Reminders on the BLASSO}




\subsection{Recovery of the Support in Presence of Noise}

Let $\posO\in \Poso{}^N$, $\ampO\in (\RR\setminus\{0\})^N$ and $\measO=\sum_{i=1}^N\ampOi \dirac{\posOi}$. The BLASSO is the variational problem
\begin{align}\label{eq:defblasso}
	\umin{m \in \radon} \frac{1}{2} \normObs{\Phi m - \obsw}^2 + \lambda \normTVX{m} \tag{$\blasso$},
\end{align}
where $\obsw\eqdef\Phi\measO+w$ are the noisy observations of a measure composed of a sum of Dirac masses.
The optimality of a measure $m_\la$ for~$\blasso$ is characterized by the fact that the function
\begin{align}
  \label{eq:certifdual}
  \eta_\la\eqdef \Phi^* p_\la \qwhereq p_\la \eqdef \frac{1}{\la}(y-m_\la) 
\end{align}
satisfies $\normLi{\eta_\la}\leq 1$. The function $\eta_\la$ is then called a dual certificate.

When one is interested in the recovery of the support, \ie finding a solution $\meas$ of $\blasso$ composed of exactly the same number of Dirac masses as the initial measure $\measO$, in a small noise regime, an important object is the so-called vanishing derivatives precertificate introduced in~\cite{duval-exact2013}.

\begin{defn}[Vanishing Derivatives Precertificate, \cite{duval-exact2013}]\label{sec:blasso-def:etaV}
 	If $\GaxO$ has full column rank, there is a unique solution to the problem
  \begin{align*}
    \inf\enscond{\normObs{p}}{\forall i=1,\ldots,N, \; (\Phi^*p)(\posOi)=\sign(\ampOi), (\Phi^*p)'(\posOi)=0_{\RR^d}}.
  \end{align*}
  Its solution $\pVV$ is given by 
  \begin{align}\label{eq-vanish-closed-form}
   	\pVV=(\GaxO^{+})^* \begin{pmatrix} \sign(\ampO)  \\ 0_{(\RR^d)^N}\end{pmatrix},
  \end{align}
and we define the vanishing derivatives precertificate as $\etaVV\eqdef\Phi^*\pVV$ ($\GaxO$ is defined in Equation~\ref{sec:intro-def:Gax}).
\end{defn}

One can show that is if $\normLi{\etaVV}\leq 1$ then $\etaVV$ is a a so-called valid certificate, which assures that $\measO$ is a solution to the constrained problem (corresponding to setting $w=0$ and $\la\to0$ in $\blasso$)
\begin{align*}
	\umin{\Phi m=\obsO} \normTVX{m} \qwhereq \obsO\eqdef\Phi\measO \tag{$\bpursuit$}.
\end{align*}
More importantly, if it satisfies a stronger nondegeneracy condition detailed in Definition~\ref{sec:blasso-def:etaVnondegen} below, then $\eta_V$ also ensures the stable recovery of the support in a small noise regime when solving the BLASSO. This result proved in~\cite{duval-exact2013} is stated in Theorem~\ref{sec:blasso-thm:supportrecovery}.

\begin{defn}[Nondegeneracy of $\etaVV$, \cite{duval-exact2013}]\label{sec:blasso-def:etaVnondegen}
We say that $\etaVV$ is \emph{nondegenerate} if
\begin{equation}\label{eq-etaV-nondegen}
\left\{\begin{split}
	\foralls \pos\in \Pos \setminus \bigcup_{i=1}^N\{\posOi\}, \quad \abs{\etaVV(\pos)}&<1,\\
	 \foralls i\in\{1,\ldots,N\}, \quad
	\det(D^2\etaVV(\posOi))&\neq 0.
\end{split}\right.
\end{equation}
\end{defn}

\begin{thm}[Exact Support Recovery, \cite{duval-exact2013}]\label{sec:blasso-thm:supportrecovery}
Assume that $\phi\in\kernel{2}$, $\GaxO$ has full column rank and $\etaVV$ is nondegenerate. Then there exists $C>0$ such that if $(\la,w)\in\RR_+^*\times\Obs$ satisfies:
\eq{\max\pa{\la,\normObs{w}/\la} \leq C,}
then there is a unique solution $\meas$ to $\blasso$ composed of $N$ Dirac masses such that $(\amp,\pos)=\fimpp(\la,w)$ where $\fimpp$ is $\Ccr^{1}$. In particular, by taking the regularization parameter $\lambda = \normObs{w}/C$ proportional to the noise level, one obtains
\eq{
	\normLiVec{(\amp,\pos)-(\ampO,\posO)}=O(\normObs{w}),
}
where $\normLiVec{\cdot}$ is the $\ell^{\infty}$ norm for vectors. 
\end{thm}

Figure~\ref{fig:EtaV_microscopy} displays some example of $\etaVV$ associated to several $\Phi$ operators for 3-D super-resolution fluorescence microscopy. This shows that for these inverse problems, the BLASSO stably recovers the support of the input measure if the noise level is not too high.

\subsection{The Super-Resolution Problem}

In this section, $\Pos$ is considered to be $1$-dimensional and we now tackle the super-resolution problem in presence of noise using the BLASSO. In this setting, we assume that the Dirac masses of the initial measure have positive amplitudes and cluster at some point $\poscluster\in \Poso$. We parametrize this cluster as
\eq{
	\meastO\eqdef\sum_{i=1}^N \ampOi \dirac{\poscluster+\postOi} \qwhereq \ampOi>0, \ \poszOi\in\RR,
}
and where the parameter $t>0$ controls the separation distance between the spikes of the input measure.

In~\cite{Denoyelle2017}, the authors proved that the recovery of the support in presence of noise in the limit $t\to0$ is controlled by the $2N-1$ vanishing derivatives precertificate.

\begin{prop}[$2N-1$ Vanishing Derivatives Precertificate,~\cite{Denoyelle2017}]\label{sec:blasso-def:etaW}
	If $\injdn$ holds at $\poscluster$ (see Definiton~\ref{sec:intro-def:injectivity}), there is a unique solution to the problem
  \begin{align*}
    \inf\enscond{\normObs{p}}{(\Phi^*p)(\poscluster)=1, (\Phi^*p)'(\poscluster)=0, \ldots , (\Phi^*p)^{(2N-1)}(\poscluster)=0}.
  \end{align*}
  We denote by $\pW$ its solution, given by 
  \begin{equation}\label{def-pW}
    \pW=(\Fdn^{+})^*\dirac{2N}   
	\qwhereq
  \dirac{2N} \eqdef (1,0,\ldots,0)^T \in \RR^{2N}, 
\end{equation}
and we define the $2N-1$ vanishing derivatives precertificate as $\etaW\eqdef\Phi^*\pW$ (see Equation~\ref{sec:intro-eq:Fk} for the definition of $\Fdn$).
\end{prop}

Figure~\ref{sec:blasso-fig:etaWgaussian} shows $\etaW$ in the case of a Gaussian convolution kernel.

\begin{figure}[!htb]
\centering
\subfigure[$N=1$]{\includegraphics[width=0.23\linewidth]{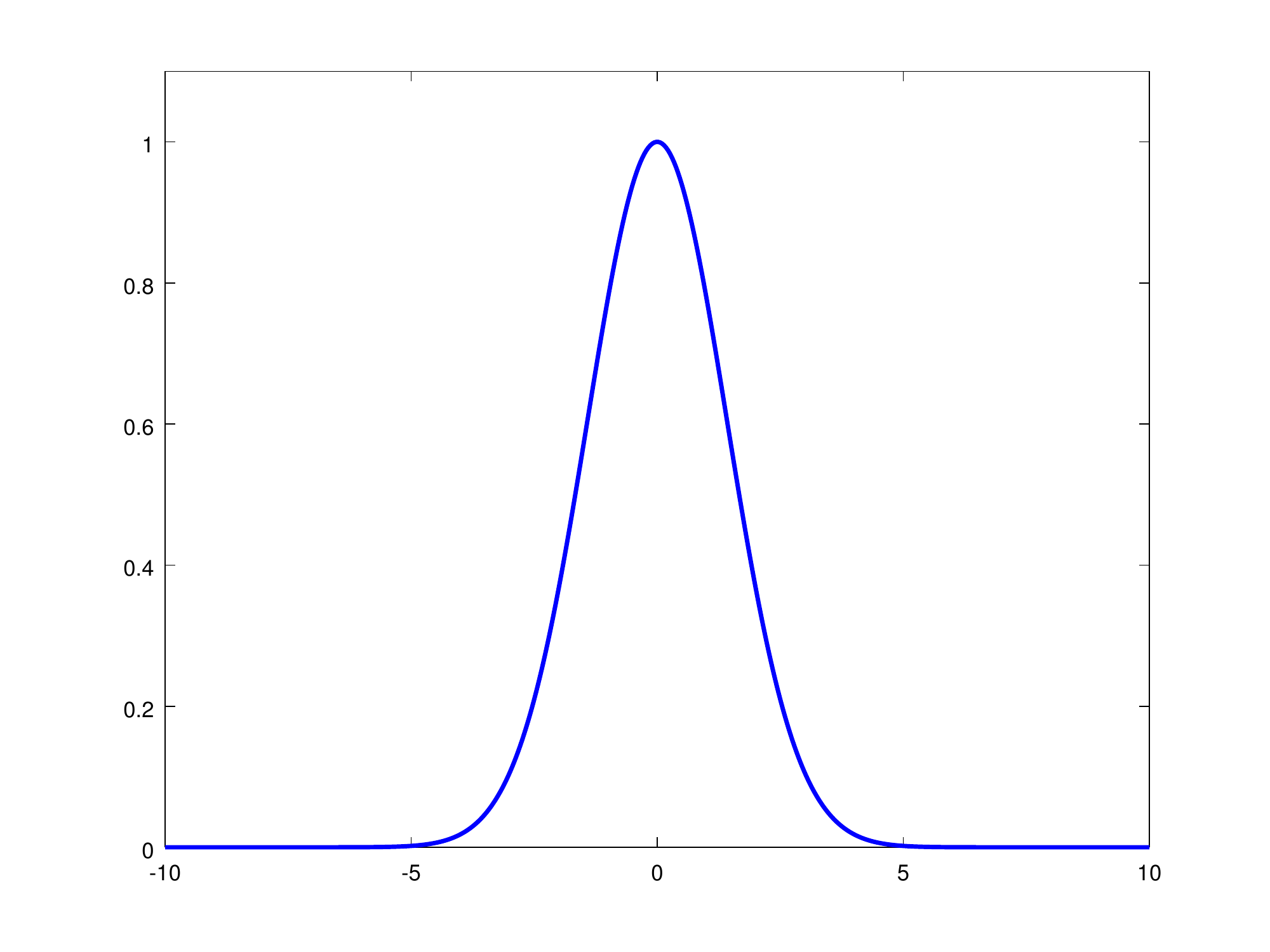}}
\subfigure[$N=2$]{\includegraphics[width=0.23\linewidth]{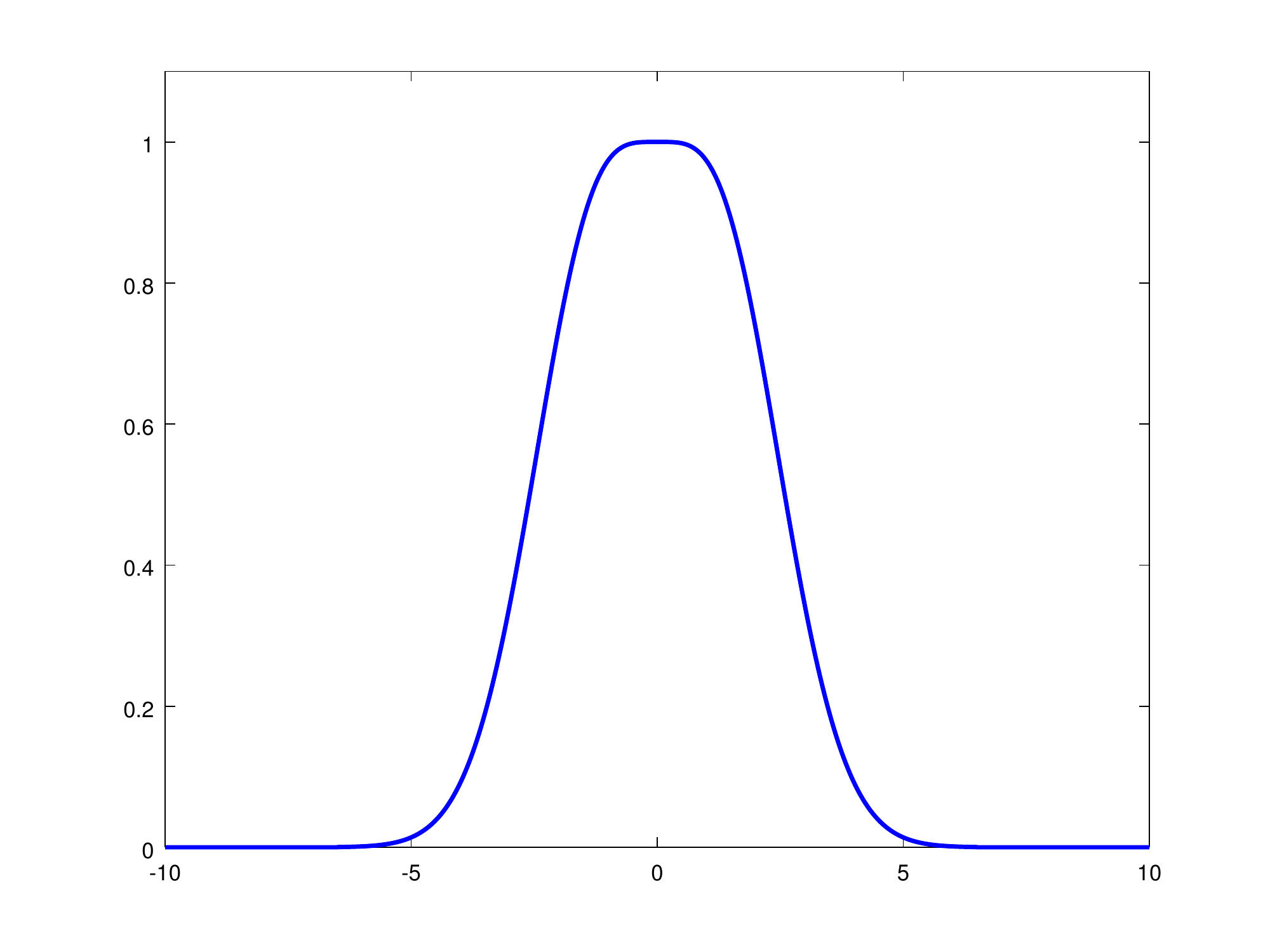}}
\subfigure[$N=4$]{\includegraphics[width=0.23\linewidth]{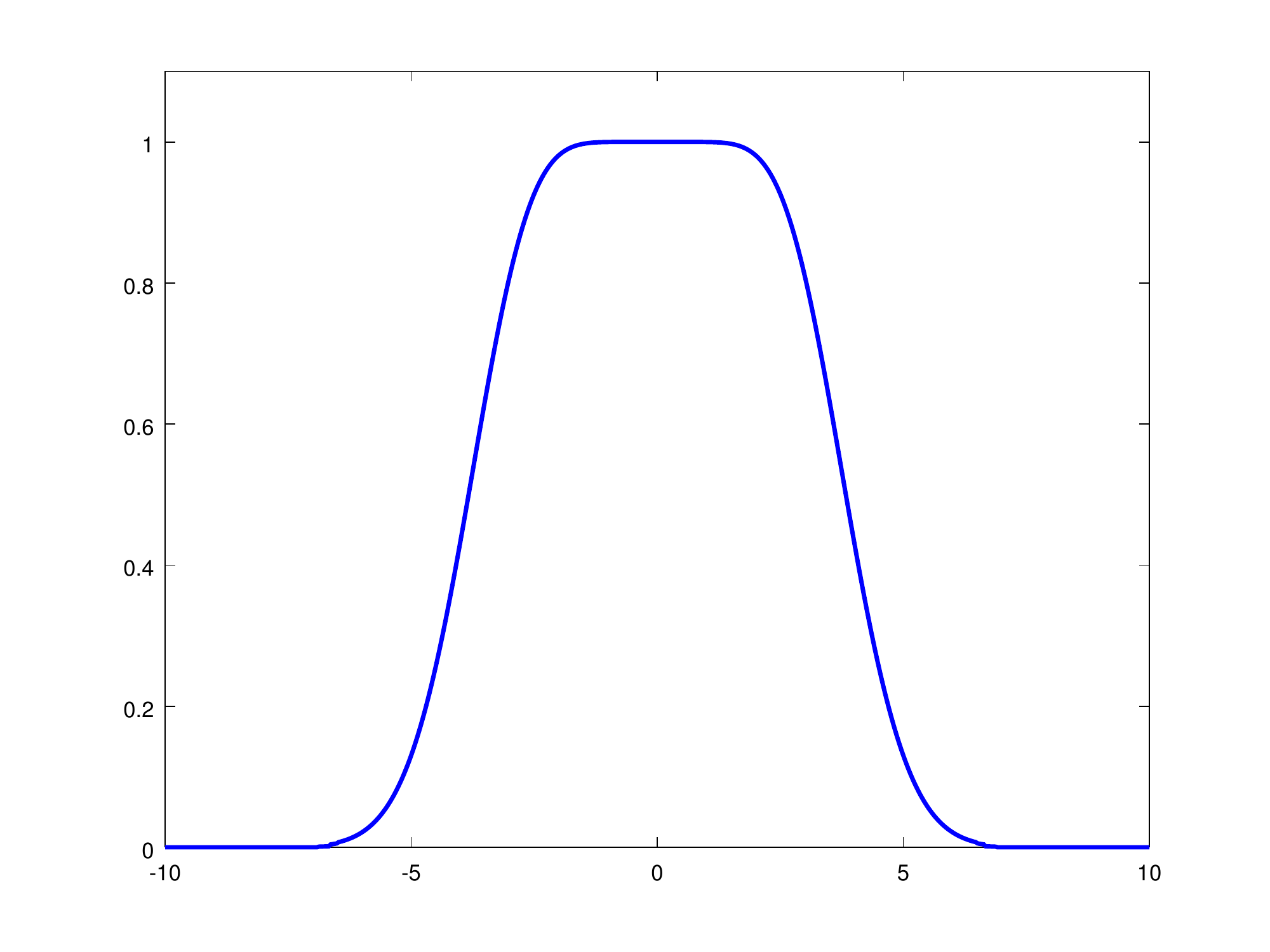}}
\subfigure[$N=7$]{\includegraphics[width=0.23\linewidth]{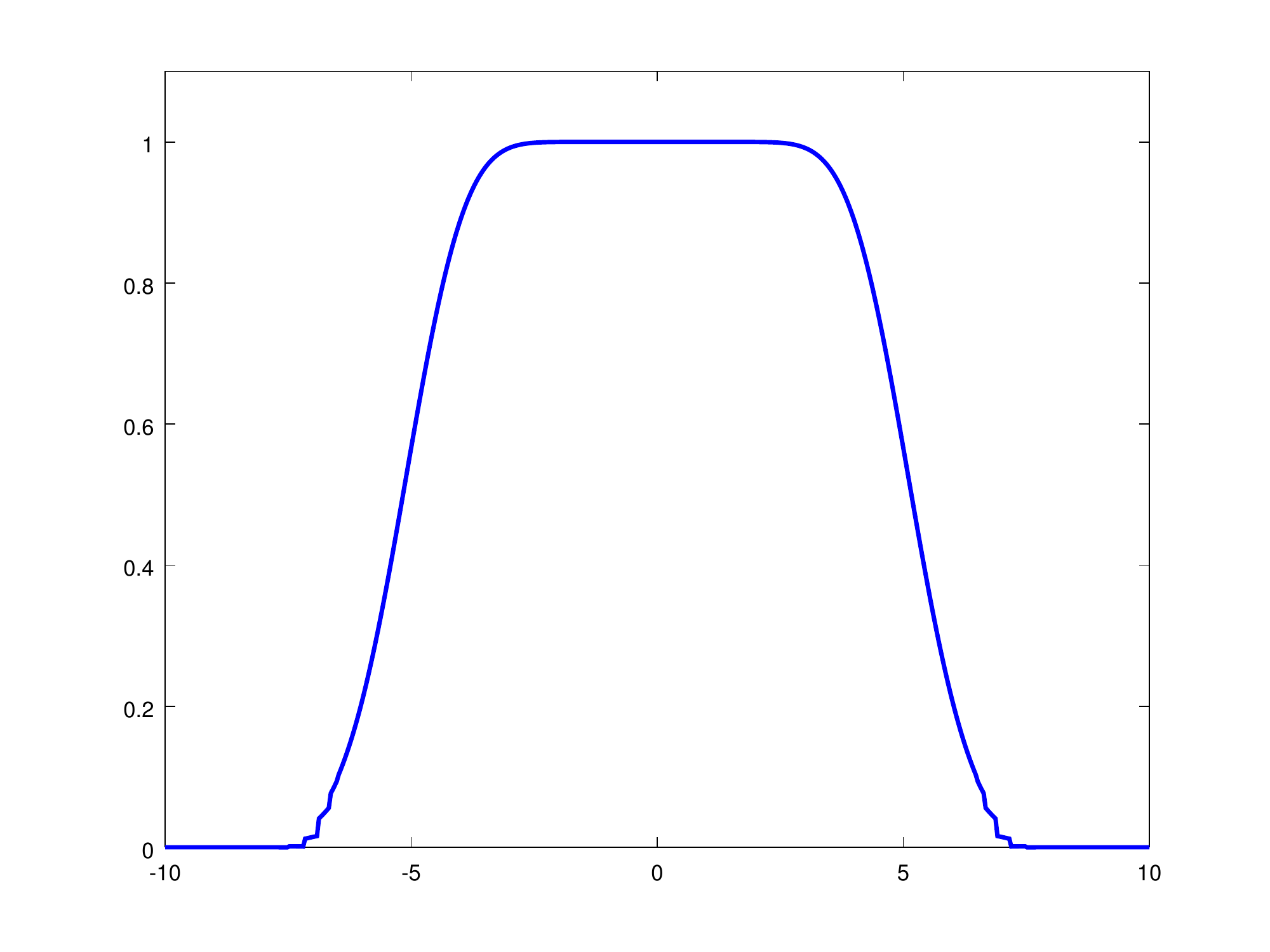}}\caption{\label{sec:blasso-fig:etaWgaussian}$\etaW$ for a Gaussian convolution ($x\in\RR$, $\phi(x)=e^{-\frac{(\cdot-x)^2}{2\sigma^2}}$) for several numbers of spikes and $\sigma=1$.}
\end{figure}

\begin{rem}
From Proposition~\ref{sec:blasso-def:etaW}, one can easily see that $\etaW$ can equivalently be written as
\begin{align}\label{sec:blasso-eq:etaWdefbis}
	\forall\pos\in\Pos,\quad\etaW(\pos)=\sum_{k=0}^{2N-1} \al_k\partial^{(k)}_2\Co(\pos,\poscluster),
\end{align}
where $\Co$ is the correlation kernel associated to the correlation operator $\Phi^*\Phi$, namely $\Co(\pos,\pos')=\dotObs{\phi(\pos)}{\phi(\pos')}$, and the coefficients $\al_k$ are defined by the equations 
\begin{align}\label{sec:blasso-eq:etaWequations}
 \forall k\in\{0,\ldots,2N-1\},\quad\etaW^{(k)}(\poscluster)=\delta_0^k.
\end{align}
\end{rem}

If $\etaW$ satisfies some nondegeneracy property (see Definition~\ref{sec:blasso-def:etaWnondegen}) then one can prove that the recovery of the support in a small noise regime when $t\to0$ is possible. Theorem~\ref{sec:blasso-thm:superresol} (see~\cite{Denoyelle2017}) makes this  statement precise by quantifying the scaling between the noise level and the separation $t$ to ensure the recovery.

\begin{defn}\label{sec:blasso-def:etaWnondegen}
 	Assume that $\injdn$ holds at $\poscluster$ and $\phi\in\kernel{2N}$. We say that $\etaW$ is $(2N-1)$-nondegenerate if $\etaW^{(2N)}(\poscluster)\neq 0$ and for all $\pos\in \Pos\setminus\{\poscluster\}$, $|\etaW(\pos)| <1$.
\end{defn}

\begin{thm}\label{sec:blasso-thm:superresol}
	Suppose that $\varphi\in\kernel{2N+1}$ and that $\etaW$ is $(2N-1)$-nondegenerate. 
	Then there exist positive constants $t_0,C,M$ (depending only on $\phi$, $\ampO$ and $\poszO$) such that 
	for all $0<t<t_0$, 
	for all $(\la,w)\in \ball{0}{C t^{2N-1}}$ with $\normObs{w}/\la\leq C$, 
\begin{itemize}
  \item the BLASSO has a unique solution,
  \item that solution has exactly $N$ spikes, and it is of the form $m_{a,\poscluster+tz}$, with $(\amp,\posz)=g(\la,w)$ (where $g$ is a $\Cder{2N}$ function),
  \item the following inequality holds
    \eq{
		\normLiVec{(\amp,\posz)-(\ampO,\poszO)}\leq\constdg\pa{\frac{|\la|}{t^{2N-1}}+\frac{\normObs{w}}{t^{2N-1}}}.
	}
\end{itemize}
\end{thm}

In the next section, we prove that the main assumption of Theorem~\ref{sec:blasso-thm:superresol} (the nondegeneracy of $\etaW$) is satisfied for some operators $\Phi$ associated to Laplace measurements.

%
%
\section{BLASSO for Laplace Inversion}\label{sec:etaW-laplace}

Most existing theoretical studies of super-resolution are focussed on translation-invariant operator $\Phi$ (convolution or Fourier measurements), see Section~\ref{sec-intro-blasso-pw}. In contrast, this section presents new results for one of the most fundamental non-translation invariant operator: the Laplace transform (and variants). 

The behavior of the Laplace transform is radically different from the one of the Fourier transform, and understanding the impact of the lack of translation invariance on super-resolution is relevant for many applications in imaging, including those considered in Section~\ref{sec:microscopy}.
A first argument in favor of the BLASSO for the Laplace transform is the study provided in~\cite{duval-ndsc2017}. It essentially shows that the recovery of $N$ positive spikes with stability of the support is possible using at least $2N$ measurements, regardless of the spacing of the spikes (and the spacings of the samples).
The stability is asserted by showing that $\etaV$ and $\etaW$ are nondegenerate, using abstract T-systems arguments.

Our strategy here is different, as we provide closed form expressions for $\etaW$ for these operators in order to show its nondegeneracy. The results presented here are thus complementary to those of~\cite{duval-ndsc2017}, providing additional theoretical guarantees which backup our numerical observations. The main differences are 
\begin{itemize}
  \item we provide closed form expressions to $\etaW$,
  \item some of the impulse responses we consider are $L^2$-normalized, a case which is not covered by the theory of~\cite{duval-ndsc2017},
  \item we cannot deal with arbitrary samplings $\mu$, contrary to~\cite{duval-ndsc2017}.
\end{itemize}

In this section, we suppose that $N$ spikes are clustered at the position $\poscluster\in\Poso$ (which appears in the following results because of the non translation invariance of the kernel). 

In the next section, we first detail the different continuous operators considered. Then, Section~\ref{subsec:etaW-unnorm} gives explicit formulas for $\etaW$ in two different setups and shows that $\eta_W$ is $(2N-1)$-nondegenerate. Finally, Section~\ref{subsec-laplaceinversion} provides some numerical material concerning $\etaW$ when the continuous kernels are approximated by a sampling.

\subsection{Laplace Operators}\label{subsec:laplace-models}

We suppose in  this section that $\Pos=[\xmin,\xmax]\subset \RR_+^*$ is a compact interval, and that $\Obs = \Ldeux(\RR_+,\mu)$ for some Radon measure $\mu$ on $\RR_+$. A generic Laplace measurement kernel is defined as
\begin{align}\label{sec:basso-eq:phi-laplace}
	 \forall \pos\in\Pos,\quad\phi(\pos) \eqdef \pa{s \mapsto \xi(\pos) e^{-s\pos}} \in \Obs.
 \end{align}
 This choice ensures that $\phi$ defines a valid operator $\Phi$ for all the  the Laplace-like transform models presented below, provided $e^{-\xmin s}\d\mu(s)$ has sufficiently many finite moments (in the following we require finite moments of order $4N-1$).
The kernel is parametrized by a positive Radon measure  $\mu$ on $\Pos$ (which models the sampling pattern) and a non-negative weighting function  $\xi \in \Cont(\Pos,\RR)$ (which takes care of the normalization of the measurement). 
The adjoint operator is thus defined as
\eq{
	(\Phi^*p)(\pos) = \xi(\pos) \int_{\RR_+} e^{-s\pos} p(s) \d \mu(s).
}

The choice of $\mu$ is let  to the experimentalist and corresponds to the way samples are chosen.
A discrete measure $\mu = \sum_{k=1}^K \mu_k \de_{s_k}$ corresponds to using a finite set of samples values $s_k$. In this case, one can equivalently consider finite-dimensional observations $\Obs=\RR^K$ and define $\phi(\pos) \eqdef ( \xi(\pos)\mu_k e^{-s_k \pos} )_{k=1}^K \in \Obs$.
A continuous measure $\d\mu(s)=h_\mu(s) \d s$ is a mathematical idealization, where a high value of $h_\mu(s)$ indicates that a high number of measurements have been taken for the  index $s$ (or equivalently that there is less noise for this measurement). On contrast, a value $\mu(s)=0$ indicates that this measurement is not available. 

In contrast, $\xi$ can be freely chosen but strongly impacts the BLASSO problem by weighting the contribution of each position. The design of such a spatially-varying weighting is crucial (and non trivial) here because the operator $\Phi$ is not translation-invariant.    
The most frequent normalization for LASSO-type problems is
\begin{equation}\label{eq-normalization}
	\xi(\pos)^2 = \frac{1}{ \int_{\RR_+} e^{-2 s\pos} \d\mu(s)  },
\end{equation}
which guarantees that $\normObs{\phi(\pos)}=1$ for all $\pos\in\Pos$. See Section~\ref{par:L2-norm-laplace} for more details for this normalization.

Note that both $\mu$ and $\xi$ can be independently chosen, since they operate separately on the input and output variables $\pos$ and $s$. 

\myparagraph{Correlation kernel} The properties of the BLASSO problem (and also the implementation of BLASSO solvers) only depend on the correlation operator $\Phi^*\Phi$ (rather than on the operator $\Phi$ itself). This operator reads $(\Phi^*\Phi m)(x) = \int_X \Co(x,x') \d m(x')$ where $\Co$ is a symmetric positive kernel. 
For Laplace-type operators, it reads
\eq{
	\forall\pos,\pos'\in\Pos,\quad\Co(\pos,\pos')  = \xi(\pos)\xi(\pos') \int_{\RR_+} e^{-(\pos+\pos')s} \d \mu(s).
}
The choice of normalization~\eqref{eq-normalization}  ensures that $\Co(\pos,\pos)=1$. 

We now detail in the following sections several particular cases covered by Equation~\ref{sec:basso-eq:phi-laplace} and study the associated $\etaW$.

\subsection{Preliminaries Results}

This section gathers preliminary results useful for the computation of $\etaW$.

One begins with two elementary lemmas. Their proofs are left to the reader.
The first one is a simple consequence of the Faa di Bruno formula.
\begin{lem}
\label{lem:faadibruno}
  Let $I$, $I'\subset \RR$ be open intervals, and $h:I'\rightarrow I$ be a smooth diffeomorphism.
  Let $\poscluster\in I$, $\tcluster:=h^{-1}(\poscluster)\in I'$, and let $\eta: I\rightarrow \RR$ be a smooth function. 
  Then $\eta$ satisfies
  \begin{equation}
    \eta(\poscluster)=1, \ \eta'(\poscluster)=0, \ldots, \eta^{(2N-1)}(\poscluster)=0,
  \end{equation}
  if and only if $\nu\eqdef \eta\circ h$ satisfies 
  \begin{equation}
    \nu(\tcluster)=1, \ \nu'(\tcluster)=0, \ldots, \nu^{(2N-1)}(\tcluster)=0.\label{eq:laplacefaadibruno}
  \end{equation}
  Moreover, in that case, $\nu^{(2N)}(\tcluster)=\eta^{(2N)}(\poscluster)(h'(\tcluster))^{2N}$.
\end{lem}

The next one follows from the general Leibniz rule.
\begin{lem}
\label{lem:leibniz}
  Let $I$ be an open interval, $\tcluster\in I$ and let $g:I\rightarrow \RR$, $\eta: I\rightarrow \RR$ be two smooth functions.
 If $\eta$ satisfies:
  \begin{equation}
    \eta(\poscluster)=1, \ \eta'(\poscluster)=0, \ldots, \eta^{(2N-1)}(\poscluster)=0,
  \end{equation}
then $P\eqdef \eta \times g$ satisfies:
  \begin{equation}
    P(\poscluster)=g(\poscluster), \ P'(\poscluster)=g'(\poscluster), \ldots,\ P^{(2N-1)}(\poscluster)=g^{(2N-1)}(\poscluster).
  \end{equation}
  In particular, if $P\in \RR_{2N-1}[T]$, then $P$ is the Taylor expansion of $g$ at $\poscluster$ of order $2N-1$, and $\etaW^{(2N)}(\poscluster)=-g^{(2N)}(\poscluster)/g(\poscluster)$ provided that $g(\poscluster)\neq 0$.
\end{lem}

\subsection{Explicit Formulas for $\etaW$ in Continuous Settings}\label{subsec:etaW-unnorm}

\subsubsection{Classical Laplace Operator}\label{par:laplace}

We suppose that $\mu=\Ll$, where $\Ll$ is the Lebesgue measure on $\RR_+$, and $\xi=1$.
%
Then one has
\begin{align}\label{sec:laplace-eq:correllapl1}
	\Co(\pos,\pos')=\frac{1}{\pos+\pos'}.
\end{align}

The following Proposition provides a formula for $\etaW$ in this unnormalized continuous setting and proves that it is nondegenerate.


\begin{prop}\label{prop:etaW-unnorm-Laplace}
  $\etaW$ is $(2N-1)$-nondegenerate. More precisely, we have 
  \begin{align}\label{etaW-expr}
\forall \pos\in\Pos, \quad \etaW(\pos)=1-\pa{\frac{\pos-\poscluster}{\pos+\poscluster}}^{2N}.
\end{align}
\end{prop}

\begin{figure}[!h]
\centering
\subfigure[$N=2$]{\includegraphics[width=0.48\linewidth]{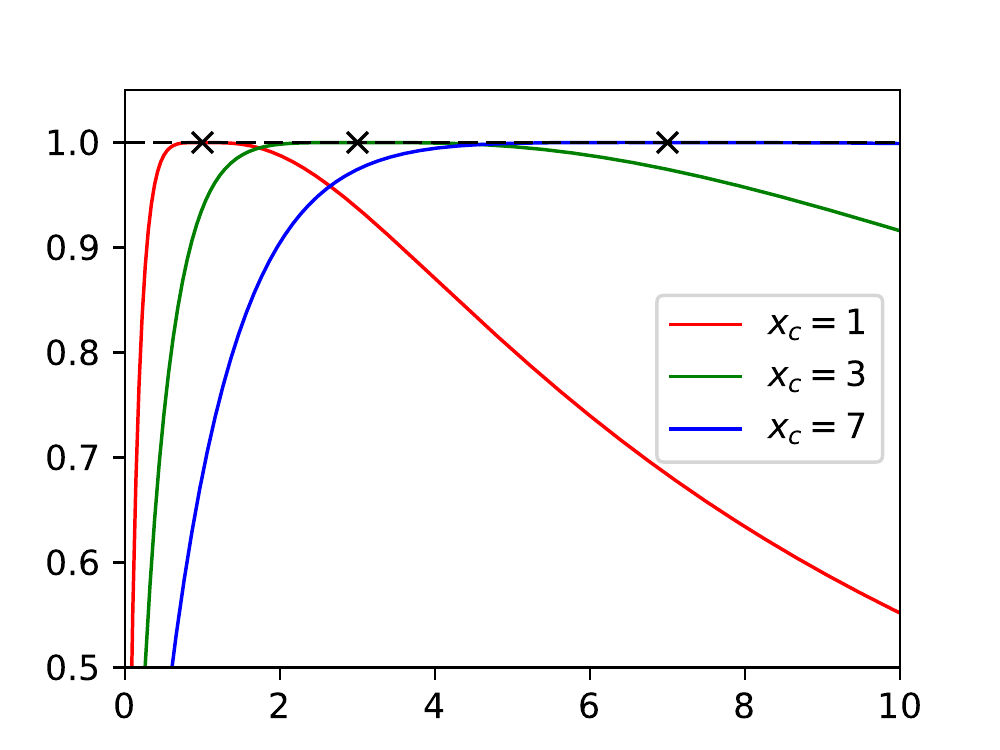}}
\subfigure[$\poscluster=1$]{\includegraphics[width=0.48\linewidth]{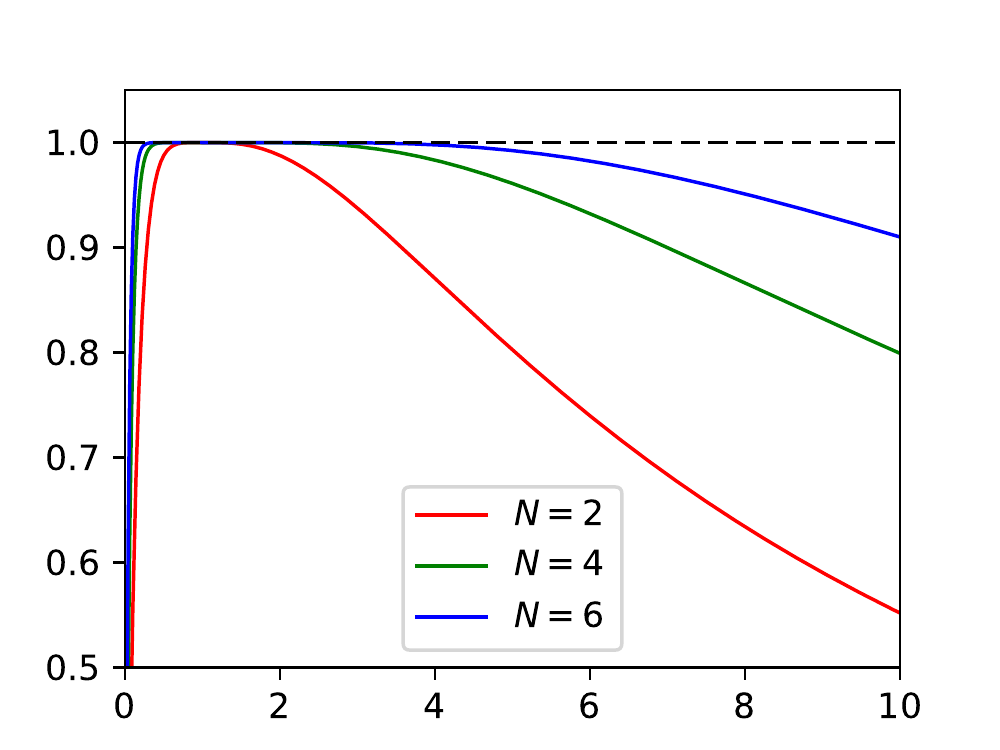}}
\caption{\label{sec:blasso-fig:etaW-laplace-unnorm}$\etaW$ for the unnormalized Laplace model for a varying $\poscluster$ with fixed $N=2$ and a fixed $\poscluster=1$ with varying $N\in\{2,4,6\}$.}
\end{figure}

In Figure~\ref{sec:blasso-fig:etaW-laplace-unnorm}, one sees that when the position $\poscluster$ where the spikes cluster increases, the curvature of $\etaW$ at $\poscluster$ decreases. This means that it is harder in this situation to perform the recovery. It reflects the exponential decay of the kernel $\phi$. 

\begin{proof}[Proof of Proposition~\ref{prop:etaW-unnorm-Laplace}]
  From Equations~\eqref{sec:blasso-eq:etaWdefbis} and~\eqref{sec:laplace-eq:correllapl1}, one sees that $\etaW$ has the form 
\begin{align*}
  \etaW(\pos)=\sum_{k=1}^{2N}\frac{\beta_k}{(x+\poscluster)^k}, \qwhereq \beta_k\in \RR.
\end{align*}
We set $h:t\mapsto (1/t-\poscluster)$, $\nu\eqdef \eta\circ h$ so that \begin{align*}
  \nu(t) =\sum_{k=1}^{2N}\beta_k t^k,
\end{align*}
is a polynomial with degree at most $2N$ with $\nu(0)=0$. By Lemma~\ref{lem:faadibruno}, $\nu$ satisfies~\eqref{eq:laplacefaadibruno} at $\tcluster\eqdef \frac{1}{2\poscluster}$.
As a result, $\nu(t)=1+\beta_{2N}(t-\tcluster)^{2N}$. The constant $\beta_{2N}$ is fixed by the condition $\nu(0)=0$, so that $\nu(t)=1-\left(\frac{t-\tcluster}{\tcluster}\right)^{2N}$, and $\etaW$ is given by~\eqref{etaW-expr}.

The $2N$ derivative is $\nu^{(2N)}(\tcluster)=-\frac{(2N)!}{(\tcluster)^{2N}}$, so that $\etaW(\poscluster)=-\frac{(2N)!}{(2\poscluster)^{2N}}<0$.
\end{proof}

\subsubsection{$L^2$-Normalized Laplace Operator}\label{par:L2-norm-laplace}

We choose $\mu=\Ll$, where $\Ll$ is the Lebesgue measure on $\RR_+$ , and
\eq{
	\forall \pos\in\Pos, \quad \xi(\pos) = \sqrt{\frac{1}{ \int_{\RR_+} e^{-2 s\pos} \d s  }}=\sqrt{2x},
}
so that for all $\pos\in\Pos$, $\varphi(\pos): s\mapsto \sqrt{2\pos}e^{-s\pos}$ and $\normObs{\phi(\pos)}=1$. One gets
\begin{equation}\label{eq-normalized-lapl-correl}
  \forall \pos,\pos'\in\Pos,\quad\Co(\pos,\pos')\eqdef\dotObs{\phi(\pos)}{\phi(\pos')}=\frac{2\sqrt{\pos\pos'}}{\pos+\pos'}.
\end{equation}

The following Proposition provides a formula for $\etaW$ in this normalized setting and proves that it is nondegenerate.


\begin{prop}\label{prop:etaW-normL2-Laplace}
  $\etaW$ is $(2N-1)$-nondegenerate. More precisely, we have the following formula:
  \begin{equation}
   \forall \pos\in\Pos,\quad \etaW(\pos)=\frac{2\sqrt{\pos\poscluster}}{\pos+\poscluster}\sum_{k=0}^{N-1}\frac{(2k)!}{2^{2k}(k!)^2}\left(\frac{\pos-\poscluster}{\pos+\poscluster}\right)^{2k}.
  \end{equation}
\end{prop}

\begin{figure}[!h]
\centering
\subfigure[$N=2$]{\includegraphics[width=0.48\linewidth]{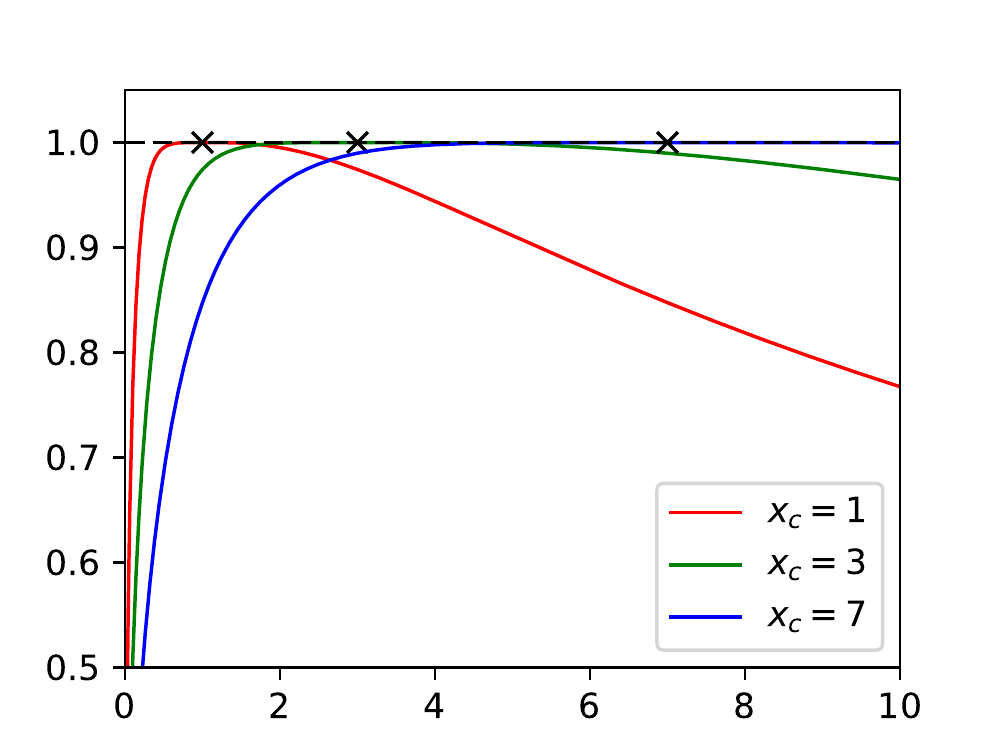}}
\subfigure[$\poscluster=1$]{\includegraphics[width=0.48\linewidth]{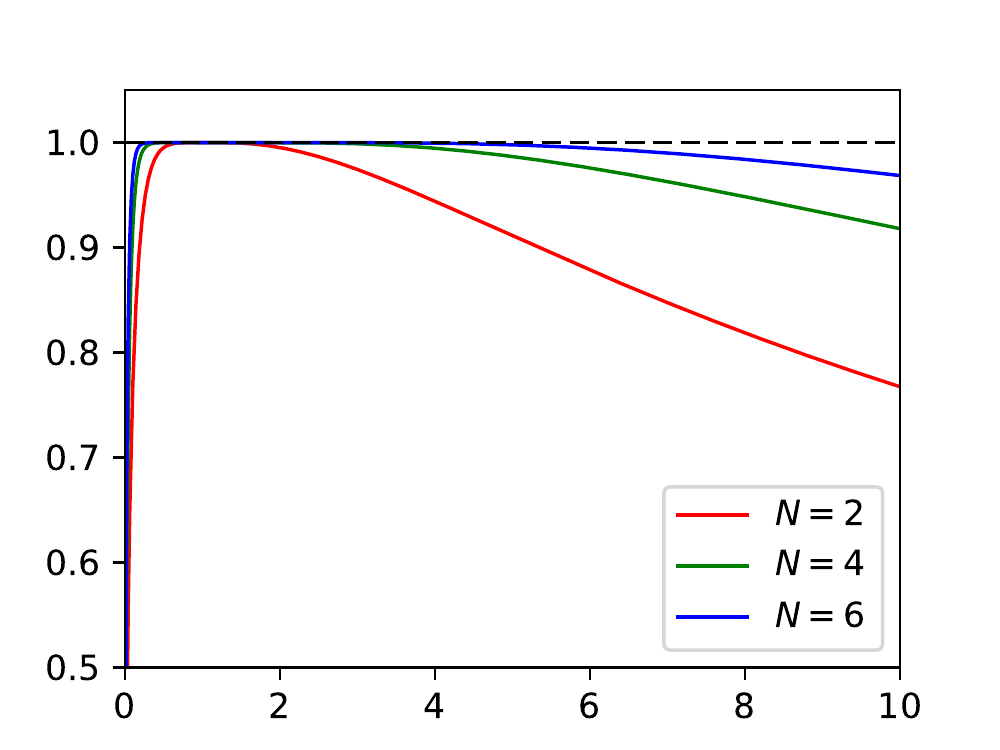}}
\caption{\label{sec:blasso-fig:etaW-laplace-norm}$\etaW$ for the normalized Laplace model for a varying $\poscluster$ with fixed $N=2$ and a fixed $\poscluster=1$ with varying $N\in\{2,4,6\}$.}
\end{figure}

In Figure~\ref{sec:blasso-fig:etaW-laplace-norm}, one sees that when the position $\poscluster$ where the spikes cluster increases then the curvature of $\etaW$ at $\poscluster$ decreases. The interpretation is the same as in the previous paragraph.

\begin{proof}[Proof of Proposition~\ref{prop:etaW-normL2-Laplace}]
From the general Leibniz rule, we have for all $n\in\{0,\ldots,2N-1\}$ and for all $\pos,\pos'\in\Pos$:
\begin{equation*}
	\frac{\d^{n}}{\d {\pos'}^{n}}\pa{\Co(\pos,\pos')}=2\sqrt{\pos}\sum_{k=0}^{n}\binom{n}{k}\frac{\d^{n-k}}{\d {\pos'}^{n-k}}\left(\sqrt{\pos'}\right)\frac{\d^k}{\d {\pos'}^k}\left(\frac{1}{\pos+\pos'}\right)
\end{equation*}
Evaluating this expression at $\pos'=\poscluster$, one gets that:
\eq{
	\partial_2^{n}\Co(\pos,\poscluster)=\sqrt{\pos}\sum_{k=0}^{n} \frac{\alpha_k}{(\pos+\poscluster)^{k+1}},
}
for some coefficients $\alpha_k\in \RR$. As a result, $\etaW$ is the unique function the form
\begin{equation*}
  \etaW(\pos)=\sqrt{\pos}\sum_{k=0}^{2N-1} \frac{\be_k}{(\pos+\poscluster)^{k+1}}
\end{equation*}
for some coefficients $\beta_k\in \RR$, which satisfies~\eqref{sec:blasso-eq:etaWequations}.
As before, we set $t=\frac{1}{\pos+\poscluster}$, that is $\pos=h(t)\eqdef \frac{1}{t}-\poscluster$, and $h$ is a diffeomorphism of $(0,1/\poscluster)$ onto $(0,+\infty)$.
Then:
\eq{ \etaW\circ h(t)=\sqrt{\frac{1}{t}-\poscluster}tP(t)=\sqrt{t-t^2\poscluster}P(t),}
where $P(T)=\sum_{k=0}^{2N-1}\beta_k T^k\in \RR_{2N-1}[T]$.

By Lemma~\ref{lem:faadibruno} and Lemma~\ref{lem:leibniz}, $P$ is the Taylor expansion of order $2N-1$ of $g:t\mapsto \frac{1}{\sqrt{t-t^2\poscluster}}$ at $\tcluster=h^{-1}(\poscluster)=\frac{1}{2\poscluster}$. Setting $t=u+\frac{1}{2\poscluster}$, we note that:
\begin{align*}
  \frac{1}{\sqrt{t-t^2\poscluster}}&=\frac{2\sqrt{\poscluster}}{\sqrt{1-(2u\poscluster)^2}}
  \qandq \frac{1}{\sqrt{1-z^2}} =\sum_{k=0}^{N-1}\frac{(2k)!}{2^{2k}(k!)^2}z^{2k} + o(z^{2N-1}).
\end{align*}
One deduces that 
\begin{align*}
\frac{1}{\sqrt{t-t^2\poscluster}}&= 2\sqrt{\poscluster}\sum_{k=0}^{N-1}\frac{(2k)!}{2^{2k}(k!)^2}\left[2\poscluster(t-\tcluster)\right]^{2k}+o(\pa{t-\tcluster}^{2N-1}).
\end{align*}

As a result, $P$ is given by $P(t)=2\sqrt{\poscluster}\sum_{k=0}^{N-1}\frac{(2k)!}{2^{2k}(k!)^2}\left[2\poscluster(t-\tcluster)\right]^{2k}$
and 
\begin{align}
  \etaW\circ h(t)&=\sqrt{t-t^2\poscluster}P(t)\\
                 &=1-\frac{\sum_{k=M}^{+\infty}\frac{(2k)!}{2^{2k}(k!)^2}\left[2\poscluster(t-\tcluster)\right]^{2k}}{\sum_{k=0}^{+\infty}\frac{(2k)!}{2^{2k}(k!)^2}\left[2\poscluster(t-\tcluster)\right]^{2k}}.
\end{align}
One sees that $\abs{\etaW\circ h(t)}<1$ for all $t\in (0,\frac{1}{\poscluster})\setminus\{\frac{1}{2\poscluster}\}$, and by Lemma~\ref{lem:leibniz},
  \begin{equation}
    (\etaW\circ h)^{(2N)}(\tcluster)=-g^{(2N)}(\tcluster)/g(\tcluster)=-\frac{((2N)!)^2}{(N!)^2}\poscluster^{2N}<0
  \end{equation}
 so that $\etaW\circ h$ (hence $\etaW$) is $(2N-1)$-nondegenerate. One recovers $\etaW$ by composing with $h^{-1}$, noting that $2\poscluster(t-\tcluster)=\frac{\poscluster-\pos}{\pos+\poscluster}$.
\end{proof}

\subsection{Sampled Approximations}\label{subsec-laplaceinversion}

The previous two cases (normalized and unnormalized versions of the Laplace transform) correspond to mathematical idealizations. In practice, one needs to restrict the sampling patterns by limiting their ranges and considering discrete samples. The following two setups are involved in the application of Section~\ref{sec:microscopy}.





\myparagraph{Discretized Unnormalized Laplace}~
We assume that $\mu=\sum_{k=0}^{K-1} \dirac{s_k}$ and $\xi=1$. Then $\phi(\pos)=(e^{-s_k\pos})_{k=0}^{K-1} \in\RR^K$ and:
\eq{
  \Co(\pos,\pos') = \sum_{k=0}^{K-1} e^{-s_k(\pos+\pos')}.
}

\myparagraph{Discretized $L^2$-normalized Laplace}~
We let $\mu=\sum_{k=0}^{K-1} \dirac{s_k}$ and $\xi(\pos)=\pa{\sum_{k=0}^{K-1} e^{-2s_k \pos}}^{-1/2}$. Then $\phi(\pos)=\xi(\pos)(e^{-s_k\pos})_{k=0}^{K-1} \in\RR^K$, $\normObs{\phi(\pos)}=1$ and:
\eq{
  \Co(\pos,\pos') = \xi(\pos)\xi(\pos')\sum_{k=0}^{K-1} e^{-s_k(\pos+\pos')}.
}

In contrast to the continuous setups of Section~\ref{subsec:etaW-unnorm}, we do not have closed-form expressions for $\etaW$. However, if a sequence of measures, \eg{} $\mun=\sum_{k=0}^{K_{n}-1} \munk \dirac{s_{n,k}}$ converges in a suitable sense towards the Lebesgue measure $\mu=\Ll$, the following proposition shows that the corresponding $\etaW$ must be nondegenerate for $n$ large enough. We consider both the unnormalized and $L^2$-normalized setups, corresponding respectively to
\begin{align*}
  \Con(\pos,\pos')&=\int_{\RR_+}e^{-(x+x')s}\d\mun(s), \mbox{ and}\\
\Con(\pos,\pos')&=\xi_n(x)\xi_n(x')\int_{\RR_+}e^{-(x+x')s}\d\mun(s) \qwhereq \xi_n(x)=\int_{\RR_+}e^{-2xs}\d\mun(s),
\end{align*}
and similarly for $\Co$ and $\mu=\Ll$.

\begin{prop}
  Let $(\mu_n)_{n\in\NN}$ be a sequence of positive measures which converges towards the Lebesgue measure $\mu$ in the local weak-* topology, \ie{}
  \begin{align*}
\forall  \psi\in\Cder{}_c(\RR_+),\quad    \lim_{n\to +\infty}\int_{\RR_+}\psi(s)\d\mu_n(s) = \int_{\RR_+}\psi(s)\d s,
  \end{align*}
and such that 
  \begin{align}\label{eq:cvphimoment}
    \sup_{n\in\NN} \int_{\RR_+} (1+s^{4N-1})e^{-\xmin s}\d\mun(s)<+\infty.
  \end{align}
Then, both in the unnormalized and the $L^2$-normalized case, for $n$ large enough, the $2N-1$ vanishing derivatives precertificate  $\etaWk{n}$ is $(2N-1)$-nondegenerate. 
\end{prop}

\begin{proof}
  Let us denote by $\Fdn^{[n]}=(\phiD{0}, \ldots,\phiD{2N-1})$ (resp. $\Fdn$)  the impulse response derivatives corresponding to $\mun$ (resp. $\mu=\Ll$), and by $\etaW$ the $2N-1$ vanishing derivatives precertificate for $\mu=\Ll$.
  First, in view of Sections~\ref{par:laplace} and~\ref{par:L2-norm-laplace}, we observe that the result follows immediately if we prove that
\begin{align}
  \lim_{n\to +\infty}\Fdn^{[n]*}\Fdn^{[n]} &= \Fdn^*\Fdn, \label{eq:cvfdn}
\end{align}
(as it implies the linear independence of $(\phiD{0}, \ldots,\phiD{2N-1})$ for $n$ large enough), and that
\begin{align}
\forall i\in \{0,1,\ldots, 2N\},\quad   \lim_{n\to +\infty}	 \normLi{\etaWk{n}^{(i)}-\etaW^{(i)}}&=0,\label{eq:cvetawk}
\end{align}
(as it implies $\abs{\etaWk{n}(\pos)}<1$ for $\pos\neq \poscluster$ and $\etaWk{n}^{(2N)}(\poscluster)<0$ for $n$ large enough).

  We recall from~\eqref{sec:blasso-eq:etaWdefbis} that $\etaWk{n}$ is given by
    $\etaWk{n}(\pos)=\sum_{i=0}^{2N-1} \al_{i}^{[n]}\partial^{(i)}_2\Con(\pos,\poscluster)$
  where $\alpha^{[n]} =(\Fdn^{[n]*}\Fdn^{[n]})^{-1} \dirac{2N}$ (provided the matrix is invertible), and the $(i,j)$-entry of $(\Fdn^{[n]*}\Fdn^{[n]})$ is $\partial^{(i)}_1\partial^{(j)}_2\Con(\poscluster,\poscluster)$.
  As a consequence, both~\eqref{eq:cvfdn} and~\eqref{eq:cvetawk} are established if we can prove that
  \begin{align}\label{eq:cvuphi}
    \lim_{n\to+\infty} \sup_{x,x'\in [\xmin,\xmax]} \abs{\partial^{(i)}_1\partial^{(j)}_2\Con(x,x')-\partial^{(i)}_1\partial^{(j)}_2\Co(x,x')}=0,
  \end{align}
  for all $i\in \{0,\ldots, 2N\}$, $j\in \{0,\ldots, 2N-1\}$.

  First, we prove~\eqref{eq:cvuphi} in the unnormalized case, \ie{} $\Con(\pos,\pos')=\int_{\RR_+}e^{-(x+x')s}\d\mun(s)$. The dominated convergence theorem ensures that $\partial^{(i)}_1\partial^{(j)}_2\Con(x,x')=\int_{\RR_+}s^{i+j}e^{-(x+x')s}\d\mun(s)$ (and similarly for $\Co$ and $\mu$).

Let $(x,x')\in [\xmin,\xmax]^2$ and let $\psi\in \Cder{}_c(\RR_+)$ such that $\psi(s)=1$ for $s\in [0,1]$, $\psi(s)=0$ for $s\geq 2$, and $0\leq \psi\leq 1$ on $\RR_+$.
We denote by $C$ the supremum in~\eqref{eq:cvphimoment}.

  Let $\varepsilon>0$ and $A>0$.  Then,
\begin{align*}
  &\abs{\int_{\RR_+}s^{i+j} e^{-{(x+x')} s}\d\mun(s)- \int_{\RR_+} s^{i+j}e^{-{(x+x')} s}\d s}\\ 
  &\leq  \underbrace{\abs{\int_{\RR_+} s^{i+j}e^{-{(x+x')} s}\psi\left(\frac{s}{A}\right)\d\mun(s)- \int_{\RR_+}s^{i+j} e^{-{(x+x')} s}\psi\left(\frac{s}{A}\right)\d s}}_{\eqdef a}\\
  & + \underbrace{\abs{ \int_{\RR_+} s^{i+j} e^{-{(x+x')} s}(1-\psi\left(\frac{s}{A}\right))\d \mun(s)}}_{=b}+ \underbrace{\abs{ \int_{\RR_+} s^{i+j}e^{-{(x+x')} s}(1-\psi\left(\frac{s}{A}\right))\d s}}_{=c}.
\end{align*}
We have 
\begin{align*}
  c&\leq \int_A^{+\infty} (1+s^{4N-1})e^{-2\xmin s}\d s,\\ \qandq
  b &\leq e^{-\xmin A} \int_{\RR_+} (1+s^{4N-1}) e^{-{\xmin} s}\d\mun(s) \leq e^{-\xmin A} C.
\end{align*}
We choose $A>0$ sufficiently large so that $\int_A^{+\infty} (1+s^{4N-1})e^{-2{\xmin} s}\d s\leq\varepsilon$ and $ e^{-\xmin A} C\leq \varepsilon$, hence $\max(b,c)\leq \varepsilon$.

Now, to prove that $a$ is uniformly small for $(x,x')\in[\xmin,\xmax]^2$ as $n\to +\infty$, we apply Lemma~\ref{lem:uniformcv} to $((x,x'),s)\mapsto s^{i+j}e^{-(x+x')s}\psi\left(\frac{s}{A}\right)$ defined on $[\xmin,\xmax]^2\times [0,2A]$. This yields the desired result. 

The proof for the normalized case readily follows from the uniform convergence of the unnormalized case and the fact that the normalization factors $\xi_n(\pos)=\left(\int_{\RR_+}e^{-2s\pos}\d\mu_n(s)\right)^{-1/2}\leq \left(\int_{\RR_+}e^{-2s\xmax}\d\mu_n(s)\right)^{-1/2}$ are upper bounded by some positive constant independent of $n$.

\end{proof}

\begin{lem}\label{lem:uniformcv}
  Let $X$ and $S$ be two compact metric spaces, and $\psi\in \Cder{}(X\times S)$. If $\{\mun\}_{n\in\NN}$ and $\mu$ are Radon measures such that $\mun \stackrel{*}{\rightharpoonup} \mu$ in the weak-* convergence of $\Mm(S)$, then
  \begin{align*}
    \lim_{n\to +\infty} \int_{S}\psi(x,s)\d\mun(s) =  \int_{S}\psi(x,s)\d\mu(s),
  \end{align*}
  uniformly in $x\in X$.
\end{lem}

\begin{proof}
  
  We note that the mapping $(\eta,\nu)\mapsto \int_S \eta\d \nu$ is continuous on $\Cder{}(S)\times \Mm(S)$. Since $x\mapsto \psi(x,\cdot)$ is continuous from $X$ to $\Cder{}(S)$, the mapping 
\begin{align}
  F: (x,\nu)\longmapsto   \int_S \psi(x,s)\d \nu(s)
\end{align}
is continuous on $X\times \Mm(S)$. 
  
Now, since $S$ is compact, $\Mm(S)$ is the dual of the Banach space $\Cder{}(S)$, and the Banach-Steinhaus theorem implies that there exists $R>0$ such that $\sup_n |\mun|(S)\leq R$ (and $|\mu|(S)\leq R$).

The subspace $\Bb_R\eqdef\enscond{\nu\in\Mm(S)}{ |\nu|(S)\leq R}$ is metrizable for the weak-* topology and compact. As a result, the mapping $F$ is uniformly continuous on the compact $X\times \Bb_R$. In particular, as $\mun\to \mu$ in $\Bb_R$,
\begin{align*}
  \sup_{x\in X} \abs{\int_{S}\psi(x,s)\d\mun(s)-\int_{S}\psi(x,s)\d\mu(s)}\to 0. 
\end{align*}
\end{proof}

Figure~\ref{sec:laplace-fig:etaWapprox} illustrates this convergence between the precertificates in the unnormalized case.

\begin{figure}[!htb]
\centering
\subfigure[$K=10$]{\includegraphics[width=0.32\linewidth]{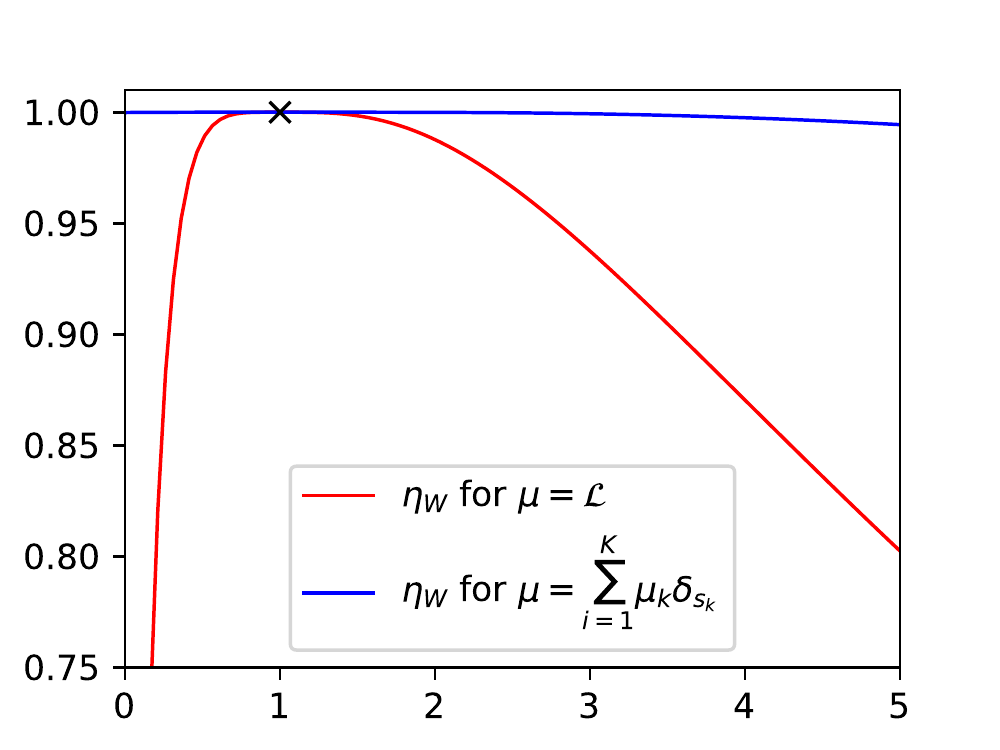}}
\subfigure[$K=120$]{\includegraphics[width=0.32\linewidth]{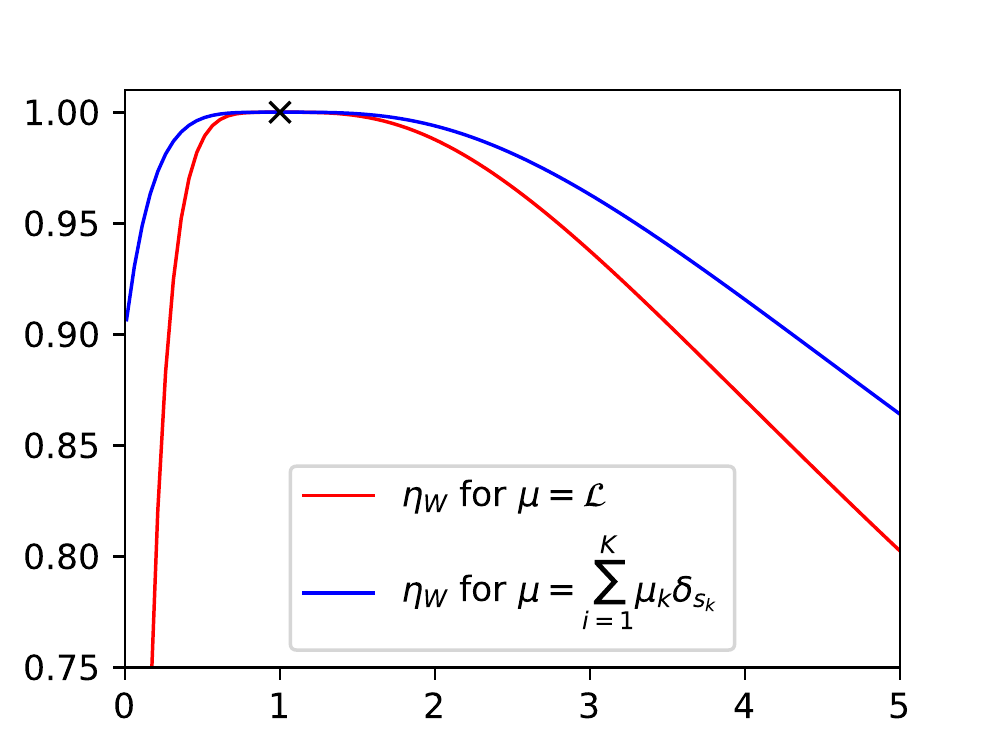}}
\subfigure[$K=800$]{\includegraphics[width=0.32\linewidth]{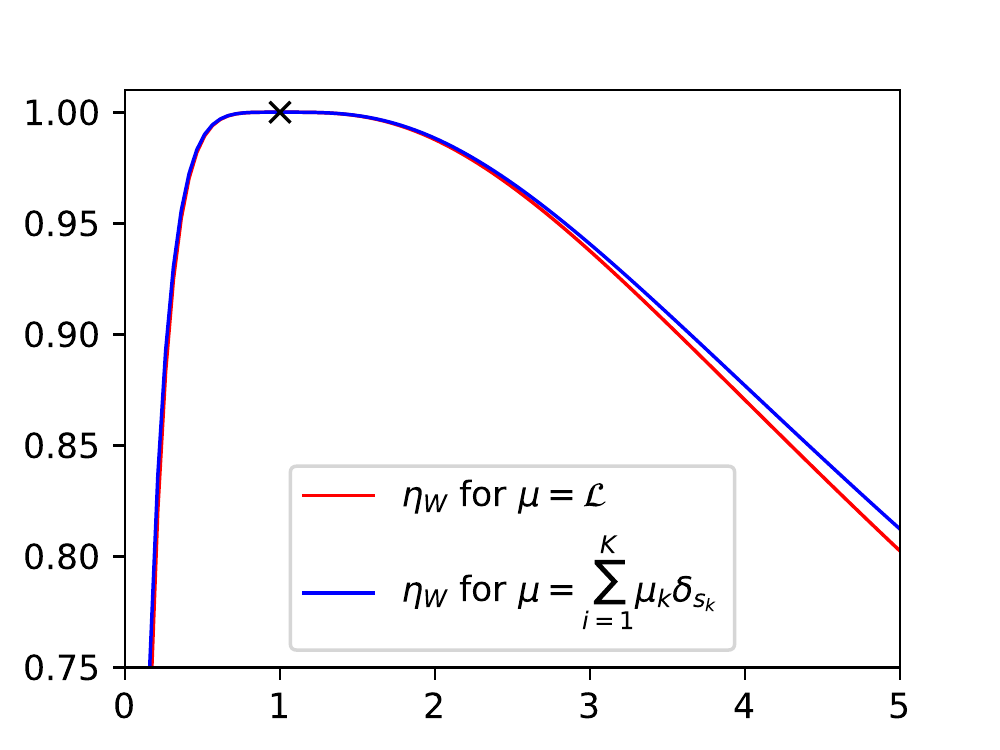}}
\caption{\label{sec:laplace-fig:etaWapprox}Approximation of $\etaW$ for the unnormalized continuous Laplace operator (see Proposition~\ref{prop:etaW-unnorm-Laplace}) by the $\etaW$ obtained for discretized unnormalized Laplace operators.}
\end{figure}

%
\section{The \adcgshort~Algorithm}\label{sec:sfw}

In this section, we present the \adcg~(see Algorithm~\ref{sec:sfw-alg:sfw}), a new version of the modified Frank-Wolfe algorithm introduced in~\cite{bredies-inverse2013}. Moreover, we prove in Theorem~\ref{sec:sfw-thm:cvksteps} that it converges in a finite number of steps under mild assumptions. The code can be found in~\url{https://github.com/qdenoyelle}.

We suppose in this section that $\Pos\subset\RR^d$ is compact, or $\Pos=\TT^d$ with $d\in\NN^*$ and $\phi\in\kernel{2}$ (see Definition~\ref{sec:intro-def:admkernel}).

\subsection{The Algorithm}\label{sec:sfw-subsec:greedy}

\myparagraph{Frank-Wolfe Algorithm}
The Frank-Wolfe (FW) algorithm~\cite{frank-fw1956}, also called the Conditional Gradient Method (CGM)~\cite{levitin-constrained1966} solves the following optimization problem
\begin{equation}\label{eq:minfw}
\min_{m \in C}\ f(m),
\end{equation}
where $C$ is a weakly compact convex set of a Banach space, and $f$ is a differentiable convex function. 
For instance, in the case of sparse recovery problems, $m$ is a measure and $C$ is a subset of $\Mm(X)$.
A chief advantage of FW with respect to most first order optimization scheme (such as gradient descent or proximal splitting method) is that it does not rely on any underlying Hilbertian structure and only makes use of directional derivatives. It is thus particularly well adapted to optimize over the space of Radon measures.
The algorithm is detailed in Algorithm~\ref{sec:sfw-alg:fw}.
\begin{algorithm}
    \caption{Frank-Wolfe Algorithm}
    \label{sec:sfw-alg:fw}
    \begin{algorithmic}[1]
    \For{$k=0,\ldots, n$}
    \State\label{fw-greedy-step} Minimize: $\iter{s}\ni\mbox{argmin}_{s\in C} f(\iter{m})+df(\iter{m})[s-\iter{m}]$.
    \If{$df(\iter{m})[\iter{s}-\iter{m}]=0$}\label{stopping-conditionFW}
       \State $\iter{m}$ solution of \eqref{eq:minfw}. Stop.
    \Else{}
    \State\label{fw-stepresearch} Step research: $\iter{\ga}\gets \frac{2}{k+2}$ or $\iter{\ga}\ni\mbox{argmin}_{\gamma\in [0,1]} f(\iter{m}+\gamma(\iter{s}-\iter{m}))$.
    \State\label{fw-tentativeupdate} Update: $\iterpo{m}\gets \iter{m}+\iter{\ga}(\iter{s}-\iter{m})$. 
    \EndIf
    \EndFor
    \end{algorithmic}
\end{algorithm}

Let us note that the FW algorithm is naturally endowed with a stopping criterion in Step~\ref{stopping-conditionFW} (see for instance~\cite[Ch. 3, Sec.1.2]{demyanov-1970approximate}) which is equivalent to the standard optimality condition for constrained convex problems
\begin{align}\label{eq:fwoptimality}
 \forall s\in C,\quad df(\iter{m})[s-\iter{m}]\geq 0. 
\end{align}

\myparagraph{Frank-Wolfe for the BLASSO}
The FW algorithm cannot be applied directly to the BLASSO because it is an optimization problem over $\radon$ which is not bounded and the objective function
\begin{align}\label{sec:sfw:eq-fobj}
	\forall m\in\radon, \quad \fobj(m)\eqdef\frac{1}{2}\normObs{\Phi m-\obsw}^2+\la\normTVX{m},
\end{align}
is not differentiable. Instead, we propose to consider an equivalent problem to the BLASSO, using an epigraphical lift (following an idea of~\cite{harchaoui-2015conditional}), which is presented in Lemma~\ref{sec:sfw-lem:blasso-eq}.

\begin{lem}\label{sec:sfw-lem:blasso-eq}
The BLASSO
\begin{align}\label{sec:sfw-def:blasso}
	\umin{m\in\radon} \fobj(m) \eqdef \frac{1}{2}\normObs{\Phi m-y}^2+\la \normTVX{m}.\tag{$\blasso$}
\end{align}
 is equivalent to
\begin{align}\label{sec:sfw:blassoeq}\tag{$\blassoeq$}
	\umin{(t,m)\in C} \tfobj(m,t) \eqdef \frac{1}{2}\normObs{\Phi m-y}^2+\la t, 
\end{align}
where we defined $C \eqdef \enscond{(t,m)\in\RR_+\times\radon}{ \normTVX{m}\leq t\leq M }$ and  $M\eqdef\frac{\normObs{\obsw}^2}{2\la}$. 
\end{lem}
The equivalence stated in Lemma~\ref{sec:sfw-lem:blasso-eq} is to be understood in the following sense: $m$ is a solution to~\eqref{sec:sfw-def:blasso} if and only if $(t,m)$ is a solution to~\eqref{sec:sfw:blassoeq} for some $t\geq 0$. Moreover, in that case $t=\normTVX{m}$ and $\tfobj(m,t)= \fobj(m)$.
 As a result,  one can directly translate the FW algorithm (see Algorithm~\ref{sec:sfw-alg:fw}) to $\blasso$.

\begin{proof}
Let $\mlimit$ be a minimizer of $\fobj$ on $\radon$, then we have
\begin{equation}
	\fobj(\mlimit)\leq \fobj(0)=\la M.
\end{equation}
Hence, one can restrict the BLASSO to the set of measures $m\in\radon$ such that $\normTVX{m}\leq M$ and  $\blassoeq$ is obtained using an epigraphical representation.
\end{proof}

The next two remarks discuss the applicability of standard results on FW to the BLASSO.
\begin{rem}[Well-posedness]
The FW algorithm is well defined for $\blassoeq$. Indeed, $\tfobj$ is a differentiable functional on the Banach space $\RR\times\radon$, with differential 
\begin{equation}
\d\tfobj(t,m): (t',m')\longmapsto\int_\Pos \Phi^*(\Phi m-\obsw)\d m' + \la t'.
\end{equation}
Although $C$ is not weakly compact (otherwise, by the Eberlein-Shmulyan theorem, $\radon$ would be reflexive), it is compact for the weak-* topology: as $\d\tfobj(t,m)$ is represented by $(\la,\Phi^*(\Phi m-\obsw))\in \RR\times \ContX(\Pos)$, it does reach its minimum on $C$.
\end{rem}
\begin{rem}[Rate of convergence]
Let us note that $\d\tfobj$ is Lipschitz continuous (because $\phi\in\kernel{2}$), 
 hence by classical results for the study of the convergence of the FW algorithm, one obtains  the $O(1/k)$ rate of convergence in the objective function for any minimizing sequence for the BLASSO.
\end{rem}
\begin{lem}[\protect{\cite[Th. 3.1.7]{demyanov-1970approximate}}]\label{sec:sfw-lem:fw-blasso-rate}
	Let $(t_k,\iter{m})_{k\in\NN}$ be a sequence generated by Algorithm~\ref{sec:sfw-alg:fw} applied to $\blassoeq$. Then, there exists $C_1>0$ such that for any $\mlimit$ solution of $\blasso$ we have
	\begin{equation}
	\forall k\in\NN^*, \quad \fobj(\iter{m})-\fobj(\mlimit)\leq \frac{C_1}{k}.
	\end{equation}
\end{lem}
Next, we discuss how the minimization step yields a greedy approach and a natural stopping criterion. The following two remarks also crucially relate the algorithm to the dual certificate of~\eqref{eq:defblasso}.
\begin{rem}[Greedy approach]
Obviously, the FW algorithm is only interesting if, in step~\ref{fw-greedy-step} of Algorithm~\ref{sec:sfw-alg:fw}, one is able to minimize the linear form $s\mapsto \d\tfobj(\iter{t},\iter{m})[s]$ on $C$. That linear form reaches its minimum at least at one extreme point of $C$, \ie{} $s=(0,0)$ or points of the form  
$s=\pa{M,\pm M\dirac{\pos}}$ for $x\in \Pos$. 
Finding a minimizer among those points amounts to finding a point $x$ in
\begin{align*}
  &\argmin_{x\in \Pos} \left(\pm \frac{1}{\la}\left(\Phi^*(\obsw-\Phi \iter{m})\right)(x)+1\right) \la M,\\
  \mbox{or equivalently in}\quad
  &\argmax_{x\in\Pos}\left(\abs{\iter{\eta}(x)}-1\right) \qwhereq  \iter{\eta}\eqdef {\frac{1}{\la}\left(\Phi^*(\obsw-\Phi \iter{m})\right)}
\end{align*}
(note the similarity of $\iter{\eta}$ with the dual certificate defined in~\eqref{eq:certifdual}).

As a consequence, at each Step~\ref{fw-tentativeupdate} of Algorithm~\ref{sec:sfw-alg:fw}, a new spike is created at some point in $\argmax_\Pos\abs{\iter{\eta}}$ (unless $s=(0,0)$ is optimal, which means that $\normLi{\iter{\eta}}\leq 1$).
This spike creation step is at the core of the algorithms in \cite{bredies-inverse2013} and~\cite{boyd-adcg2015}.
\end{rem}

\begin{rem}[Stopping criterion]\label{rem:stopfw}
  It is interesting to relate the stopping criterion \eq{(\iter{t},\iter{m})\in \argmin_{s\in C}  \d\tfobj(\iter{t},\iter{m})[s],} with the dual certificates for~\eqref{eq:defblasso}. As noted above (see Equation~\ref{eq:fwoptimality}), the stopping criterion is equivalent to $(\iter{t},\iter{m})$ being a solution, hence $\iter{t}=\normTVX{\iter{m}}$. If $\iter{m}\neq 0$, without loss of generality we write $\iter{m}=\sum_{i=1}^{\iter{N}} \iter{a}_i\delta_{\iter{x}_i}$ where the $\iter{x}_i$'s are distinct, so that $\iter{t}=\normTVX{\iter{m}}=\sum_{i}\abs{\iter{a}_i}$. We also set $\iter{\varepsilon}_i\eqdef \sign(\iter{a}_i)$ and $L\eqdef \d\tfobj(\iter{t},\iter{m})$.
  
  Assume first that $\normTVX{\iter{m}}<M$, so that the smallest face of $C$ which contains $(\iter{t},\iter{m})$ is 
  \begin{align*}
   F\eqdef \mathrm{conv} \left\{(0,0),(M,M\iter{\varepsilon}_1\delta_{\iter{x}_1}),\ldots, (M,M\iter{\varepsilon}_{\iter{N}}\delta_{\iter{x}_{\iter{N}}})\right\}.
  \end{align*}
  Since $\argmin_{s\in C} L$ is a face of $C$ containing $(\iter{t},\iter{m})$ (see~\cite[Sec. 18]{rockafellar2015convex}), it must contain $F$. Hence
  \begin{align}\label{eq:sfwminL}
    L(0,0)=L(M,M\iter{\varepsilon}_1\delta_{\iter{x}_1})=\cdots= L(M,M\iter{\varepsilon}_{\iter{N}}\delta_{\iter{x}_{\iter{N}}})=\min_{C}L.
  \end{align}
  Now, if $\normTVX{\iter{m}}=M$, it means that $\tfobj(\iter{t},\iter{m})=\tfobj(0,0)$, so that by convexity of $\tfobj$ and optimality of $(\iter{t},\iter{m})$ one has $L(\iter{t},\iter{m})=\d\tfobj(\iter{t},\iter{m})[\iter{t},\iter{m}]=0=L(0,0)$. As the smallest face which contains $(\iter{t},\iter{m})$ is
  \eq{
  F'\eqdef \mathrm{conv} \left\{(M,M\iter{\varepsilon}_1\delta_{\iter{x}_1}),\ldots, (M,M\iter{\varepsilon}_{\iter{N}}\delta_{\iter{x}_{\iter{N}}})\right\},
  }
  we deduce as above that~\eqref{eq:sfwminL} holds.

  In particular, $L(0,0)\leq \inf_{x\in\Pos} L(M,\pm M\delta_{x})$ yields
  \begin{align}\label{eq:sfwzeroopt}
  0\leq \inf_{x\in\Pos} \left(-\abs{\iter{\eta}(x)}+1 \right),
  \end{align}
  that is $\normLi{\iter{\eta}}\leq 1$. Moreover $L(\iter{t},\iter{m})=\sum_{j=1}^{\iter{N}} \abs{\iter{a}_j} L\left(M,M\iter{\varepsilon}_j\delta_{\iter{x}_j}\right)\leq \sum_{j=1}^N\abs{\iter{a}_j} L(M,\pm M\delta_{\iter{x}_j})$ , yields
\begin{align*}
  -\sum_{j=1}^{\iter{N}} \iter{a}_j\iter{\eta}(\iter{x}_j) \leq -\sum_{j=1}^{\iter{N}} \abs{\iter{a}_j}\abs{\iter{\eta}(\iter{x}_j)},
\end{align*}
from which we deduce $\iter{\eta}(\iter{x}_j)=\sign(\iter{a}_j)$.

As a result, when the FW algorithm stops (if it does), we observe that \emph{the quantity $\iter{\eta}$ it has constructed is the dual certificate} for~\eqref{eq:defblasso}. If $\iter{m}=0$, the argument is similar (as~\eqref{eq:sfwzeroopt} must hold).
\end{rem}


\myparagraph{The \adcg~algorithm}
Applying directly Algorithm~\ref{sec:sfw-alg:fw} yields a sequence of measures $(\iter{m})_{k\in \NN}$ which weakly-* converges towards some solution $\mlimit$ in a greedy way. But the generated measures $\iter{m}$ are not very sparse compared to $\mlimit$, each Dirac mass of $\mlimit$ being approximated by a multitude of Dirac masses of $\iter{m}$ with inexact positions. 
It is therefore suggested in~\cite{bredies-inverse2013}, and strongly advocated in~\cite{boyd-adcg2015}, to  modify the Frank-Wolfe algorithm for the resolution of the BLASSO and to let the Dirac positions move. 

One important feature of the FW algorithm, as noted in~\cite{jaggi2013revisiting,boyd-adcg2015}, is that in the update step~\ref{fw-tentativeupdate}, \emph{the point $\iterpo{m}$ may be replaced with any point $m\in C$ which has lower energy}, without breaking the convergence property and the convergence rate. The Frank-Wolfe algorithm with our modified update step is described in Algorithm~\ref{sec:sfw-alg:sfw}, we call it the \adcg{} (\adcgshort)~algorithm. Since the $t$ variable is only auxiliary in~\eqref{sec:sfw:blassoeq}, we omit it and we formulate directly Algorithm~\ref{sec:sfw-alg:sfw} in terms of $m$ only.

\begin{algorithm}[t]
    \caption{\ADCG~Algorithm}
    \label{sec:sfw-alg:sfw}
    \begin{algorithmic}[1]
    \State Initialize with $\iterO{m}=0$ and $n=0$.
    \For{$k=0,\ldots,n$}
    \State\label{computeNextPos}$\iter{m}=\sum_{i=1}^{\iter{N}} \iter{\amp_i} \dirac{\iter{\pos_i}}$, $\iter{\amp_i}\in\RR$, $\iter{\pos_i}$ pairwise distincts, find $\iter{\pos_*}\in\Pos$ s.t.:
    \eq{%
    	\iter{\pos_*}\in \mathrm{arg} \,  \underset{\pos\in\Pos}{\mathrm{max}}\ |\iter{\eta}(x)| \qwhereq \iter{\eta}\eqdef\frac{1}{\la}\Phi^*(y-\Phi \iter{m}),
} 
    \If{$|\iter{\eta}(\iter{\pos_*})|\leq 1$}\label{stopping-condition}
    \State $\iter{m}$ is a solution of $\blasso$. Stop.
    \Else{}
    \State\label{computeNextAmp} Obtain $\iterph{m}=\sum_{i=1}^{\iter{N}} \iterph{\amp_i}\dirac{\iter{\pos_i}}+\iterph{\amp_{\iter{N}+1}}\dirac{\iter{\pos_*}}$, s.t.: 
    \begin{align*}
	\iterph{\amp}\in \mathrm{arg} \, \underset{{\amp} \in\RR^{\iter{N}+1}}{\mathrm{min}}\ \frac{1}{2}\normObs{\Phi_{\iterph{\pos}} {\amp}-y}^2+\la \normu{{\amp}}
	 \\
	\qwhereq \iterph{\pos}=(\iter{\pos_1},\ldots,\iter{\pos_{\iter{N}}},\iter{\pos_*})
    \end{align*}
    \State\label{BFGS} Obtain $\iterpo{m}=\sum_{i=1}^{\iter{N}+1} \iterpo{\amp_i}\dirac{\iterpo{\pos_i}}$, s.t.:
    \eq{%
	(\iterpo{\amp},\iterpo{\pos})\in\underset{({\amp},{\pos})\in\RR^{\iter{N}+1}\times\Pos^{\iter{N}+1}}{\mathrm{arg \, min}}\ \frac{1}{2}\normObs{\Phi_{{\pos}} {\amp}-y}^2+\la \normu{{\amp}},
    }
    using a non-convex solver initialized with $(\iterph{\amp},\iterph{\pos})$.
    \State Eventually remove zero amplitudes Dirac masses from $\iterpo{m}$.
    \EndIf
    \EndFor
    \end{algorithmic}
\end{algorithm}

As we detail below, the algorithm slightly (but crucially) differs from the one in~\cite{boyd-adcg2015}. The main ingredient is to replace the final update with the minimization of a non-convex minimization problem updating both the positions and the amplitudes of the spikes (whereas~\cite{boyd-adcg2015} update successively the amplitudes and the positions).

\begin{rem}[Links between FW applied to $\blassoeq$ and the~\adcgshort]\label{sec:sfw-rem:links}
  Algorithm~\ref{sec:sfw-alg:sfw} is a valid variant of FW, as the update step decreases more the energy than the standard convex combination using $\iter{\ga}$. Indeed,
\eq{%
	\fobj(\iterpo{m})\leq \fobj(\iterph{m})\leq \fobj(\iter{m}+\iter{\ga}(\sign(\iter{\eta}(\iter{\pos_*}))M\dirac{\iter{\pos_*}}-\iter{m})).
}

It is noteworthy that other forms were previously used in~\cite{bredies-inverse2013,boyd-adcg2015}, but, to our knowledge, the update procedure (Steps~\ref{computeNextAmp} and~\ref{BFGS}) described in the present paper is new. As we show in Theorem~\ref{sec:sfw-thm:cvksteps}, optimizing over \emph{both the amplitudes and the positions} is essential to prove the convergence of the algorithm in a finite number of iterations.
\end{rem}

\begin{rem}[Stopping criterion of the~\adcgshort]
  One may observe that the condition $df(\iter{m})[\iter{s}-\iter{m}]=0$ of Algorithm~\ref{sec:sfw-alg:fw} (or equivalently $\iter{m}\in \argmin_{s\in C} \d f(\iter{m})[s]$) has been replaced with $|\iter{\eta}(\iter{\pos_*})|\leq 1$. In fact the optimality conditions for the non-convex local descent (Step~\ref{BFGS}) at iteration $k-1$ imply 
\eq{
  \forall i\in\{1,\ldots,\iter{N}\}, \quad \iter{\eta}(\iter{\pos}_i)=\sign(\iter{\amp}_i), 
}
whereas $|\iter{\eta}(\iter{\pos_*})|\leq 1$ implies $\normLi{\iter{\eta}}\leq 1$, hence $\iter{\eta}$ is a valid dual certificate. 

With the words of Remark~\ref{rem:stopfw}, Step~\ref{BFGS} implies that $L\left(M,M\iter{\varepsilon}_j\delta_{\iter{x}_j}\right)=0$ for $1\leq j\leq \iter{N}$, whereas the condition $|\iter{\eta}(\iter{\pos_*})|\leq 1$ means $0=L(0,0)=\min_{C}L$. As $m$ is a convex combination of those points, we deduce that $(\normTVX{\iter{m}},\iter{m})\in\argmin_C L$, that is the optimality condition~\eqref{eq:fwoptimality}.
\end{rem}

\begin{rem}[Adaptation for the positive BLASSO]
In many applications, one is often interested in recovering positive spikes (see for example in Section~\ref{sec:microscopy}). As a result, in these cases it is better to add a positivity constraint $m\geq 0$ to the BLASSO. This leads to several changes in Algorithm~\ref{sec:sfw-alg:sfw}
\begin{itemize}
\item the stopping condition $|\iter{\eta}(\iter{\pos_*})|\leq 1$ becomes $\iter{\eta}(\iter{\pos_*})\leq 1$,
\item the LASSO is solved on $\RR_+^{\iter{N}+1}$,
\item the optimization problem of Step~\ref{BFGS} is solved on $\RR_+^{\iter{N}+1}\times\Pos^{\iter{N}+1}$.
\end{itemize}
\end{rem}
%

\myparagraph{Implementation details}
\begin{itemize}
\item A Newton method, initialized by a grid search, is used to to find the maximum of $|\iter{\eta}|$ over the compact domain $\Pos$ in step~\ref{computeNextPos}. The size of the grid depends on the operator~$\Phi$. For example, when $\Phi$ is the convolution by the Dirichlet kernel with cutoff frequency $f_c$, we choose a number of points proportional to $f_c$.
\item The LASSO problem at step~\ref{computeNextAmp} is solved using the fast iterative shrinkage thresholding algorithm (FISTA)~\cite{beck-fista2009}.
\item To solve the non-convex optimization problem at step~\ref{BFGS}, we deploy a bounded BFGS. It allows to enforce the positions $x_i$ to be in the compact domain $\Pos$ and to preserve the sign of the amplitudes $a_i$. These constraints ensure the differentiability of the objective function which is required by BFGS.
\end{itemize}

\subsection{Study of the Convergence of the \adcgshort~Algorithm}\label{sec:sfw-subsec:conv-sfw}
We now study the convergence properties of the \adcg~algorithm presented last section (see Algorithm~\ref{sec:sfw-alg:sfw}). Our main result is Theorem~\ref{sec:sfw-thm:cvksteps} where one shows that if $\etaLL=\frac{1}{\lambda}\Phi^*(y-\Phi \meas)$, where $\meas=\sum_{i=1}^N \amp_i \dirac{\pos_i}$, is the unique solution of $\blasso$ and is nondegenerate (see Equation~\eqref{sec:sfw-eq:etaLnondegen}), then Algorithm~\ref{sec:sfw-alg:sfw} recovers $\meas$ in a finite number of iterations. But, first, one shows that our algorithm produces a sequence of measures $(\iter{m})_{k\in\NN}$ that converges towards $\mlimit$ (if $\mlimit\in\radon$ is the unique solution of the BLASSO) for the weak-* topology on $\radon$.

\begin{prop}\label{sec:sfw-prop:cvfaible}
  Let $(\iter{m})_{k\in\NN}$ be the sequence obtained from the \adcg~algorithm. Then it has an accumulation point for the weak-* topology on $\radon$, and that point is a solution to~\eqref{sec:sfw-def:blasso}.
\end{prop}

\begin{proof}
By Remark~\ref{sec:sfw-rem:links}, we know that $(\iter{m})_{k\in\NN}$ is a sequence obtained by applying Algorithm~\ref{sec:sfw-alg:fw} to $\blassoeq$ where the final update is step~\ref{computeNextAmp} and~\ref{BFGS} of the \adcgshort. As a result, using Lemma~\ref{sec:sfw-lem:fw-blasso-rate}, one gets that for any $\mlimit$ solution of $\blasso$,
\eq{%
	\forall k\in\NN, \quad \fobj(\iter{m})-\fobj(\mlimit)\leq\frac{C_1}{k}.
}
Hence $(\iter{m})$ is a bounded minimizing sequence. One can extract from it a subsequence that converges towards some $m\in \radon$ (with $\normTVX{m}\leq M$) for the weak-* topology. Since $\fobj$ is convex and l.s.c., it is also weak-* l.s.c.\ so that one obtains:
\eq{%
	\fobj(m)=\fobj(\mlimit).
}
Hence $m$ is a solution of $\blasso$.
\end{proof}

From this Proposition, one easily deduces the following Corollary.
\begin{cor}\label{sec:sfw-cor:weakstarcv}
 If $\mlimit\in\radon$ is the unique solution of $\blasso$ then $(\iter{m})_{k\in\NN}$ weak-* converges towards $\mlimit$.
\end{cor}

In fact, under mild assumptions, our algorithm even converges towards the solution of the BLASSO in a finite number of iterations, thanks to the displacement of the spikes over the continuous domain $\Pos$. For the sake of clarity, we state and prove this Theorem in the case of $d=1$ but the changes for $d\in\NN^*$ can be easily done.

\begin{thm}\label{sec:sfw-thm:cvksteps}
 Suppose that $\phi\in\kernel{2}$, that $\meas=\sum_{i=1}^N\amp_i\dirac{\pos_i}$ is the unique solution of $\blasso$, and that $\etaLL=\frac{1}{\la}\Phi^*(y-\Phi \meas)$ is nondegenerate, \ie
	\begin{align}\label{sec:sfw-eq:etaLnondegen}
		\forall \pos\in\Pos\setminus\bigcup_{i=1}^N\{\pos_i\}, \quad |\etaLL(\pos)|<1 \qandq \forall i\in\{1,\ldots,N\}, \quad \etaLL''(\pos_i)\neq 0.
	\end{align}
	Then Algorithm~\ref{sec:sfw-alg:sfw} recovers $\meas$ after a finite number of steps (\ie there exists $k\in\NN$ such that $\iter{m}=\meas$).
\end{thm}

\begin{proof}
Since $\meas$ is the unique solution of $\blasso$, one knows by Corollary~\ref{sec:sfw-cor:weakstarcv} that the sequence $(\iter{m})_{k\in\NN}$ produced by Algorithm~\ref{sec:sfw-alg:sfw} converges for the weak-* topology towards $\meas$.

As $\Phi$ is weak-* to weak continuous and by defining $\iter{p}\eqdef\frac{1}{\la}(y-\Phi \iter{m})$, one gets that $(\iter{p})_{k\in\NN}$ converges towards $\pLL$  in the weak topology of $\Obs$ and that $\iter{\eta}\eqdef\Phi^*\iter{p}$ converges pointwise towards $\etaLL$. Then one can show that $\Phi^*$ is a compact operator. Indeed, for any bounded subset $A\in\Obs$,  one can check easily that $\Phi^*A$ is equicontinuous and pointwise relatively compact so that by Ascoli theorem $\Phi^*A$ is relatively compact for the strong topology of $\ContX(\Pos,\RR)$. As a result one can extract a subsequence of $(\iter{\eta})_{k\in\NN}$ that converges towards $\etaLL$ in uniform norm. $\etaLL$ is then the unique accumulation point in the uniform norm of the bounded sequence $(\iter{\eta})_{k\in\NN}$ hence its convergence towards $\etaLL$ in uniform norm. One can repeat this argument for $(\iter{\eta}{}')_{k\in\NN}$ and $(\iter{\eta}{}'')_{k\in\NN}$ (since $\phi\in\kernel{2}$), obtaining for all $j\in\{0,1,2\}$
\begin{align}\label{sec:sfw-eq:cvetak}
	(\iter{\eta})^{(j)}\overset{\normLi{\cdot}}{\underset{k\to+\infty}{\longrightarrow}}\etaLL^{(j)}.
\end{align}

Because $\etaLL$ is nondegenerate, there exists a small neighborhood around each $\pos_i$ on which  $\etaLL''\neq0$. Hence, we deduce from Equation~\eqref{sec:sfw-eq:cvetak} that there exist $\epsilon>0$ and $k_1\in\NN$ such that:
\eq{
	\forall k\geq k_1,\forall i\in\{1,\ldots,N\},\forall\pos\in]\pos_i-\epsilon,\pos_i+\epsilon[, \quad \iter{\eta}{}''(\pos)\neq 0.
}
We denote in the following
\eq{
	I_{\pos_i,\varepsilon}\eqdef]\pos_i-\epsilon,\pos_i+\epsilon[, \quad \forall i\in\{1,\ldots,N\}.
}

Since $\iter{m}$ converges towards $\meas$ in the weak-* topology and $\abs{\meas}$ does not charge the boundary of $I_{\pos_i,\varepsilon}$, we have
\eq{\forall i\in\{1,\ldots,N\},\quad \iter{m}(I_{\pos_i,\varepsilon})\to \meas(I_{\pos_i,\varepsilon})=\amp_i\neq 0,}
so that there exists $k_2\in\NN$ such that for all $k\geq k_2$, $\iter{m}$ has at least one spike in each $I_{\pos_i,\varepsilon}$. In particular $\iter{m}$ has at least $N$ spikes.

Again, from Equation~\eqref{sec:sfw-eq:cvetak}, since $(\iter{\eta})_{k\in\NN}$ converges uniformly towards $\etaLL$, one deduces that there exists $k_3\in\NN$ such that for all $k\geq k_3$:
\eq{
	\sat{\iter{\eta}}\subset \pa{\sat{\etaLL}}\oplus\pa{]-\epsilon,\epsilon[\times\{0\}},
}
where the set of saturation points of a given $\eta\in\ContX(\Pos,\RR)$ is defined as:
\begin{align*}
	\sat{\eta}\eqdef\left\{(\pos,v)\in\Pos\times\{-1,1\}; \ \eta(\pos)=v\right\}.
\end{align*}
Moreover,
\eq{
	\forall \pos\in\Pos\setminus\bigcup_{i=1}^N I_{\pos_i,\varepsilon}, \quad |\iter{\eta}(\pos)|<1.
}
In particular for $k\geq k_3$, $\iter{m}$ has no spikes in $\Pos\setminus\bigcup_{i=1}^N I_{\pos_i,\varepsilon}$ because it would contradict the optimality conditions of Step~\ref{BFGS} of Algorithm~\ref{sec:sfw-alg:sfw}: for all $i\in\{1,\ldots,\iter{N}\}$, $\iter{\eta}(\iter{\pos_i})=\sign(\iter{\amp_i})$.

Suppose now that $k\geq \max(k_1,k_2,k_3)$. Then $\iter{m}$ has at least one spike in each neighborhood of $\pos_i$ and no spikes outside. Moreover $|\iter{\eta}|<1$ outside the neighborhoods and $\iter{\eta}{}''\neq 0$ inside. Let $i\in\{1,\ldots,N\}$ and denote $\iter{\pos_j}\in I_{\pos_i,\varepsilon}$ a position of a spike of $\iter{m}$. From the optimality conditions of Step~\ref{BFGS}, one has also that $\iter{\eta}{}'(\iter{\pos_j})=0$. This combined with  $\iter{\eta}{}''\neq 0$ in $I_{\pos_i,\varepsilon}$ implies that $|\iter{\eta}|<1$, except at $\iter{\pos_j}$. Hence,  $\iter{m}$ has exactly one spike in this neighborhood. As a consequence, we proved that $\iter{m}$ has exactly $N$ spikes (one inside each neighborhood) and:
\eq{
	\forall \pos\in\Pos\setminus\bigcup_{i=1}^N\{\iter{\pos_i}\}, \quad |\iter{\eta}(\pos)|<1.
}
Hence $\iter{m}$, composed of $N$ spikes, is a solution of $\blasso$. Since $\meas$ is supposed to be the unique solution of $\blasso$, one concludes that:
\eq{
	\iter{m}=\meas,
}
\ie the algorithm recovers $\meas$ in a finite number of iterations.
\end{proof}

Note that one proved the convergence in a \emph{finite} number of iterations but not exactly $N$ iterations if $\meas$ is composed of $N$ spikes. However in practice this is exactly what we observe.

\subsection{Illustration of the $N$-Steps Convergence of the \adcgshort}\label{sec:sfw-subsec:algonumexp}

We now illustrate how the algorithm works and we shows that it converges in exactly $N$ iterations in practice (when the noise level and the regularization parameter are appropriate, \ie $\max(\la,\normObs{w}/\la)$ is low enough).

We consider $\Pos=[0,1]$ and a convolution operator with a sampled Gaussian kernel for $\Phi$
\eq{
  \Phi: m\in\Mm(\Pos)\mapsto \int_{[0,1]} \phi\d m\in\RR^K \qwhereq \phi(\pos)=\pa{\frac{1}{\sqrt{2\pi\sigma^2}}e^{-\frac{(\frac{i-1}{K-1}-\pos)^2}{2\sigma^2}}}_{1\leq i\leq K}.
}
We set $\sigma=0.05$ and $K=100$. The initial measure used is $\measO=1.3\dirac{0.3}+0.8\dirac{0.37}+1.4\dirac{0.7}$ and the noise is small ($w=10^{-4}w_0$ where $w_0=\mbox{randn}(K)$).

\begin{figure}[!htb]
\centering
\includegraphics[width=.5\linewidth]{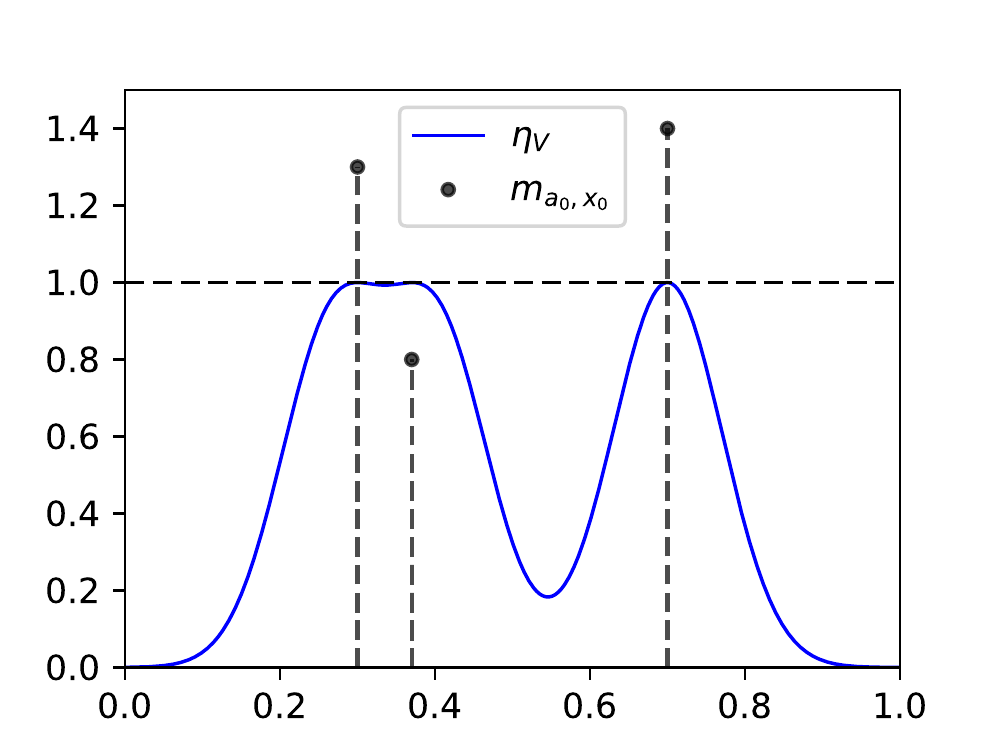}
\vspace{-0.3cm}
\caption{$\etaVV$ for $\measO=1.3\dirac{0.3}+0.8\dirac{0.37}+1.4\dirac{0.7}$.}\label{sec:sfw-fig:Nstep-etaV}
\end{figure}

Figure~\ref{sec:sfw-fig:Nstep-etaV} shows $\etaVV$ for this configuration. One can see that it is nondegenerate. Hence, in a small noise now regime, with the appropriate choice of $\lambda$, there is a unique measure solution of $\blassoplus$ which is composed of the same number of spikes as $\measO$. Moreover,  by Theorem~\ref{sec:sfw-thm:cvksteps}, the \adcgshort~algorithm recovers it in a finite number of iterations.

\begin{figure}[!htb]
\centering
\includegraphics[width=.5\linewidth]{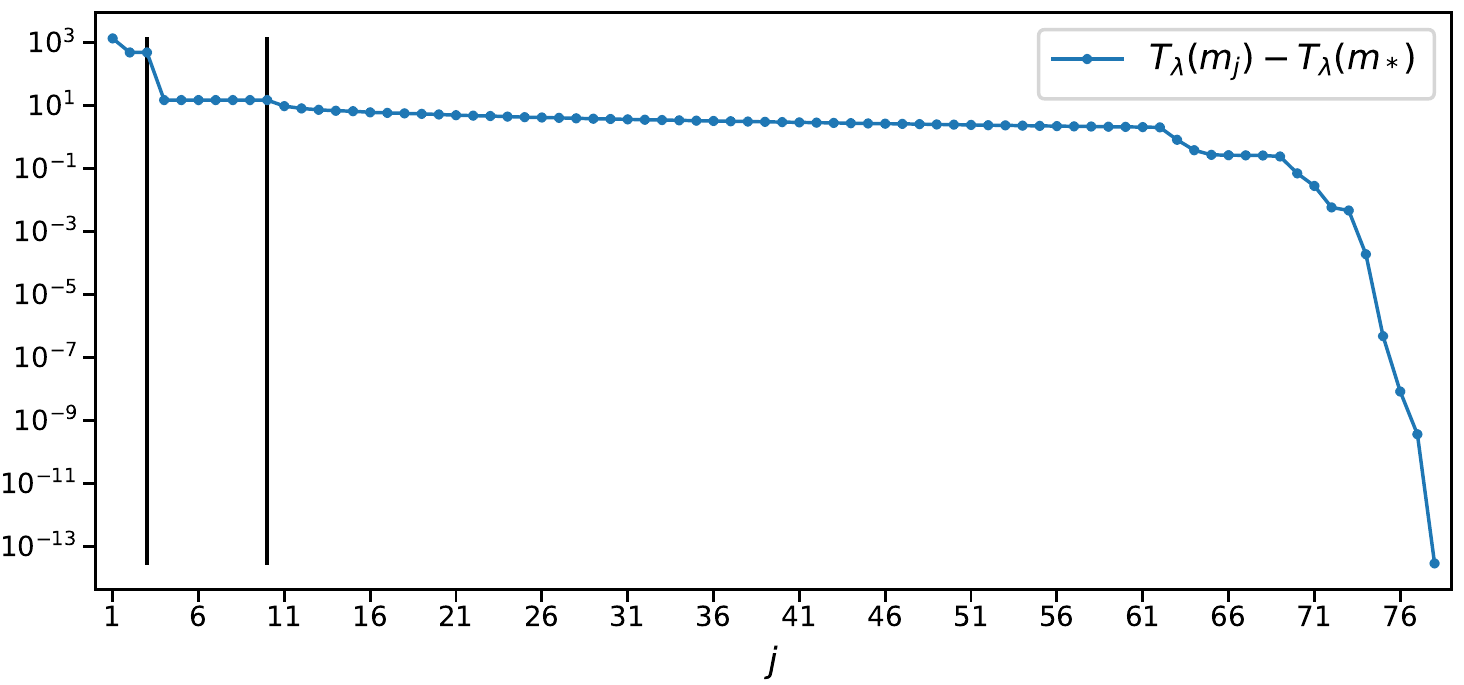}
\vspace{-0.3cm}
\caption{Values of the objective function throughout the \adcgshort~algorithm (cumulative iterations of the BFGS). The vertical black lines separate the main outer iterations of the algorithm.}\label{sec:sfw-fig:Nstep-fobj}
\end{figure}
The decrease of the objective function throughout the algorithm iterations  (cumulative iterations of  BFGS) is presented  in Figure~\ref{sec:sfw-fig:Nstep-fobj}.
As indicated by the two vertical black lines, which show the intermediate iterations, the algorithm converges in exactly $3$ iterations. One can observe an important decrease of the objective function  each time a spike is added.  Also, it is noteworthy that BFGS converges with very few iterations when $k=0$ and $k=1$ (first two spikes added) and that the main computational load for the non-convex step occurs for $k=2$ (more iterations of  BFGS).

Figure~\ref{sec:sfw-fig:Nstep-cv} shows $\iter{m}$ and $\iter{\eta}$ at different times of the algorithm. More precisely, for $k\in\{0,1,2\}$,
we display  the initial measure $\measO$, the recovered measure, and the associated $\eta$. Moreover, we present them after the LASSO step (\ie $\iterph{m}$ and $\iterph{\eta}$) as well as after the BFGS step (\ie $\iterpo{m}$ and $\iterpo{\eta}$) .

One remarks, as expected, that for all $i$, $\iterph{\eta}(\pos_i)=1$, $\iterpo{\eta}(\pos_i)=1$ and $\iterpo{\eta}{}'(\pos_i)=0$. In the first two main iterations, the spikes are almost not moved by the BFGS. However, at the last iteration, the displacement of the positions and amplitudes of the spikes  is crucial to obtain $\iterpo{\eta}\in\partial \normTVX{\iterpo{m}}$, and thus recover the solution of $\blassoplus$ in three steps.

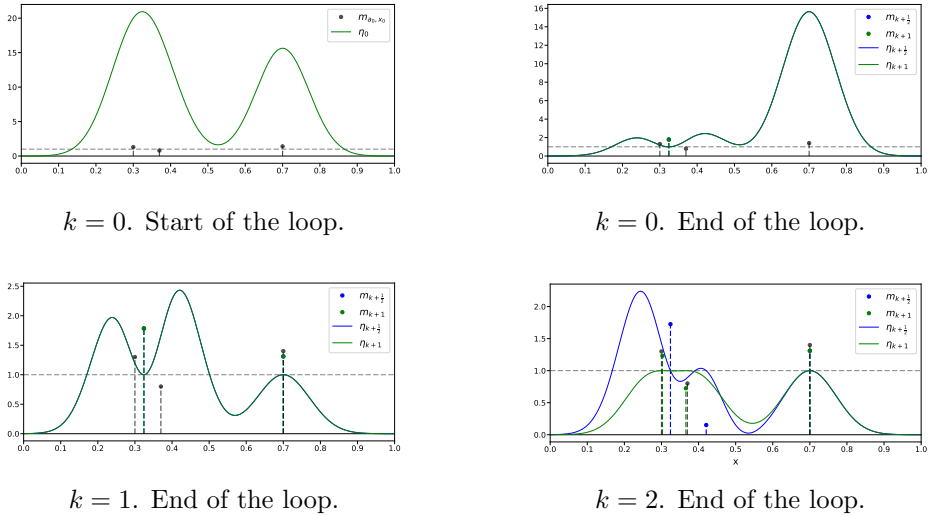
\begin{figure}
\centering
\begin{tikzpicture}
\node at (0,0) {\includegraphics[width=.4\linewidth]{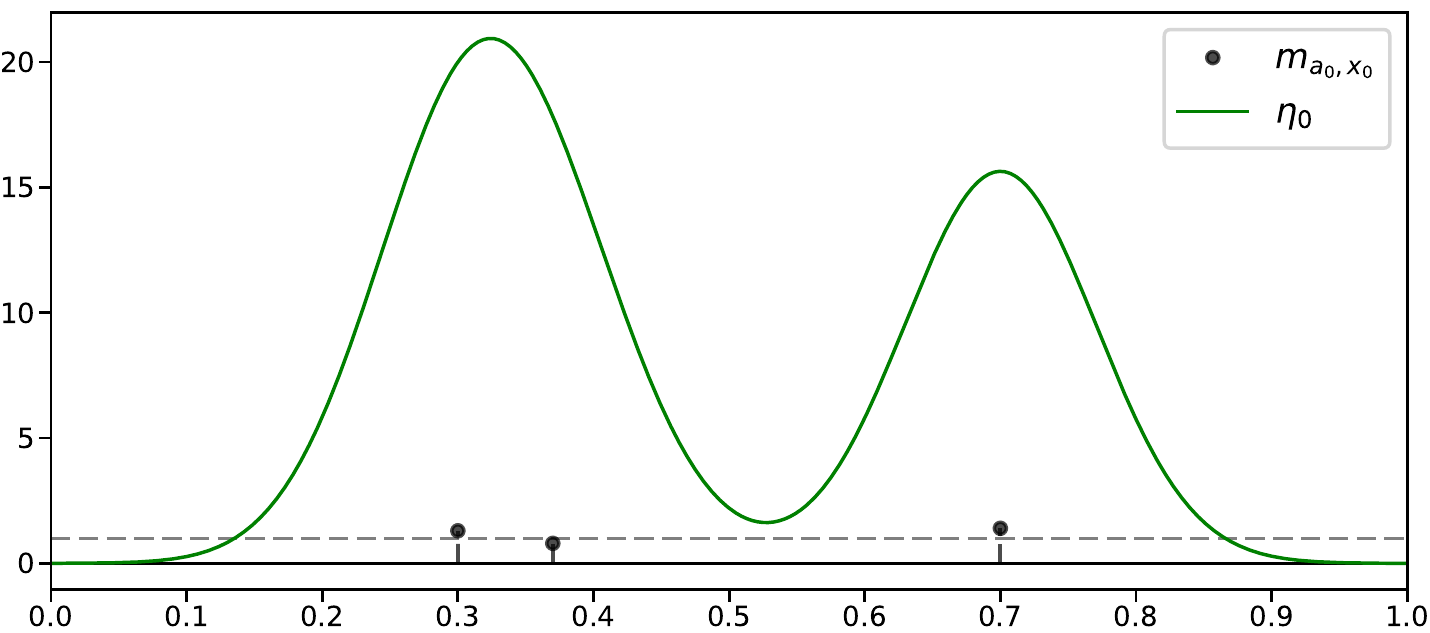}};
\node[fill=white] at (0,-1.8) {$k=0$. Start of the loop.};
\node at (7,0) {\includegraphics[width=.4\linewidth]{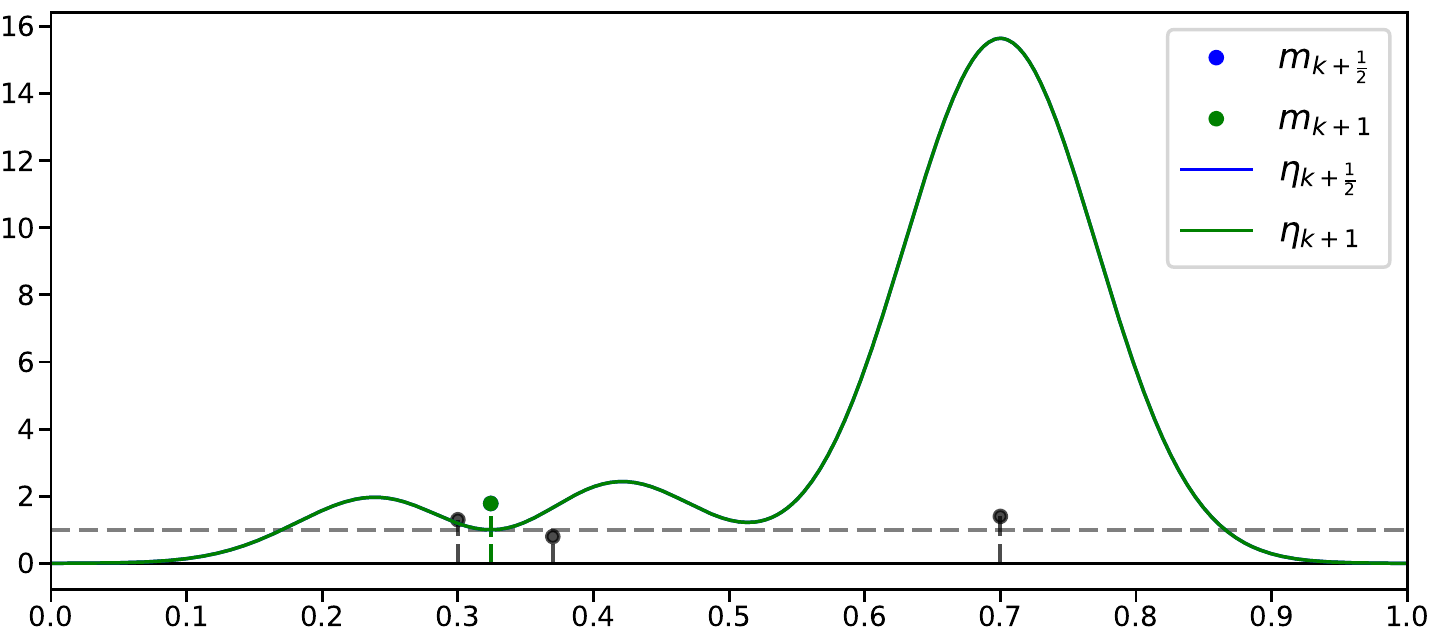}};
\node[fill=white] at (7,-1.8) {$k=0$. End of the loop.};
\node at (0,-3.7) {\includegraphics[width=.4\linewidth]{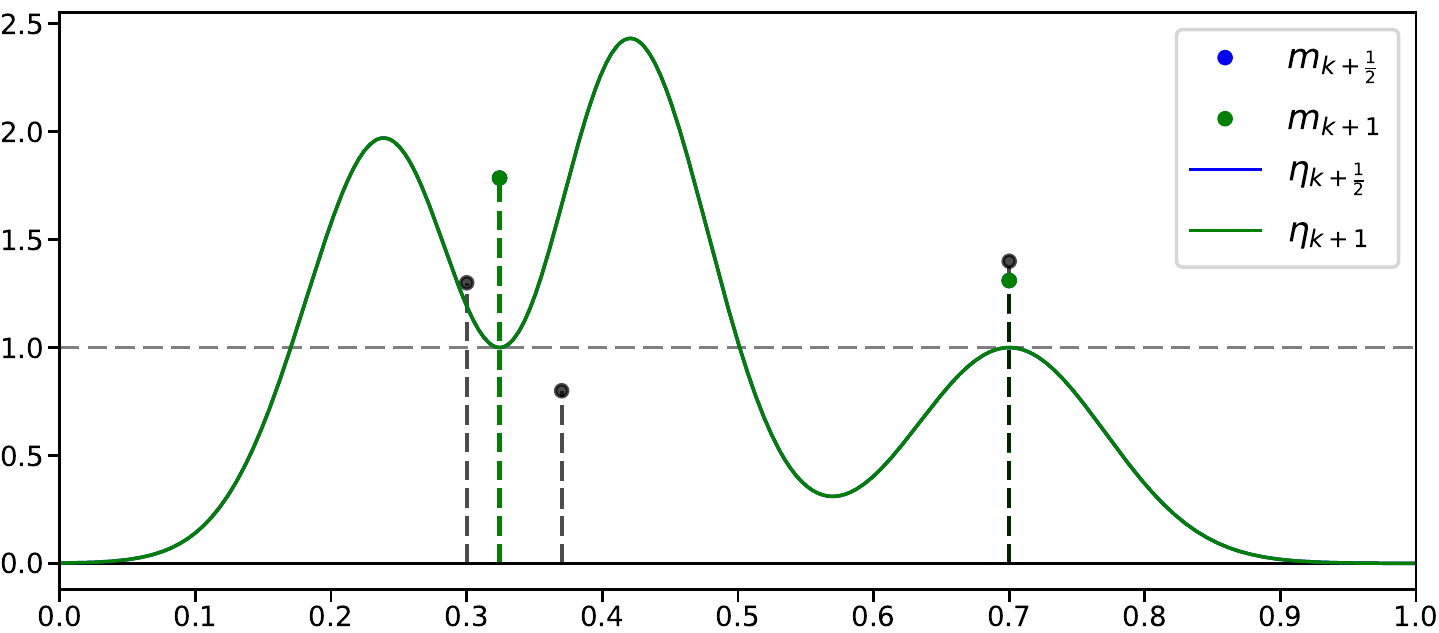}};
\node[fill=white] at (0,-5.5) {$k=1$. End of the loop.};
\node at (7,-3.8) {\includegraphics[width=.4\linewidth]{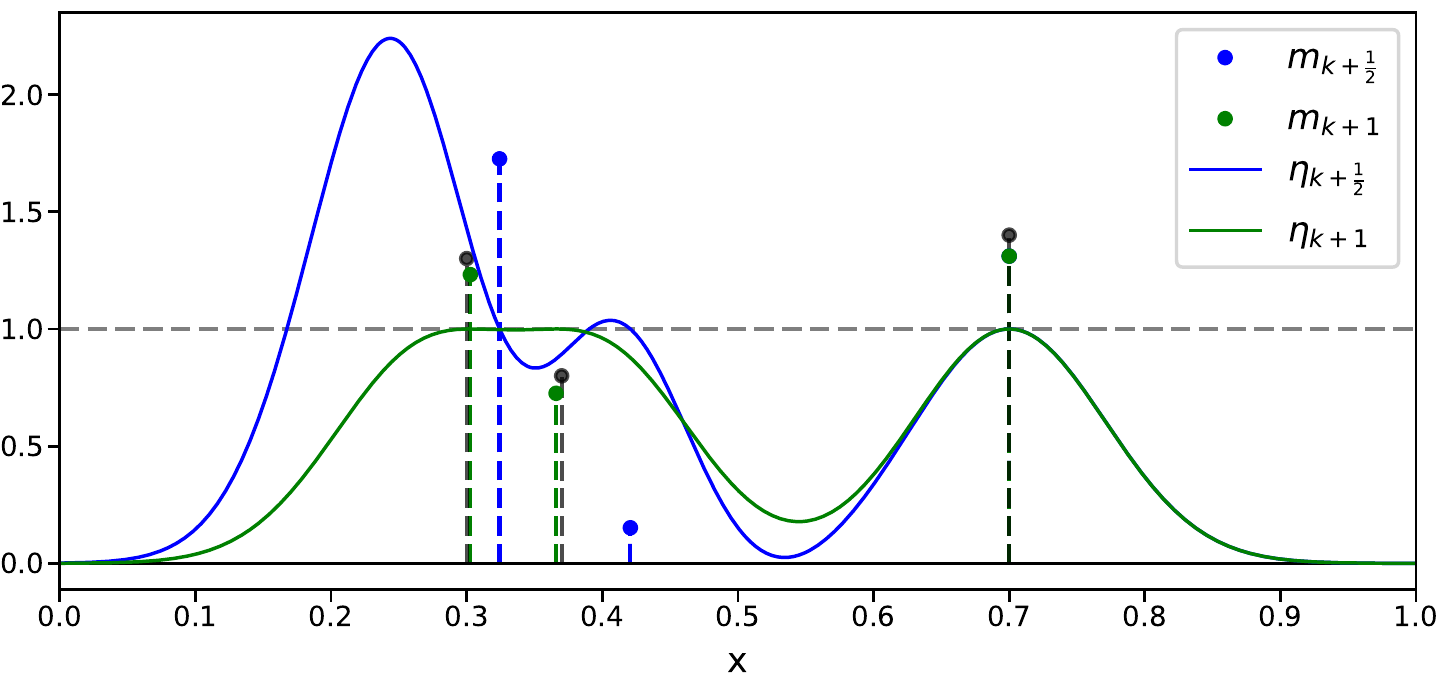}};
\node[fill=white] at (7,-5.5) {$k=2$. End of the loop.};
\end{tikzpicture}
\caption{Main steps of the \adcgshort~algorithm.}\label{sec:sfw-fig:Nstep-cv}
\end{figure}



%


\section{Single Molecule Localization Microscopy}\label{sec:microscopy}

The field of fluorescent microscopy has experienced an important revolution during the past two decades with the emergence of  super-resolution techniques. These modalities, such as structured illumination microscopy (SIM) \cite{gustafsson-surpassing2000}, stimulated emission depletion (STED) \cite{Hell:94}, or single molecule localization microscopy (SMLM)---which includes photoactivated localization microscopy (PALM)  \cite{Betzig1642,hess-ultra2007} and stochastic optical reconstruction microscopy (STORM) \cite{rust2006sub}---bypass the diffraction limit so as to reach unprecedented  nanoscale resolution.  The main principle behind these methods relies on a combined use of optics and numerical processing, which is commonly  called computational imaging. The resolution improvement is thus directly related to the performance of the reconstruction algorithms employed to process the acquired data. 

SMLM techniques use photoactivables fluorescent probes to sequentially image a subset of activated molecules. Then, dedicated algorithms are deployed to  precisely extract the position of these molecules. While the difficulty of the localization problem increases with the density of activated molecules per acquisitions, low density activations drastically reduce the temporal resolution of the system which makes the method limited for live imaging. Hence, current trends in SMLM concern the development of efficient algorithms dealing with high density data for which classical point-spread function (PSF) fitting or centroid localization methods \cite{henriques2010quickpalm} fail. In particular, off-the-grid sparse regularized methods have shown their efficiency for high density settings \cite{huang2017super,boyd-adcg2015}.   For a  complete review and comparisons of existing methods, we refer the reader to the two recent SMLM challenges \cite{sage-quantitative2015,Sage362517}.

 Initially introduced for two-dimensional imaging, SMLM has been extended to 3D thanks to PSF engineering.  The principle relies on the design of PSFs which vary in the axial direction  (\ie  z) in order to encode an information about the depth of molecules. Conventional PSF models include  astigmatism  \cite{huang-three2008} and  double-helix  \cite{rama-three2009}. An alternative to PSF engineering is to record  simultaneously multiple focal planes, as in the biplane modality  \cite{juette-three2008}. It is noteworthy that these two approaches can also be combined  as in \cite{huang-3dastmultifp} where the authors use both an astigmatism PSF and multi-focal acquisitions.  \\
 
 In this section, we study the performance of the SFW algorithm on both astigmatism and double-helix modalities with various number of focal planes (typically from $1$ to $4$). We emphasize that conventional astigmatism and double-helix SMLM devices---in particular commercial ones---use a single focal plane. As opposed to single-focal acquisitions, multi-focal acquisitions require to mount and synchronize several cameras in parallel. To the best of our knowledge, such a setting has only been reported by Huang et al~\cite{huang-3dastmultifp} for the astigmatism SMLM. 
Moreover, we propose to compare these two modalities to an alternative approach where depth information is extracted from multi-angle total internal reflection fluorescence (MA-TIRF) microscopy acquisitions. Such an approach has never been reported yet and we expect our numerical simulations to serve as a proof of concept for further developments. One of the main interest in combining SMLM with MA-TIRF is that classical PSFs, which are better localized laterally than astigmatism or double-helix, can be used. This would reduce the difficulty of lateral molecule localization for high density settings while recovering the depth  through the MA-TIRF acquisitions. 

\subsection{Forward Operators} \label{sec:FwdModels}
 
In this section, we define the forward operator $\Phi$ for the three modalities considered in this paper. The first two correspond to conventional three-dimensional SMLM with astigmatism or double-helix PSFs. The third one, on the contrary, uses a MA-TIRF excitation in order to get an information about the depth of molecules. The operator  $\Phi : \radon \rightarrow \RR^{N_1N_2K}$ maps the  Radon measures $m\in\radon$ to the discrete noiseless measurements $\Phi m \in \RR^{N_1N_2K}$,
\begin{equation}\label{eq:Fwd}
  \Phi m=\int_X \phi(x)\d m(x).
\end{equation}
It is fully characterized by the function $\phi : X \rightarrow \RR^{N_1N_2K}$. Hence, for each modality, we only have to define $\phi$. In the following, $X \eqdef [0,b_1] \times [0,b_2] \times [0,b_3]$ is a subset of $\RR^3$, and we write $x=(x_1,x_2,x_3)\in X$. Then, we consider a camera containing $N_1 \times N_2 $ pixels and we denote the center of the ith pixel by $(c_{i,1},c_{i,2})$.
 Finally,  we provide expressions of $\phi$ which enclose the integration over camera pixels 
$$
\Omega_i \eqdef (c_{i,1},c_{i,2})+\left[-\frac{b_1}{2N_1},\frac{b_1}{2 N_1}\right]\times \left[-\frac{b_2}{2 N_2},\frac{b_2}{2 N_2}\right] \subset \Omega \eqdef [0,b_1] \times [0,b_2].
$$

\paragraph{Astigmatism model.} 

This modality provides depth information using an astigmatism deformation of the PSF with respect to the axial direction $z$. It is customary to model the latter with a Gaussian function whose variances $\sigma_1$  and $\sigma_2 $ vary with $z$ according to \cite{huang2017super,kirshner2013}
\begin{equation}\label{eq:astig_var}
	\sigma_1(z) \eqdef \sigma_0\sqrt{1+\pa{\frac{\alpha z-\beta}{d}}^2} \qandq
  	\sigma_2(z) \eqdef \sigma_1 (-z).
\end{equation} 
The constants involved in~\eqref{eq:astig_var} can be calibrated from real data~\cite{huang-three2008,kirshner2013}.  Then, integrating this Gaussian model over camera pixels,  we have for all $i\in \{1,\ldots, N_1 N_2\}$ and $k \in \{1,\ldots,K\}$
$$
	[\phi(x)]_{i,k} \eqdef\frac{1}{2\pi\sigma_{1}(x_3-z_k)\sigma_{2}(x_3-z_k)}  \int_{\Omega_i} e^{-\left(\frac{(x_1-s_1)^2}{2\sigma^2_{1}(x_3-z_k)}+\frac{(x_2-s_2)^2}{2\sigma^2_{2}(x_3-z_k)}\right)} \d s_1 \d s_2,
$$
where $(z_k)_{k=1}^K$ are the positions of the considered focal planes.



\paragraph{Double-helix model.}
 
 Here, depth information is obtained by using a PSF formed out of two lobes which coil around each other along $z$  to form a double-helix shape. In this paper, we model these lobes by two Gaussian functions with fixed variances $\sigma_1= \sigma_2$, and with a center whose lateral position $(r_1,r_2) $ (respectively, $(-r_1,-r_2)$) varies with $z$ according to
 \begin{equation}
	r_1(z) \eqdef \frac{\om}{2}\cos(\theta(z)) \; \text{ and } \;
	r_2(z) \eqdef -\frac{\om}{2}\sin(\theta(z)) \; \text{ where } \; 
\theta(z)=\theta_{\mathrm{speed}} z \label{sec:microscopy-eq:theta}.
\end{equation}
Parameters $\omega >0$ and $\theta_{\mathrm{speed}}>0$ correspond to the distance between the two Gaussian and the rotation speed of the double-helix (rad/nm), respectively. Then, integrating this model over camera pixels,  we have for all $i\in \{1,\ldots, N_1 N_2\}$ and $k \in \{1,\ldots,K\}$
$$
 	[\phi(x)]_{i,k} \eqdef  \frac{1}{{2\pi}\sigma_{1}\sigma_{2}} \sum_{u\in\{-1,1\}}   \int_{\Omega_i}  e^{-\left( \frac{(x_1+u   r_1(x_3-z_k)-s_1)^2}{2\sigma^2_{1}} +\frac{(x_2+u   r_2(x_3-z_k)-s_2)^2}{2\sigma^2_{2}}\right)  } \d s_1 \d s_2,
$$
where $(z_k)_{k=1}^K$ are the positions of the considered focal planes.

\paragraph{MA-TIRF model.} 

With this modality, each activated set of molecules is imaged using $K\in \NN$ TIRF illuminations with incident angles $(\al_k)_{k =1}^K$.  Let $\nii>0$ and $n_t>0$ be the refractive indices of the incident (\ie glass coverslip) and the transmitted (\ie sample) medium, respectively. A TIRF excitation is obtained when the incident angle $\al$ is  greater than the critical angle $\alc = \arcsin (n_t/\nii)$  for which we have total internal reflection of the light within the incident medium. This phenomenon produces an evanescent wave which decays in the transmitted medium as $ \exp (-s x_3) $, where $s = ({4\pi\nii})/{\lamb}\pa{\sin^2(\al)-\sin^2(\alc)}$ is the penetration depth and $\lamb$ is the wavelength of the incident laser beam~\cite{axelrod1981cell,axelrod-total2008}. Because the decay of this evanescent excitation vary with the incident angle, the depth of biological structures can be recovered with a nanometric precision from multi-angle acquisitions~\cite{boulanger2014fast,dos2016,Zheng2018}. 
Combining  this principle with SMLM techniques lead to a forward model $\Phi$ defined, for all  $i \in  \{1,\ldots , N_1 N_2\}$ and $k \in \{1,\ldots,K\}$, by

 \begin{equation}\label{eq:MATIRF_Fwd}
 	[\phi(x)]_{i,k} \eqdef \frac{\xi(x_3) e^{-s_k x_3}}{{2\pi}\sigma_{1}\sigma_{2}} \int_{\Omega_i}   e^{-\left( \frac{(x_1 - s_1)^2}{2\sigma^2_{1}} +\frac{(x_2 -s_2)^2}{2\sigma^2_{2}}\right)  } \d s_1 \d s_2,  
 \end{equation}
where $ \xi(z)=\pa{\sum_{k=1}^\K e^{-2s_k z}}^{-1/2}.$
This model comes from the combination of a lateral convolution with the axial  TIRF excitation. Here the PSF of the system  is assumed to be a Gaussian with variances $\sigma_1=\sigma_2$, and to be constant along $x_3$ (because only a thin layer of few hundred nanometers is excited by the evanescent wave). The  values  $(s_k)_{k=1}^K$ correspond to the penetration depths associated to the incident angles $(\alpha_k)_{k=1}^K$.

\begin{rem}\label{sec:microscopy-rem:separable-lap}
One particularity of the MA-TIRF modality is that  the kernel $\phi$  in~\eqref{eq:MATIRF_Fwd} is separable. This can be exploited numerically to reduce the overall algorithm complexity.
\end{rem}

\paragraph{Illustrations and numerical computation of $\etaVV$.} 

Examples of noiseless measurements $\obsO=\Phi\measO$  with
\begin{align}\label{sec:microscopy-eq:measO}
	\measO=\dirac{(1.5,2.5,0.1)}+\dirac{(1.5,3,0.5)}+\dirac{(2,5,0.7)}+\dirac{(4.5,3.5,0.4)}+\dirac{(5,1,0.2)}.
\end{align}
are presented in Figure~\ref{fig:NoiselessEx}  for the three modalities. The parameters used for these simulations are provided in Table~\ref{Table:rparametersSImu}. One can observe the effect of the three modalities on molecules at different depths. For the astigmatism modality,  the orientation along which the PSF is defocuced  indicates the position of the molecule with respect to the focal plane (above/below). Moreover, the  larger is this defocucing, the deeper is the molecule. In the case of the double-helix modality, we can clearly see the rotation of the PSF with depth. Finally, for the MA-TIRF modality, we can observe that the recorded intensities for deep molecules decrease, with the incident angle, faster than the intensity for molecules which are close to the glass coverslip (\ie $x_3=0$).

\begin{figure}[t]
		\centering
				\begin{tikzpicture}
		\begin{groupplot}[group style={group size= 4 by 3,                      
    					  horizontal sep=0.3cm, vertical sep=0.3cm},          
						  xmin=0,xmax=6.4,
					   	  ymin=0,ymax=6.4,
						  title style={yshift=-0.10cm},
						  grid=both,
						  xtick={0,2,4,6},
						  xticklabels={0,2,4,6},
						  ytick={0,2,4,6},
						  yticklabels={0,2,4,6},
						  axis equal image,
						  grid style={black},
    					  width=0.35\textwidth]
    	\nextgroupplot[enlargelimits=false,xticklabels={,,},ylabel style={align=center},ylabel={{Astigmatism \\[0.2cm] $x_2$ ($\micron$)}}] 	
    	    \addplot[] graphics[xmin=0,ymin=0,xmax=6.4,ymax=6.4] {microscopy/models/psf_ast_1-up};
    	    \addplot[blue!12.5!red,mark=*,mark size=1] (1.5,2.5); \addplot[blue!62.5!red,mark=*,mark size=1] (1.5,3);
    	    \addplot[blue!87.5!red,mark=*,mark size=1] (2,5); \addplot[blue!50!red,mark=*,mark size=1] (4.5,3.5);
    	    \addplot[blue!25!red,mark=*,mark size=1] (5,1);
    	    \node[white,anchor=west] at (axis cs:0.1,5.9) {{$z_1=0.16\micron$}};
    	\nextgroupplot[enlargelimits=false,yticklabels={,,},xticklabels={,,}] 	
    	    \addplot[] graphics[xmin=0,ymin=0,xmax=6.4,ymax=6.4] {microscopy/models/psf_ast_2-up};
    	    \addplot[blue!12.5!red,mark=*,mark size=1] (1.5,2.5); \addplot[blue!62.5!red,mark=*,mark size=1] (1.5,3);
    	    \addplot[blue!87.5!red,mark=*,mark size=1] (2,5); \addplot[blue!50!red,mark=*,mark size=1] (4.5,3.5);
    	    \addplot[blue!25!red,mark=*,mark size=1] (5,1);
    	    \node[white,anchor=west] at (axis cs:0.1,5.9) {{$z_2=0.32\micron$}};
    	\nextgroupplot[enlargelimits=false,yticklabels={,,},xticklabels={,,}] 			  
    		\addplot[] graphics[xmin=0,ymin=0,xmax=6.4,ymax=6.4] {microscopy/models/psf_ast_3-up};
    	   \addplot[blue!12.5!red,mark=*,mark size=1] (1.5,2.5); \addplot[blue!62.5!red,mark=*,mark size=1] (1.5,3);
    	    \addplot[blue!87.5!red,mark=*,mark size=1] (2,5); \addplot[blue!50!red,mark=*,mark size=1] (4.5,3.5);
    	    \addplot[blue!25!red,mark=*,mark size=1] (5,1);
    		 \node[white,anchor=west] at (axis cs:0.1,5.9) {{$z_3=0.48\micron$}};
    	\nextgroupplot[enlargelimits=false,yticklabels={,,},xticklabels={,,}] 			  
    		\addplot[] graphics[xmin=0,ymin=0,xmax=6.4,ymax=6.4] {microscopy/models/psf_ast_4-up};			  
    		 \addplot[blue!12.5!red,mark=*,mark size=1] (1.5,2.5); \addplot[blue!62.5!red,mark=*,mark size=1] (1.5,3);
    	    \addplot[blue!87.5!red,mark=*,mark size=1] (2,5); \addplot[blue!50!red,mark=*,mark size=1] (4.5,3.5);
    	    \addplot[blue!25!red,mark=*,mark size=1] (5,1);
    	     \node[white,anchor=west] at (axis cs:0.1,5.9) {{$z_4=0.64\micron$}};    		  
    	\nextgroupplot[enlargelimits=false,xticklabels={,,},ylabel style={align=center},ylabel={{Double-helix \\[0.2cm] $x_2$ ($\micron$)}}] 	
    	    \addplot[] graphics[xmin=0,ymin=0,xmax=6.4,ymax=6.4] {microscopy/models/psf_dh_1-up};
    	    \addplot[blue!12.5!red,mark=*,mark size=1] (1.5,2.5); \addplot[blue!62.5!red,mark=*,mark size=1] (1.5,3);
    	    \addplot[blue!87.5!red,mark=*,mark size=1] (2,5); \addplot[blue!50!red,mark=*,mark size=1] (4.5,3.5);
    	    \addplot[blue!25!red,mark=*,mark size=1] (5,1);
    	    \node[white,anchor=west] at (axis cs:0.1,5.9) {{$z_1=0.16\micron$}};
    	\nextgroupplot[enlargelimits=false,yticklabels={,,},xticklabels={,,}] 	
    	    \addplot[] graphics[xmin=0,ymin=0,xmax=6.4,ymax=6.4] {microscopy/models/psf_dh_2-up};
    	    \addplot[blue!12.5!red,mark=*,mark size=1] (1.5,2.5); \addplot[blue!62.5!red,mark=*,mark size=1] (1.5,3);
    	    \addplot[blue!87.5!red,mark=*,mark size=1] (2,5); \addplot[blue!50!red,mark=*,mark size=1] (4.5,3.5);
    	    \addplot[blue!25!red,mark=*,mark size=1] (5,1);
    	     \node[white,anchor=west] at (axis cs:0.1,5.9) {{$z_2=0.32\micron$}};
    	\nextgroupplot[enlargelimits=false,yticklabels={,,},xticklabels={,,}] 			  
    		\addplot[] graphics[xmin=0,ymin=0,xmax=6.4,ymax=6.4] {microscopy/models/psf_dh_3-up};
    	    \addplot[blue!12.5!red,mark=*,mark size=1] (1.5,2.5); \addplot[blue!62.5!red,mark=*,mark size=1] (1.5,3);
    	    \addplot[blue!87.5!red,mark=*,mark size=1] (2,5); \addplot[blue!50!red,mark=*,mark size=1] (4.5,3.5);
    	    \addplot[blue!25!red,mark=*,mark size=1] (5,1);
    		  \node[white,anchor=west] at (axis cs:0.1,5.9) {{$z_3=0.48\micron$}};
    	\nextgroupplot[enlargelimits=false,yticklabels={,,},xticklabels={,,}] 			  
    		\addplot[] graphics[xmin=0,ymin=0,xmax=6.4,ymax=6.4] {microscopy/models/psf_dh_4-up};		
    	    \addplot[blue!12.5!red,mark=*,mark size=1] (1.5,2.5); \addplot[blue!62.5!red,mark=*,mark size=1] (1.5,3);
    	    \addplot[blue!87.5!red,mark=*,mark size=1] (2,5); \addplot[blue!50!red,mark=*,mark size=1] (4.5,3.5);
    	    \addplot[blue!25!red,mark=*,mark size=1] (5,1);	  
    		  \node[white,anchor=west] at (axis cs:0.1,5.9) {{$z_4=0.64\micron$}};   		  
    	\nextgroupplot[enlargelimits=false,ylabel style={align=center},ylabel={{MA-TIRF \\[0.2cm] $x_2$ ($\micron$)}},xlabel=$x_1$ ($\micron$)]     		  
    	    \addplot[] graphics[xmin=0,ymin=0,xmax=6.4,ymax=6.4] {microscopy/models/psf_laplace_1-up};
    	    \addplot[blue!12.5!red,mark=*,mark size=1] (1.5,2.5); \addplot[blue!62.5!red,mark=*,mark size=1] (1.5,3);
    	    \addplot[blue!87.5!red,mark=*,mark size=1] (2,5); \addplot[blue!50!red,mark=*,mark size=1] (4.5,3.5);
    	    \addplot[blue!25!red,mark=*,mark size=1] (5,1);
    	    \node[white,anchor=west] at (axis cs:0.1,5.9) {{$\alpha_1=61.63^o$}};
    	\nextgroupplot[enlargelimits=false,yticklabels={,,},xlabel=$x_1$ ($\micron$)] 	
    	    \addplot[] graphics[xmin=0,ymin=0,xmax=6.4,ymax=6.4] {microscopy/models/psf_laplace_2-up};
    	    \addplot[blue!12.5!red,mark=*,mark size=1] (1.5,2.5); \addplot[blue!62.5!red,mark=*,mark size=1] (1.5,3);
    	    \addplot[blue!87.5!red,mark=*,mark size=1] (2,5); \addplot[blue!50!red,mark=*,mark size=1] (4.5,3.5);
    	    \addplot[blue!25!red,mark=*,mark size=1] (5,1);
    	     \node[white,anchor=west] at (axis cs:0.1,5.9) {{$\alpha_2=67.61^o$}};
    	\nextgroupplot[enlargelimits=false,yticklabels={,,},xlabel=$x_1$ ($\micron$)] 			  
    		  \addplot[] graphics[xmin=0,ymin=0,xmax=6.4,ymax=6.4] {microscopy/models/psf_laplace_3-up};
    		  \addplot[blue!12.5!red,mark=*,mark size=1] (1.5,2.5); \addplot[blue!62.5!red,mark=*,mark size=1] (1.5,3);
    	     \addplot[blue!87.5!red,mark=*,mark size=1] (2,5); \addplot[blue!50!red,mark=*,mark size=1] (4.5,3.5);
    	     \addplot[blue!25!red,mark=*,mark size=1] (5,1);
    		  \node[white,anchor=west] at (axis cs:0.1,5.9) {{$\alpha_3=73.6^o$}};
    	\nextgroupplot[enlargelimits=false,yticklabels={,,},xlabel=$x_1$ ($\micron$)] 			  
    		  \addplot[] graphics[xmin=0,ymin=0,xmax=6.4,ymax=6.4] {microscopy/models/psf_laplace_4-up};
    		  \addplot[blue!12.5!red,mark=*,mark size=1] (1.5,2.5); \addplot[blue!62.5!red,mark=*,mark size=1] (1.5,3);
    	     \addplot[blue!87.5!red,mark=*,mark size=1] (2,5); \addplot[blue!50!red,mark=*,mark size=1] (4.5,3.5);
    	     \addplot[blue!25!red,mark=*,mark size=1] (5,1);
    		  \node[white,anchor=west] at (axis cs:0.1,5.9) {{$\alpha_4=79.58^o$}};
    \end{groupplot}
	\end{tikzpicture}
		\caption{\label{fig:NoiselessEx} Noiseless acquisitions $\obsO$ for the measure $\measO$ given in~\eqref{sec:microscopy-eq:measO} and $K=4$. The parameters used for these simulations  are given in Table~\ref{Table:rparametersSImu}. The color of the molecules represent their depths: 0 (red) -- $0.8\micron$ (blue). }
	\end{figure}
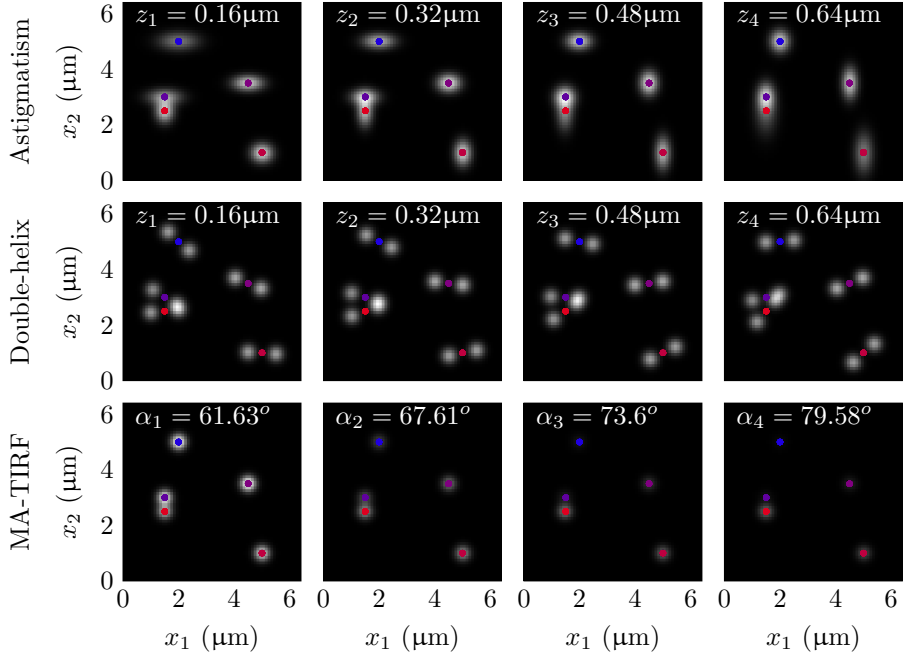
		
	 	
Although, for these three-dimensional models, an explicit expression of $\etaVV$ seems challenging to come by, the latter can be computed numerically for specific points $x \in X$.	A representation of $\etaVV$ for the measure given in~\eqref{sec:microscopy-eq:measO} at $x_3=0.1$ and $x_3=0.5$ is depicted in Figure~\ref{fig:EtaV_microscopy}. For the three modalities, we have that $\etaVV(1.5,2.5,0.1)=\etaVV(1.5,3,0.5)=1$ and otherwise $\etaVV$ is smaller than $1$. Hence, $\etaVV$ seems nondegenerate and a measure composed of the same number of Dirac masses as $\measO$ can be recovered by the SFW algorithm.
	 
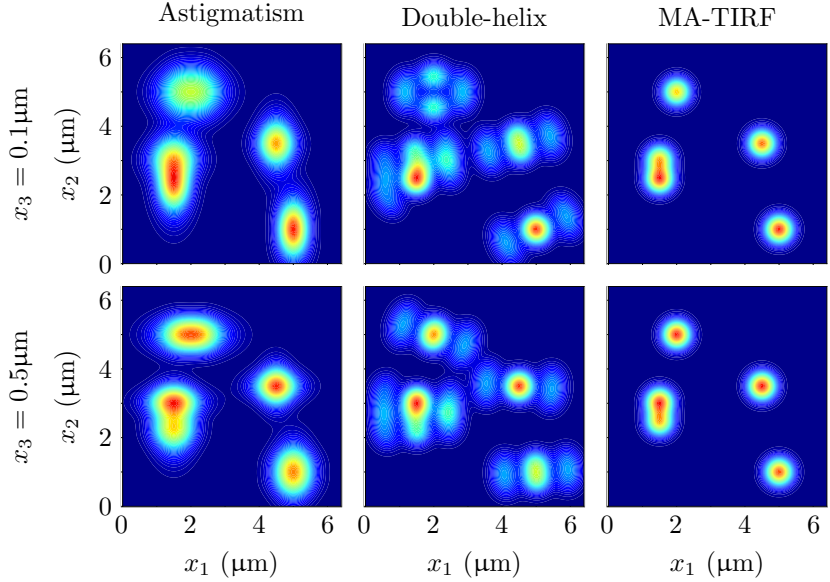
\begin{figure}[t]
		\centering
		\begin{tikzpicture}
		\begin{groupplot}[group style={group size= 3 by 2,                      
    					  horizontal sep=0.3cm, vertical sep=0.3cm},          
						  xmin=0,xmax=64,
					   	  ymin=0,ymax=64,
						  title style={yshift=-0.10cm},
						  grid=both,
						  xtick={0,20,40,60},
						  xticklabels={0,2,4,6},
						  ytick={0,20,40,60},
						  yticklabels={0,2,4,6},
						  axis equal image,
						  grid style={black},
    					  width=0.4\textwidth]
    	\nextgroupplot[enlargelimits=false,xticklabels={,,},ylabel style={align=center},ylabel={{$x_3=0.1\micron$ \\[0.2cm] $x_2$ ($\micron$)}},title=Astigmatism] 	
    	    \addplot[] graphics[xmin=-11,ymin=-9,xmax=80,ymax=74] {microscopy/models/etaV_ast_1};
    	\nextgroupplot[enlargelimits=false,yticklabels={,,},xticklabels={,,},title=Double-helix] 	
    	   \addplot[] graphics[xmin=-11,ymin=-9,xmax=80,ymax=74]  {microscopy/models/etaV_dh_1};
    	\nextgroupplot[enlargelimits=false,yticklabels={,,},xticklabels={,,},title=MA-TIRF] 			  
    		\addplot[] graphics[xmin=-11,ymin=-9,xmax=80,ymax=74]  {microscopy/models/etaV_laplace_1};    		  
    	\nextgroupplot[enlargelimits=false,ylabel style={align=center},ylabel={{$x_3=0.5\micron$ \\[0.2cm] $x_2$ ($\micron$)}},xlabel=$x_1$ ($\micron$)] 	
    	    \addplot[] graphics[xmin=-11,ymin=-9,xmax=80,ymax=74]  {microscopy/models/etaV_ast_2};
    	\nextgroupplot[enlargelimits=false,yticklabels={,,},xlabel=$x_1$ ($\micron$)] 	
    	    \addplot[] graphics[xmin=-11,ymin=-9,xmax=80,ymax=74]  {microscopy/models/etaV_dh_2};
    	\nextgroupplot[enlargelimits=false,yticklabels={,,},xlabel=$x_1$ ($\micron$)] 			  
    		\addplot[] graphics[xmin=-11,ymin=-9,xmax=80,ymax=74]  {microscopy/models/etaV_laplace_2};
    \end{groupplot}
	\end{tikzpicture}
		\caption{\label{fig:EtaV_microscopy} Numerical computation of $\etaVV$ at $x_3=0.1$ (top) and $x_3=0.5$ (bottom) for the three models and the measure $\measO$ given in~\eqref{sec:microscopy-eq:measO}. The colormap ranges from 0 (blue) to 1 (red).}
	\end{figure}

\subsection{Simulation setting}\label{subsec:simul-setting}
 
\paragraph{Imaged Structure.} 

Simulations were performed using the microtubules-like structure depicted in Figure~\ref{fig:sec-models-tubs}. It has been generated within the volume
 \begin{equation}
 	\Pos=[0,b_1]\times[0,b_2]\times[0,b_3] \subset\RR^3 \qwhereq b_1=b_2=6.4~\micron \text{ and } b_3=0.8~\micron.
\end{equation}
The filaments were obtained by randomly sampling many points along four curves defined by polynomial equations. To ensure a uniform distribution of the points along the curves, we first parametrized each curve by a piecewise linear function (with very small steps). Then, in order to give a width to the filaments, each point $x  \in \Pos$ randomly chosen on one of the curves is replaced by a point randomly chosen in a ball  centered at $x$ with  radius  $10$ nm. Thus, simulated filaments have a diameter of $20$ nm.

\begin{figure}[t]
\centering
\includegraphics[trim=3cm 3cm 3cm 3cm,width=0.5\linewidth,clip=true]{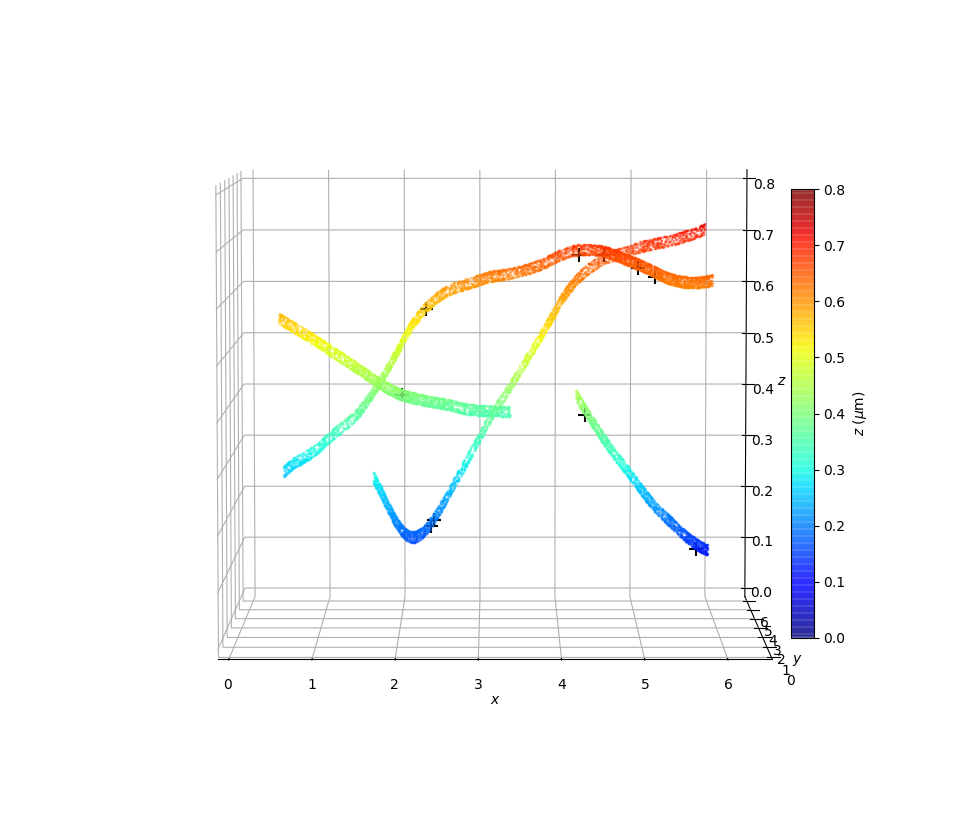}
\caption{Microtubules structure used for the simulations. The diameter of the filaments is $20$ nm. The color encodes the depth of molecules within the range $0-0.8~\micron$. Black crosses represent a subset of activated molecules (\ie a measure  $\measO$). }\label{fig:sec-models-tubs}
\end{figure}

\paragraph{Simulation of noiseless acquisitions.} 

The $\nMol_{\text{tot}}\in\NN^*$ molecules of the simulated  structure are divided into $\nGT\in\NN^*$ sparse set of  $\nMol\in\NN^*$ molecules using a random permutation (\ie $\nMol_{\text{tot}}=\nGT \times \nMol$). This models the sequential stochastic activation of fluorophores used in SMLM. For each of the $\nGT$ subsets of  molecules, we define a Radon measure composed of a sum of Dirac masses---located at the position of the molecules---with positive amplitudes
\eq{ \measO=\sum_{i=1}^\nMol \ampOi\dirac{\posOi} \qwhereq \ampOi>0 \qandq \posOi \in\Pos.}
The amplitudes are randomly generated within $[1,1.5]$.
An example of a set of activated molecules is shown in Figure~\ref{fig:sec-models-tubs} (black crosses). Now let ($N_1  \times N_2$) be the size of the grid of pixels on the detector plane, and $K$ be the number of focal planes (or the number of TIRF ``angles'', see Section~\ref{sec:FwdModels}) which are recorded. Then, the noiseless measurements $\obsO$  for an activated measure $\measO$ follow the model
\begin{equation}
\obsO = \Phi \measO,
\end{equation}
where $\Phi$ is defined in~\eqref{eq:Fwd}.

   Finally, it is noteworthy that in practice the number of activated molecules varies from one activation to another around an average value (which depends on the power of the excitation laser beam). However, fixing this number to $\nMol$ for each activated set of molecules allows us to better control the density of spikes in order to study the behaviour of the algorithm when the latter increases.

\paragraph{Noise Model.} 
 
There are two predominant sources of noise  in microscopy data. 
\begin{itemize}
 \item The shot noise which  is inherent to the quantum nature of  light (random emissions of photons). It is well modeled by a Poisson distribution whose intensity is the number of photon collected at each pixel. Given the noiseless acquisition $\obsO$, we normalize it such that 
\begin{equation}\label{eq:photNoise}
\max_{i \in \{1,\ldots,N_1N_2\}} \left(  \sum_{k=1}^{K} [\obsO]_{i,k} \right)  = \nPhoton,
\end{equation}
where $\nPhoton >0 $ denotes the maximal photon budget per pixel and controls the noise level. Then, each entry of  $\obsO$ is replaced by a realization of a Poisson distribution $\Pp$ with parameter $[\obsO]_{i,k}$.  It is noteworthy from \eqref{eq:photNoise} that the level of noise not only increases as $\nPhoton$ decreases, but it also increases with $K$. 
\item The readout noise $w_G$ of the camera. It is usually modeled by a Gaussian distribution with variance~$\sigma^2$.
\end{itemize}
Finally the noisy data are given by
\begin{equation}
   \obsw = \Pp(\obsO)+w_G.
\end{equation}

\subsection{Results}
 
For each of the three modalities presented in Section \ref{sec:FwdModels} (Double-Helix, Astigmatism, MA-TIRF), acquisitions where simulated using the optical parameters gathered in Table~\ref{Table:rparametersSImu}. These parameters have been tuned according to the experimental PSF used in the SMLM challenge \cite{Sage362517}.  Finally, we generated different experiments by varying the density of molecules $ \nMol \in \{ 5,10,15\}$ as well as the number of focal planes (or angles for the TIRF model) $K \in \{1,2,3,4\}$.


\begin{table}[t]
		\centering
	   \caption{\label{Table:rparametersSImu} Parameters used for data simulation. }
		\begin{tabular}{l|ccl} 
			\toprule
			\toprule
			  &   Parameter  &  Value  &   Description   \\  \midrule
			\multirow{9}{*}{\rotatebox{90}{All modalities}}  
						&$b_1=b_2$ & $6.4\micron$ & Region of interest \\
			& $b_3$  & $0.8\micron$ & Maximal depth of molecules \\
			& $N_1 = N_2$  & $64$ & Detector grid size \\		
			&$\NA$ & $1.49$ & Objective numerical aperture \\
			&$\nii$ & $1.515$ & Refractive index incident medium \\
			&$\nt$ & $1.333$ & Refractive index transmitted medium \\
			&$\lamb$ & $0.66\micron$ & Excitation wavelength \\
			&$\nPhoton$ & $1000$ & Photon budget\\
			&$\sigma$ & $10^{-4}$ & Variance of Gaussian noise \\
			 \midrule
			\multirow{5}{*}{\rotatebox{90}{Astigmatism}} &$\sigma_0$ & ${0.42\lamb}/{\NA}$ & PSF variance at focus\\
			&$\beta$ & $0.2\micron$ & Depth for which the variance is minimal \\
			&$d$ & ${\lamb\nii}/({2\NA^2})$ & Parameter related to the depth-of-field\\
			&$\alpha$ &$-0.79$ & Scaling constant \\
			& $(z_k)_{k=1}^K$ & $k b_3 /(K+1)$ & Focal planes \\
			 \midrule
			\multirow{4}{*}{\rotatebox{90}{Double-Helix}} & $\sigma_1=\sigma_2$ & ${0.42\lamb}/{\NA}$ & PSF variance \\
			& $\om$ & $1\micron$ & Distance between the two PSF lobes\\
			& $\theta_{\mathrm{speed}}$ & $ 0.3846 \pi$ rad/$\micron$ & Rotation speed of the PSF\\
			& $(z_k)_{k=1}^K$ & $k b_3 /(K+1)$ & Focal planes \\
			 \midrule
			\multirow{3}{*}{\rotatebox{90}{MA-TIRF}} & $\sigma_1=\sigma_2$ & ${0.42\lamb}/{\NA}$ & PSF variance \\
			& $(\alpha_k)_{k=1}^K$ & $\alc + \frac{\alm-\alc}{K-1}(k-1)$  & Incident angles  \\
			& $\alm$  & $\sin^{-1}(\NA / \nii)$  & Maximal incident angle \\
			\bottomrule
			\bottomrule
		\end{tabular}
	\end{table} 
	
\subsubsection{Metrics for evaluation}

   In order to assess the quality of the reconstructed volumes, we consider standard metrics which reflect both the detection rate and the localization error \cite{Sage362517,sage-quantitative2015}. Given a recovered frame and a tolerance radius $r >0$, we pair estimated molecules and ground truth (GT) molecules when the distance between them is lower than~$r$. Paired estimated molecules are then referred as true positive (TP) while unpaired ones as false positive (FP). Finally, the unpaired GT molecules are identified as false negative (FN). These quantities being determined for each frame, we can compute the Jaccard index (Jac), the Recall (Rec) and the Precision (Pre) metrics,
   \begin{equation}
      \mathrm{Jac} = \frac{\#\mathrm{TP}}{\#\mathrm{TP}+ \#\mathrm{FP}+\# \mathrm{FN}} \quad \mathrm{Rec}=\frac{\#\mathrm{TP}}{\#\mathrm{TP}+\# \mathrm{FN}} \quad \mathrm{Pre}= \frac{\#\mathrm{TP}}{\#\mathrm{TP}+ \#\mathrm{FP}}.
   \end{equation}
The Jaccard index measures the overall performance of detection by giving a measure of similarity between the two sets of points. The Recall and Precision metrics can then be used to measure the ability of an algorithm to minimize FN  and FP detection, respectively. Finally, the TP molecules are used to compute the root mean squared error (RMSE) along each dimension 
\begin{equation}
	\mathrm{RMSE}_{x_1} = \sqrt{\frac{1}{\# \mathrm{TP}} \sum_{i \in \mathrm{TP}} ([x_i]_1 - [x_{0,i}]_1)^2},
\end{equation}
and similarly for $\mathrm{RMSE}_{x_2}$ and $\mathrm{RMSE}_{x_3}$. Note that, by construction, the RMSE is bounded by the radius $r$. Hence, in the following, we use different values for $r$ depending on the metric of interest.

\subsubsection{Choice of the regularization parameter $\lambda$}

For each experiment (\ie  $\nMol\in \{5,10,15\}$ and   $\K \in \{1,2,3,4\}$), we choose the value of the regularization parameter $\lambda$ which maximizes the Jaccard index for a radius of $r=0.02$ (\ie $20$ nm). This training step was performed over a small subset of initial measures $\measO$ (\ie frames). Then the recovery was done on the complete dataset using the optimal $\lambda$ found. 


\subsubsection{Discussion}

The evolution of Jaccard,  Recall, and  Precision metrics with respect to  $K$ are depicted in Figure~\ref{sec:microscopy-fig:indices}. As expected, they all increase with  $K$. However, although the improvement is significant from $K=1$ to $K=2$, higher values  only provide marginal gains. This can be explained by the fact that the photon budget $\nPhoton$ is distributed over the  $K$ acquisitions (see equation \eqref{eq:photNoise}). Hence, the additional axial information brought by increasing the number of acquisitions per activation should be balanced by the higher noise corrupting  the data. Another observation from these plots concerns the degradation of the performance as the density (\ie the number of molecules $\nMol$) increases.



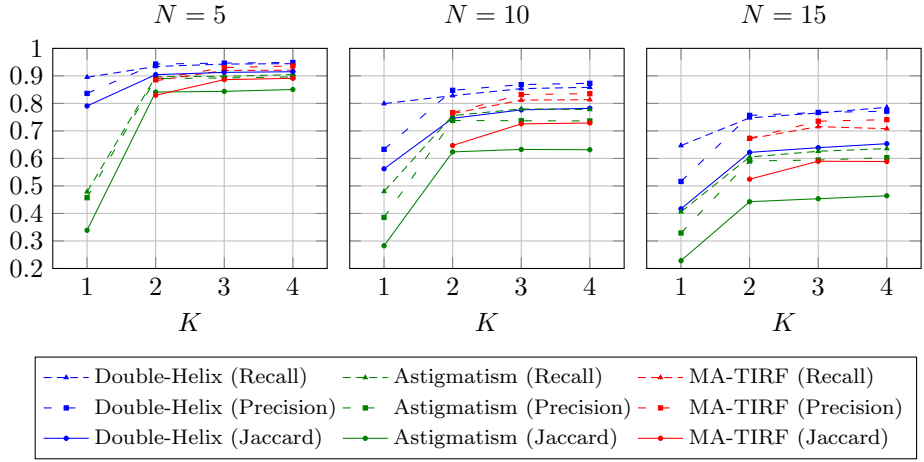
\begin{figure}[t]
		\centering
				\begin{tikzpicture}
		\begin{groupplot}[group style={group size= 3 by 1,                      
    					  horizontal sep=0.3cm, vertical sep=0.3cm},          
						  xmin=0.5,xmax=4.5,
					   	  ymin=0.2,ymax=1,
						  grid=both,
						  xtick={1,2,3,4},
						  ytick={0.2,0.3,0.4,0.5,0.6,0.7,0.8,0.9,1},
						  xlabel={$K$},
						  legend columns=3,
					      legend style={legend cell align=left,at={(3.15,-0.4)},font=\footnotesize},
    					  width=0.4\textwidth]
    	\nextgroupplot[title={$N=5$}] 	  
    	    	  \addplot[blue,densely dashed,mark=triangle*,every mark/.append style={solid},mark size=0.9] table{figures/microscopy/results/DHRec_N5.dat};	
    	    	   \addplot[darkgreen,densely dashed,mark=triangle*,every mark/.append style={solid},mark size=0.9] table{figures/microscopy/results/AstRec_N5.dat};	
    	    	   \addplot[red,densely dashed,mark=triangle*,every mark/.append style={solid},mark size=0.9] table{figures/microscopy/results/LapRec_N5.dat};	
    	    	  \addplot[blue,loosely dashed,mark=square*,every mark/.append style={solid},mark size=0.8] table{figures/microscopy/results/DHPre_N5.dat};	
    	    	  \addplot[darkgreen,loosely dashed,mark=square*,every mark/.append style={solid},mark size=0.8] table{figures/microscopy/results/AstPre_N5.dat};		
    	    	  \addplot[red,loosely dashed,mark=square*,every mark/.append style={solid},mark size=0.8] table{figures/microscopy/results/LapPre_N5.dat};		 	
    	    	   \addplot[blue,mark=*,mark size=0.8] table{figures/microscopy/results/DHJac_N5.dat};
    	    	  \addplot[darkgreen,mark=*,mark size=0.8] table{figures/microscopy/results/AstJac_N5.dat};		  
    	    	   \addplot[red,mark=*,mark size=0.8] table{figures/microscopy/results/LapJac_N5.dat};		 		
    	    	  \legend{Double-Helix (Recall),Astigmatism (Recall),MA-TIRF (Recall),Double-Helix (Precision),Astigmatism (Precision),MA-TIRF (Precision),Double-Helix (Jaccard),Astigmatism (Jaccard),MA-TIRF (Jaccard)}  
     	\nextgroupplot[title={$N=10$},yticklabels={,,}] 	
     	    	  \addplot[blue,mark=*,mark size=0.8] table{figures/microscopy/results/DHJac_N10.dat};		  
    	    	  \addplot[blue,densely dashed,mark=triangle*,every mark/.append style={solid},mark size=0.9] table{figures/microscopy/results/DHRec_N10.dat};	
    	    	  \addplot[blue,loosely dashed,mark=square*,every mark/.append style={solid},mark size=0.8] table{figures/microscopy/results/DHPre_N10.dat};	
    	    	  \addplot[darkgreen,mark=*,mark size=0.8] table{figures/microscopy/results/AstJac_N10.dat};		  
    	    	  \addplot[darkgreen,densely dashed,mark=triangle*,every mark/.append style={solid},mark size=0.9] table{figures/microscopy/results/AstRec_N10.dat};	
    	    	  \addplot[darkgreen,loosely dashed,mark=square*,every mark/.append style={solid},mark size=0.8] table{figures/microscopy/results/AstPre_N10.dat};		    	    	  
    	    	  \addplot[red,mark=*,mark size=0.8] table{figures/microscopy/results/LapJac_N10.dat};		  
    	    	  \addplot[red,densely dashed,mark=triangle*,every mark/.append style={solid},mark size=0.9] table{figures/microscopy/results/LapRec_N10.dat};	
    	    	  \addplot[red,loosely dashed,mark=square*,every mark/.append style={solid},mark size=0.8] table{figures/microscopy/results/LapPre_N10.dat};		 	 
      	\nextgroupplot[title={$N=15$},yticklabels={,,}] 	
      	     	  \addplot[blue,mark=*,mark size=0.8] table{figures/microscopy/results/DHJac_N15.dat};		  
    	    	  \addplot[blue,densely dashed,mark=triangle*,every mark/.append style={solid},mark size=0.9] table{figures/microscopy/results/DHRec_N15.dat};	
    	    	  \addplot[blue,loosely dashed,mark=square*,every mark/.append style={solid},mark size=0.8] table{figures/microscopy/results/DHPre_N15.dat};	
    	    	  \addplot[darkgreen,mark=*,mark size=0.8] table{figures/microscopy/results/AstJac_N15.dat};		  
    	    	  \addplot[darkgreen,densely dashed,mark=triangle*,every mark/.append style={solid},mark size=0.9] table{figures/microscopy/results/AstRec_N15.dat};	
    	    	  \addplot[darkgreen,loosely dashed,mark=square*,every mark/.append style={solid},mark size=0.8] table{figures/microscopy/results/AstPre_N15.dat};		    	    	  
    	    	  \addplot[red,mark=*,mark size=0.8] table{figures/microscopy/results/LapJac_N15.dat};		  
    	    	  \addplot[red,densely dashed,mark=triangle*,every mark/.append style={solid},mark size=0.9] table{figures/microscopy/results/LapRec_N15.dat};	
    	    	  \addplot[red,loosely dashed,mark=square*,every mark/.append style={solid},mark size=0.8] table{figures/microscopy/results/LapPre_N15.dat};		 	 
    				\end{groupplot}
	\end{tikzpicture}
		\caption{\label{sec:microscopy-fig:indices} Evolution of Jaccard, Recall and Precision metrics with respect to $K$, for a radius of detection $r=0.02$ ($20$nm).}
	\end{figure}
	
   These results also bring useful information in order to improve existing systems. Let us recall that current commercial systems includes Astigmatism and Double-Helix modalities with one focal plane (\ie $K=1$). Hence, it can be inferred from our simulations that recording an image at two focal planes for each activation of molecules would not only  improve significantly the reconstruction quality  but make the reconstructions more robust when the density of molecules  increases. These observations corroborate the study in \cite{huang-3dastmultifp} where the authors use a multi-focus astigmatism system. However, to preserve a reasonable temporal resolution, multi-focal acquisitions require to synchronize several cameras \cite{huang-3dastmultifp} which can be expensive and lead to delicate calibration procedures (\textit{e.g.} alignment and PSF aberrations for each camera). In that respect, the proposed combination of SMLM with MA-TIRF offers an interesting alternative to improve existing systems. First, it has the potential to provide reconstructions whose quality compares favorably with the Double-Helix model while improving over the Astigmatism modality.  Second, it only requires the use of galvanometric mirrors  to control the incident angle \cite{boulanger2014fast}. It is noteworthy that commercial SMLM systems generally use a single TIRF illumination to limit the illumination depth. Finally, as for the multi-focus strategy, MA-TIRF requires some calibrations (\textit{e.g.}  incident angles) for which there exist dedicated procedures \cite{boulanger2014fast,Soubies2016}.
   
\begin{rem}
Although the PSFs used for these simulations have been adjusted using experimental PSFs, they remain idealistic. This is particularly the case for the Double-Helix which in practice deviates from two Gaussian lobes that coil around each other along $z$ \cite{Sage362517}. In contrast, the Gaussian model yields a precise approximation of the MA-TIRF (\ie widefield) PSF \cite{Zhang2007}. The main simplification for the latter lies in the fact that each molecule is activated only during one set of multi-angle acquisitions. This would not be the case with a real implementation of the system and the model should be improved by considering the temporal aspect of the acquisition. However, the present study constitutes a first proof-of-concept and future developments will consider a more sophisticated model.
\end{rem}


\begin{figure}[!h]
		\centering
				\begin{tikzpicture}
		\begin{groupplot}[group style={group size= 3 by 1,                      
    					  horizontal sep=0.3cm, vertical sep=0.3cm},          
						  xmin=0.5,xmax=4.5,
					   	  ymin=0,ymax=35,
						  grid=both,
						  xtick={1,2,3,4},						  
						  ytick={5,10,15,20,25,30,35},
						  xlabel={$K$},
						  legend columns=3,
					      legend style={legend cell align=left,at={(3.15,-0.4)},font=\footnotesize},
    					  width=0.4\textwidth]
    	\nextgroupplot[title={$N=5$}] 			  
    	    	  \addplot[blue,densely dashed,mark=triangle*,every mark/.append style={solid},mark size=0.9] table{figures/microscopy/results/DH_RMSEx_N5.dat};	
    	    	   \addplot[darkgreen,densely dashed,mark=triangle*,every mark/.append style={solid},mark size=0.9] table{figures/microscopy/results/Ast_RMSEx_N5.dat};	
    	    	   \addplot[red,densely dashed,mark=triangle*,every mark/.append style={solid},mark size=0.9] table{figures/microscopy/results/Lap_RMSEx_N5.dat};	
    	    	  \addplot[blue,loosely dashed,mark=square*,every mark/.append style={solid},mark size=0.8] table{figures/microscopy/results/DH_RMSEy_N5.dat};	
    	    	  \addplot[darkgreen,loosely dashed,mark=square*,every mark/.append style={solid},mark size=0.8] table{figures/microscopy/results/Ast_RMSEy_N5.dat};		
    	    	  \addplot[red,loosely dashed,mark=square*,every mark/.append style={solid},mark size=0.8] table{figures/microscopy/results/Lap_RMSEy_N5.dat};	
    	    	   \addplot[blue,mark=*,mark size=0.8] table{figures/microscopy/results/DH_RMSEz_N5.dat};
    	    	  \addplot[darkgreen,mark=*,mark size=0.8] table{figures/microscopy/results/Ast_RMSEz_N5.dat};		  
    	    	   \addplot[red,mark=*,mark size=0.8] table{figures/microscopy/results/Lap_RMSEz_N5.dat};		 	 	
    	    	  \legend{Double-Helix ($\mathrm{RMSE}_{x_1}$),Astigmatism ($\mathrm{RMSE}_{x_1}$),MA-TIRF ($\mathrm{RMSE}_{x_1}$),Double-Helix ($\mathrm{RMSE}_{x_2}$),Astigmatism ($\mathrm{RMSE}_{x_2}$),MA-TIRF ($\mathrm{RMSE}_{x_2}$),Double-Helix ($\mathrm{RMSE}_{x_3}$),Astigmatism ($\mathrm{RMSE}_{x_3}$),MA-TIRF ($\mathrm{RMSE}_{x_3}$)}  
     	\nextgroupplot[title={$N=10$},yticklabels={,,}] 	
     	    	  \addplot[blue,mark=*,mark size=0.8] table{figures/microscopy/results/DH_RMSEz_N10.dat};		  
    	    	  \addplot[blue,densely dashed,mark=triangle*,every mark/.append style={solid},mark size=0.9] table{figures/microscopy/results/DH_RMSEx_N10.dat};	
    	    	  \addplot[blue,loosely dashed,mark=square*,every mark/.append style={solid},mark size=0.8] table{figures/microscopy/results/DH_RMSEy_N10.dat};	
    	    	  \addplot[darkgreen,mark=*,mark size=0.8] table{figures/microscopy/results/Ast_RMSEz_N10.dat};		  
    	    	  \addplot[darkgreen,densely dashed,mark=triangle*,every mark/.append style={solid},mark size=0.9] table{figures/microscopy/results/Ast_RMSEx_N10.dat};	
    	    	  \addplot[darkgreen,loosely dashed,mark=square*,every mark/.append style={solid},mark size=0.8] table{figures/microscopy/results/Ast_RMSEy_N10.dat};		    	    	  
    	    	  \addplot[red,mark=*,mark size=0.8] table{figures/microscopy/results/Lap_RMSEz_N10.dat};		  
    	    	  \addplot[red,densely dashed,mark=triangle*,every mark/.append style={solid},mark size=0.9] table{figures/microscopy/results/Lap_RMSEx_N10.dat};	
    	    	  \addplot[red,loosely dashed,mark=square*,every mark/.append style={solid},mark size=0.8] table{figures/microscopy/results/Lap_RMSEy_N10.dat};		 	 
      	\nextgroupplot[title={$N=15$},yticklabels={,,}] 	
      	     	  \addplot[blue,mark=*,mark size=0.8] table{figures/microscopy/results/DH_RMSEz_N15.dat};		  
    	    	  \addplot[blue,densely dashed,mark=triangle*,every mark/.append style={solid},mark size=0.9] table{figures/microscopy/results/DH_RMSEx_N15.dat};	
    	    	  \addplot[blue,loosely dashed,mark=square*,every mark/.append style={solid},mark size=0.8] table{figures/microscopy/results/DH_RMSEy_N15.dat};	
    	    	  \addplot[darkgreen,mark=*,mark size=0.8] table{figures/microscopy/results/Ast_RMSEz_N15.dat};		  
    	    	  \addplot[darkgreen,densely dashed,mark=triangle*,every mark/.append style={solid},mark size=0.9] table{figures/microscopy/results/Ast_RMSEx_N15.dat};	
    	    	  \addplot[darkgreen,loosely dashed,mark=square*,every mark/.append style={solid},mark size=0.8] table{figures/microscopy/results/Ast_RMSEy_N15.dat};		    	    	  
    	    	  \addplot[red,mark=*,mark size=0.8] table{figures/microscopy/results/Lap_RMSEz_N15.dat};		  
    	    	  \addplot[red,densely dashed,mark=triangle*,every mark/.append style={solid},mark size=0.9] table{figures/microscopy/results/Lap_RMSEx_N15.dat};	
    	    	  \addplot[red,loosely dashed,mark=square*,every mark/.append style={solid},mark size=0.8] table{figures/microscopy/results/Lap_RMSEy_N15.dat};		 	 
    				\end{groupplot}
	\end{tikzpicture}
		\caption{\label{sec:microscopy-fig:rmse} Evolution of the RMSE (nm) with respect to  $K$, for a radius of detection $r=0.1$ ($100$nm).}
	\end{figure}
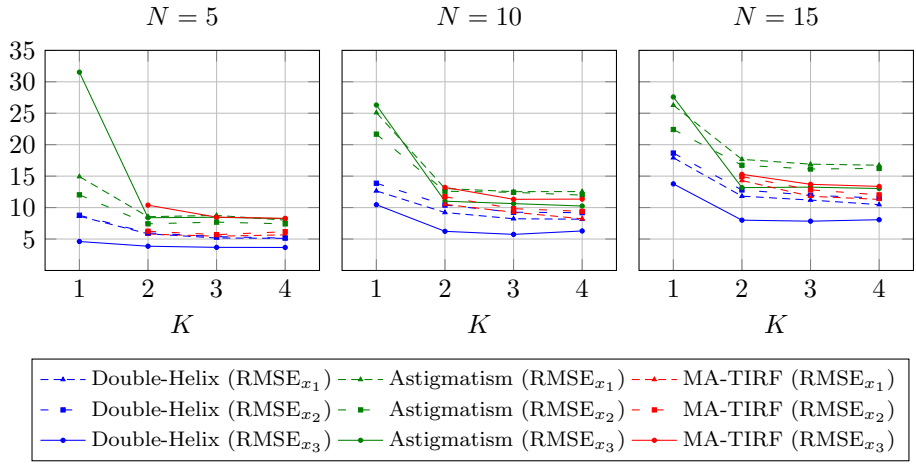

The results in terms of RMSE presented in Figure~\ref{sec:microscopy-fig:rmse} lead to similar interpretations. First, the detection accuracy is increasing with  $K$ while decreasing with $\nMol$. Second, we can observe that the differences between the Double-Helix and the MA-TIRF models mainly come from the precision in $x_3$. Indeed, they both lead to the same lateral RMSE (around $5$nm when $\nMol=5$ and $12$nm at the highest density $\nMol=15$), but the Double-Helix enjoys a better axial RMSE. This reflects the challenging problem that constitues the inversion  of the Laplace transform, which is related to the MA-TIRF model.  Nevertheless, the SWF algorithm performs quite well at this task (see also Figures~\ref{sec:microscopy-fig:tub-Kfixe} and~\ref{sec:microscopy-fig:tub-Nfixe}). Another observation concerns the fact that the Double-Helix can reach a better axial than lateral RMSE. This fact, which was also observed in the recent SMLM challenge \cite{Sage362517}, can be explained by the large lateral support of the Double-Helix PSF as well as its good axial discrimination.


Finally, three-dimensional representations of the recovered structures are presented in Figures~\ref{sec:microscopy-fig:tub-Kfixe} and~\ref{sec:microscopy-fig:tub-Nfixe} for a fixed  $K=4$ and $\nMol=10$, respectively. These figures complete and illustrate the observations made with the computed metrics.


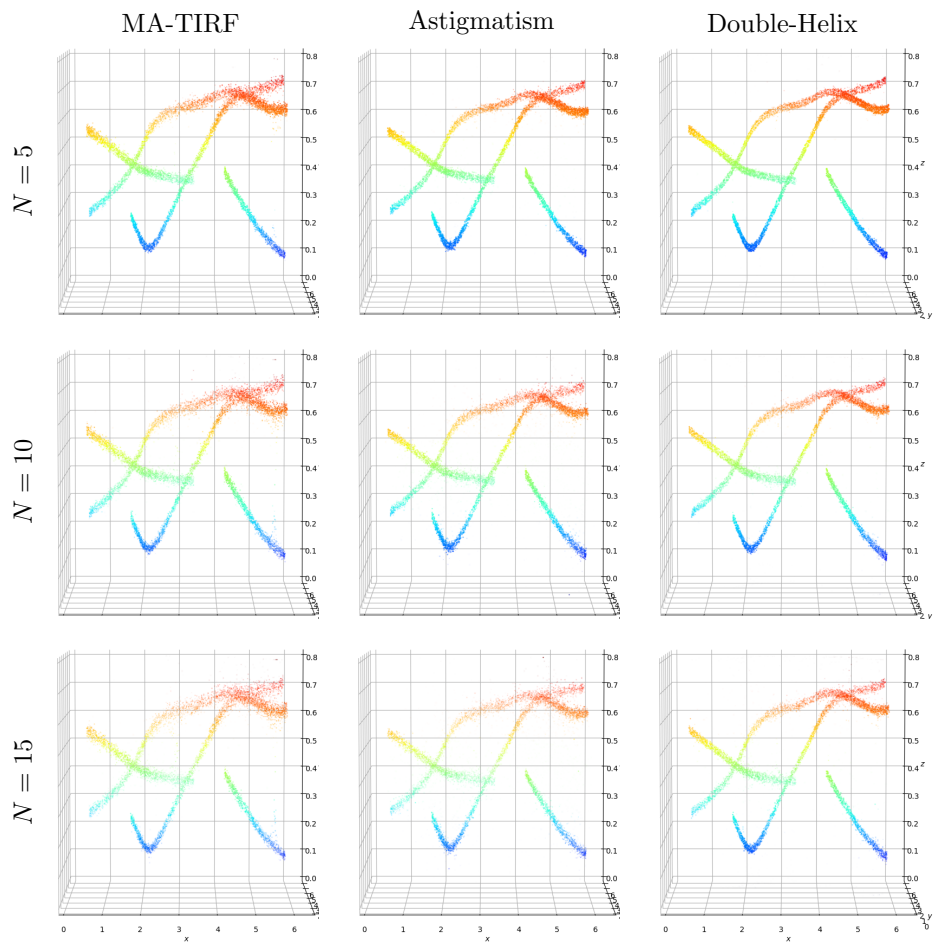
\begin{figure}
\centering
\begin{tikzpicture}
\node at (0,0) {\includegraphics[width=0.33\linewidth]{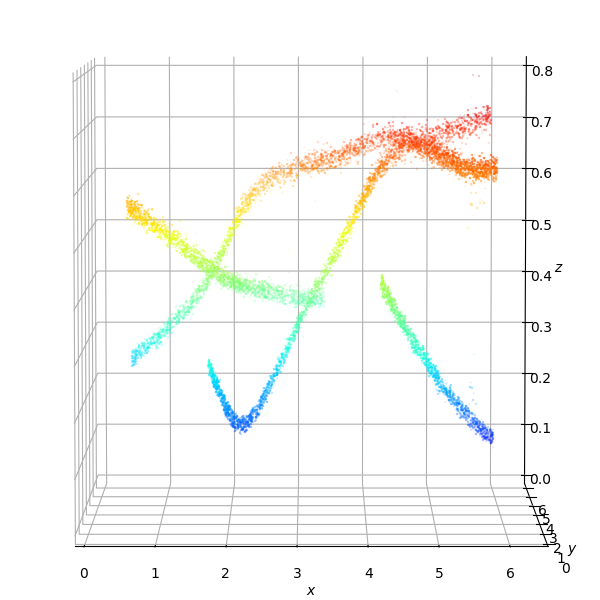}};
\node at (4,0) {\includegraphics[width=0.33\linewidth]{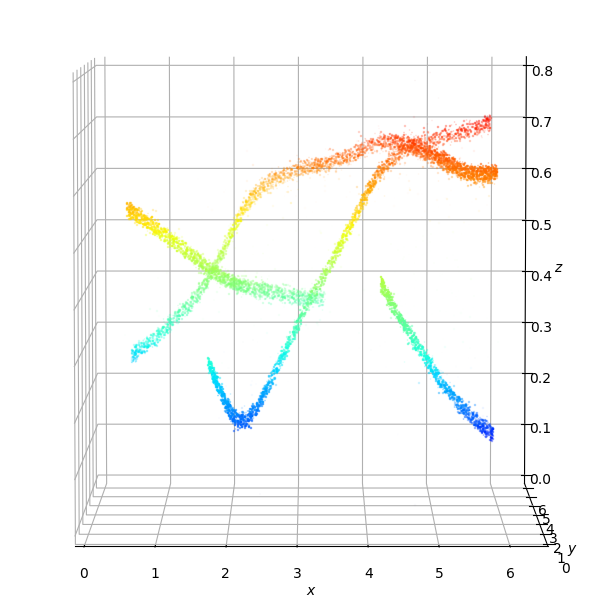}};
\node at (8,0) {\includegraphics[width=0.33\linewidth]{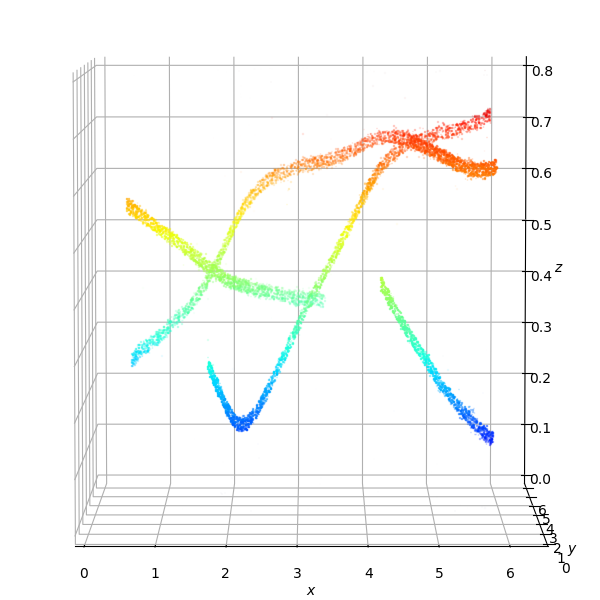}};
\node at (0,-4) {\includegraphics[width=0.33\linewidth]{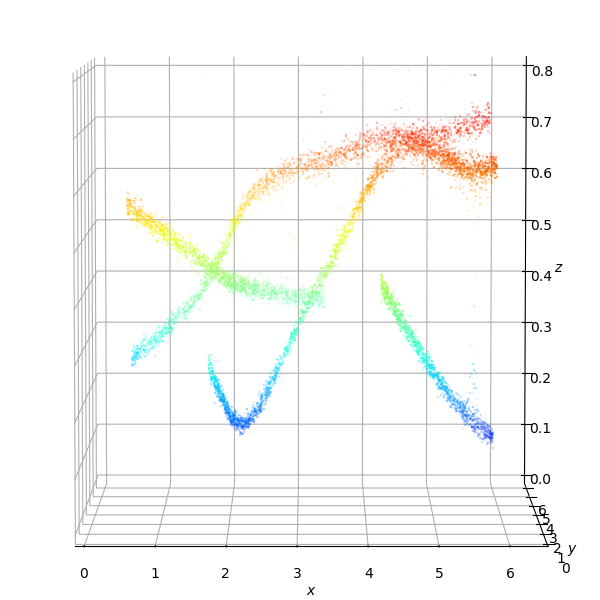}};
\node at (4,-4) {\includegraphics[width=0.33\linewidth]{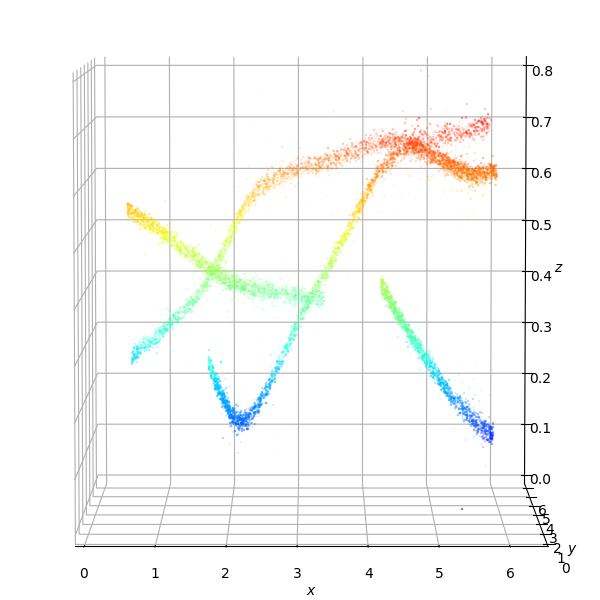}};
\node at (8,-4) {\includegraphics[width=0.33\linewidth]{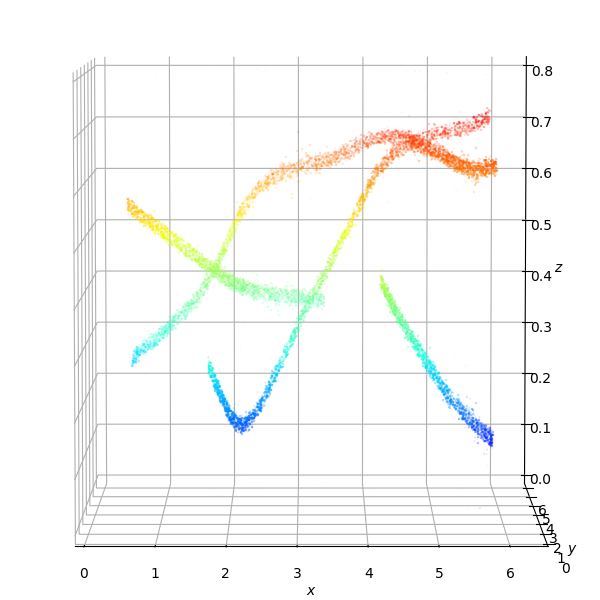}};
\node at (0,-8) {\includegraphics[width=0.33\linewidth]{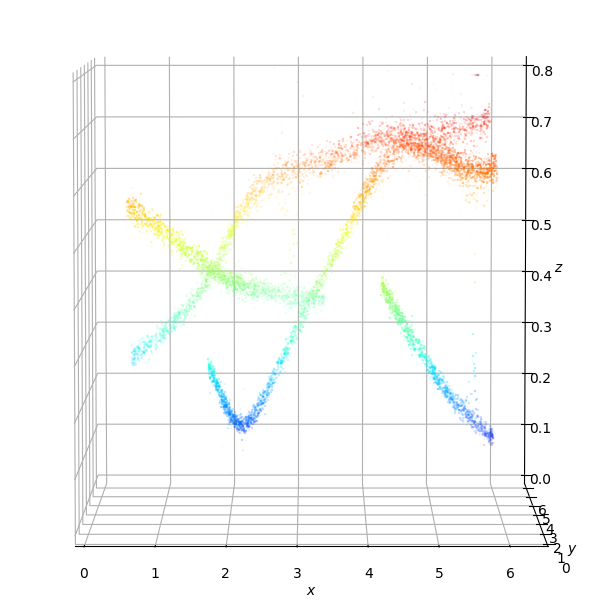}};
\node at (4,-8) {\includegraphics[width=0.33\linewidth]{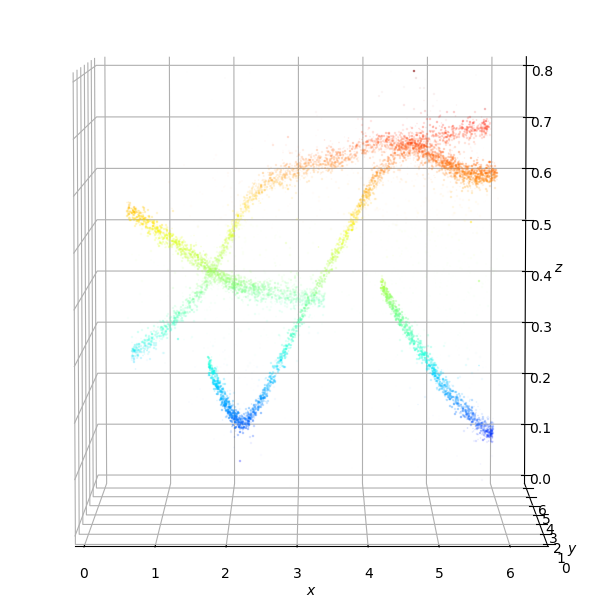}};
\node at (8,-8) {\includegraphics[width=0.33\linewidth]{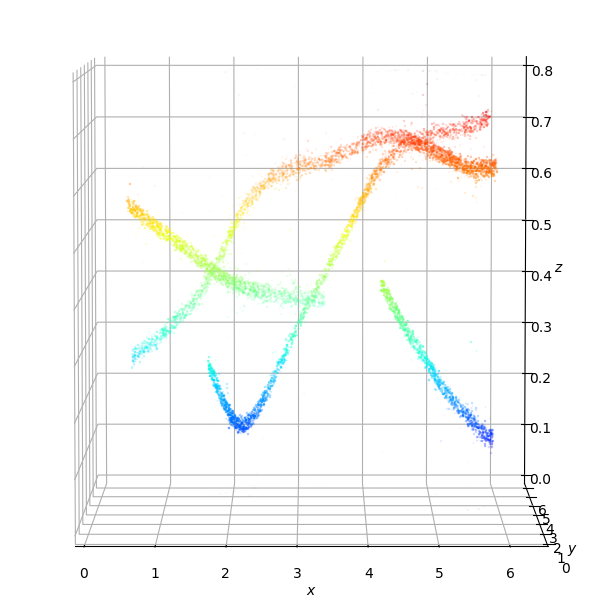}};
\node at (0,2.1) {MA-TIRF};\node at (4,2.1) {Astigmatism};\node at (8,2.1) {Double-Helix};
\node[rotate=90] at (-2.1,0) {$\nMol=5$};\node[rotate=90] at (-2.1,-4) {$\nMol=10$};\node[rotate=90] at (-2.1,-8) {$\nMol=15$};
\end{tikzpicture}
\caption{\label{sec:microscopy-fig:tub-Kfixe} Recovered structures for $\K=4$.}
\end{figure}


\begin{figure}
\centering
\begin{tikzpicture}
\node at (4,0) {\includegraphics[width=0.33\linewidth]{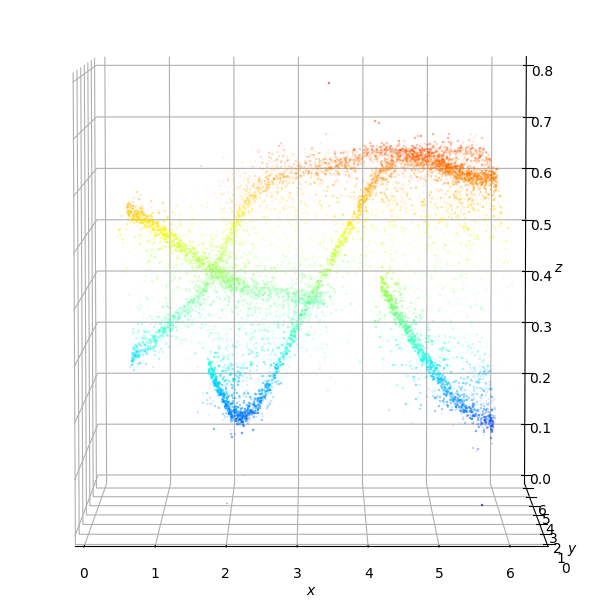}};
\node at (8,0) {\includegraphics[width=0.33\linewidth]{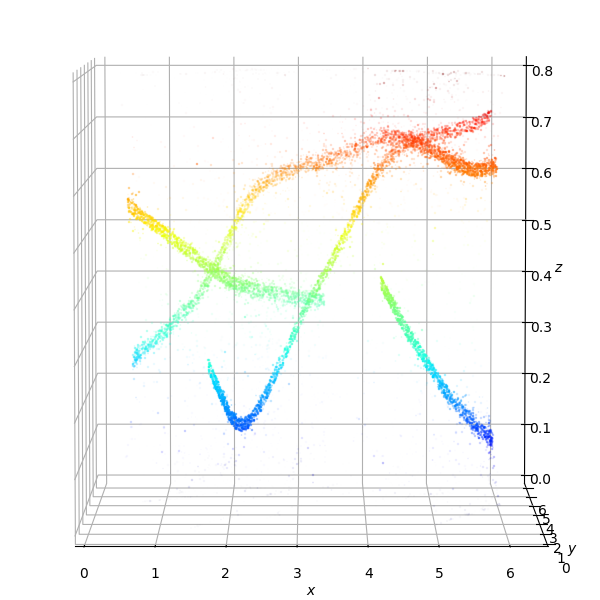}};
\node at (0,-4) {\includegraphics[width=0.33\linewidth]{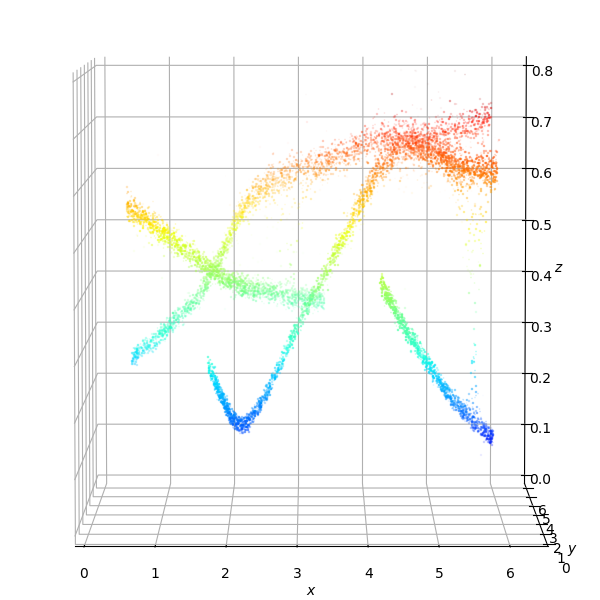}};
\node at (4,-4) {\includegraphics[width=0.33\linewidth]{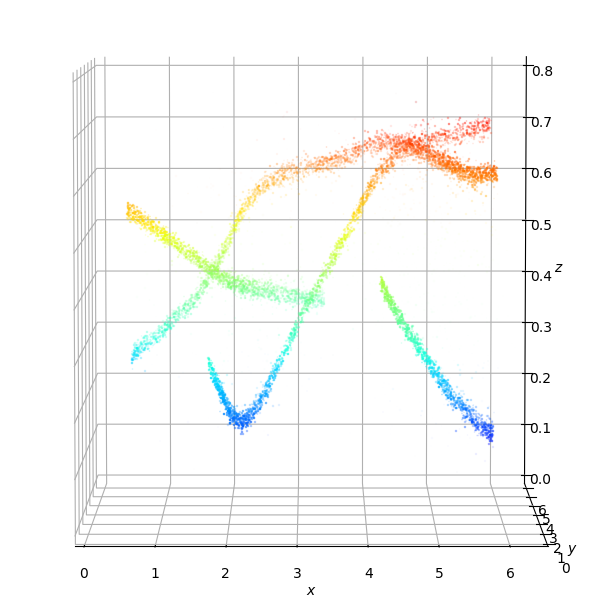}};
\node at (8,-4) {\includegraphics[width=0.33\linewidth]{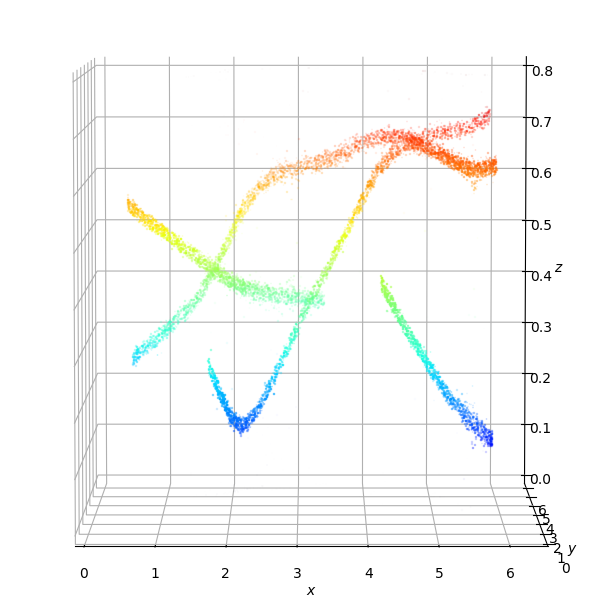}};
\node at (0,-8) {\includegraphics[width=0.33\linewidth]{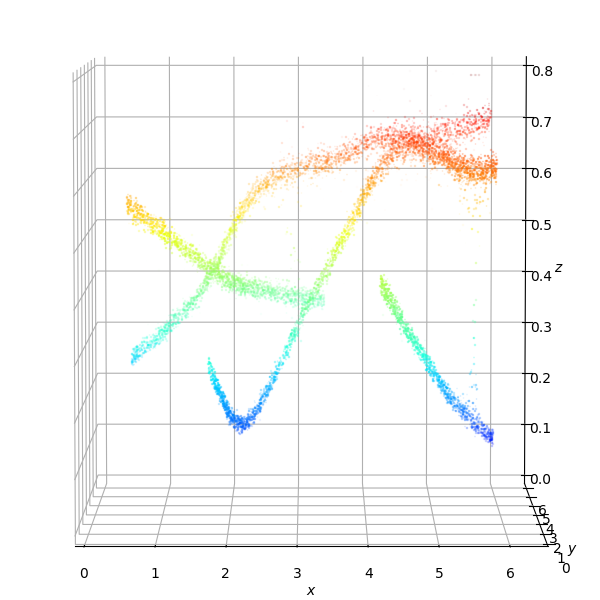}};
\node at (4,-8) {\includegraphics[width=0.33\linewidth]{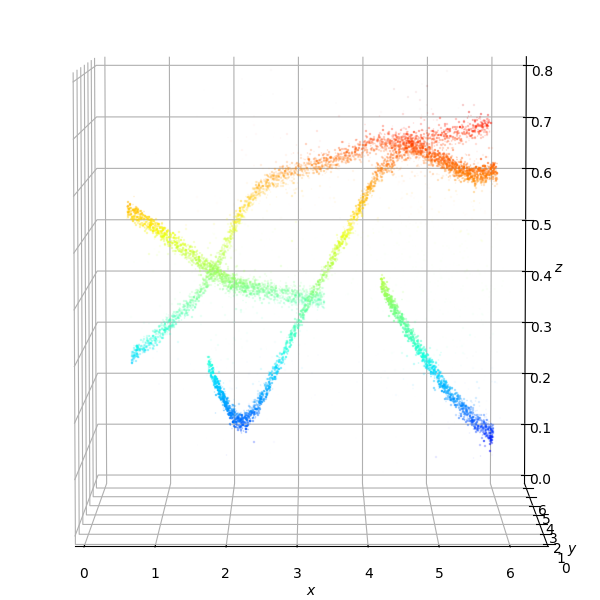}};
\node at (8,-8) {\includegraphics[width=0.33\linewidth]{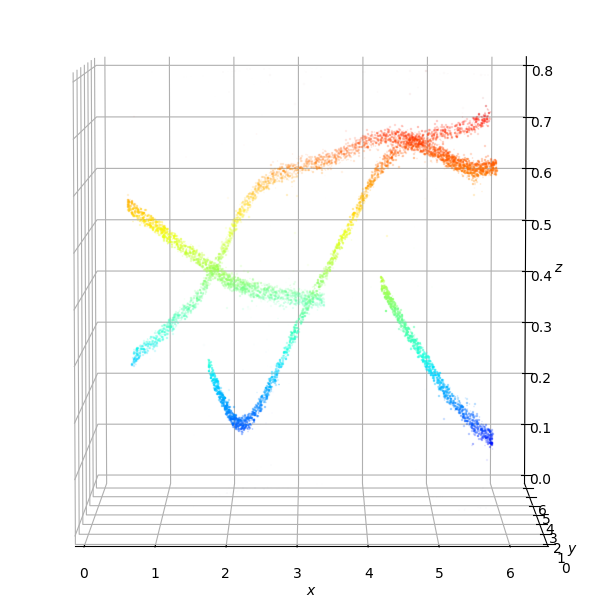}};
\node at (0,2.1) {MA-TIRF};\node at (4,2.1) {Astigmatism};\node at (8,2.1) {Double-Helix};
\node[rotate=90] at (-2.1,0) {$\K=1$};\node[rotate=90] at (-2.1,-4) {$\K=2$};\node[rotate=90] at (-2.1,-8) {$\K=3$};
\end{tikzpicture}
\caption{\label{sec:microscopy-fig:tub-Nfixe} Recovered structures for $\nMol=10$.}
\end{figure}

\section*{Conclusion}

This paper demonstrated from both theoretical and practical perspectives the Sliding Frank-Wolfe Algorithm, in particular when facing challenging non-translation invariant operator such as the Laplace kernels. Such operators lead to difficulties in estimating the spikes positions which is efficiently addressed by non-convex update step of the grid location. The BLASSO method, coupled with this Sliding Frank-Wolfe solver, is well adapted to these non-convolutive operators because it does not rely on spectral (Fourier) methods and can be analyzed theoretically through the prism of convex duality and vanishing certificates. 

\section*{Acknowledgement}

The authors would like to thank Laure Blanc-F\'eraud for initiating this collaboration and for stimulating discussions.
The work of Gabriel Peyr\'e has been supported by the European Research Council (ERC project NORIA).
The work of Emmanuel Soubies has been supported by the European Research Council (ERC project GlobalBioIm).

\section*{References}
\bibliographystyle{plain}
\bibliography{biblio,biblio3D}

\begin{thebibliography}{10}

\bibitem{axelrod1981cell}
Daniel Axelrod.
\newblock Cell-substrate contacts illuminated by total internal reflection
  fluorescence.
\newblock {\em The Journal of cell biology}, 89(1):141--145, 1981.

\bibitem{axelrod-total2008}
Daniel Axelrod.
\newblock Total internal reflection fluorescence microscopy.
\newblock {\em Methods in cell biology}, 89:169--221, 02 2008.

\bibitem{azais-spike2014}
Jean-Marc Aza\"{i}s, Yohann de~Castro, and Fabrice Gamboa.
\newblock Spike detection from inaccurate samplings.
\newblock {\em Appl. Comput. Harmon. Anal.}, 38(2):177--195, 2015.

\bibitem{beck-fista2009}
Amir Beck and Marc Teboulle.
\newblock A fast iterative shrinkage-thresholding algorithm for linear inverse
  problems.
\newblock {\em SIAM J. Imaging Sci.}, 2(1):183--202, 2009.

\bibitem{Betzig1642}
Eric Betzig, George~H. Patterson, Rachid Sougrat, O.~Wolf Lindwasser, Scott
  Olenych, Juan~S. Bonifacino, Michael~W. Davidson, Jennifer
  Lippincott-Schwartz, and Harald~F. Hess.
\newblock Imaging intracellular fluorescent proteins at nanometer resolution.
\newblock {\em Science}, 313(5793):1642--1645, 2006.

\bibitem{bhaskar-atomic2011}
Badri~Narayan Bhaskar, Gongguo Tang, and Benjamin Recht.
\newblock Atomic norm denoising with applications to line spectral estimation.
\newblock {\em IEEE Trans. Signal Process.}, 61(23):5987--5999, 2013.

\bibitem{davies-it2008}
Thomas Blumensath and Mike~E. Davies.
\newblock Iterative thresholding for sparse approximations.
\newblock {\em J. Fourier Anal. Appl.}, 14(5-6):629--654, 2008.

\bibitem{davies-iht2009}
Thomas Blumensath and Mike~E. Davies.
\newblock Iterative hard thresholding for compressed sensing.
\newblock {\em Applied and Computational Harmonic Analysis}, 27(3):265 -- 274,
  2009.

\bibitem{boulanger2014fast}
J{\'e}r{\^o}me Boulanger, Charles Gueudry, Daniel M{\"u}nch, Bertrand Cinquin,
  Perrine Paul-Gilloteaux, Sabine Bardin, Christophe Gu{\'e}rin, Fabrice
  Senger, Laurent Blanchoin, and Jean Salamero.
\newblock Fast high-resolution 3d total internal reflection fluorescence
  microscopy by incidence angle scanning and azimuthal averaging.
\newblock {\em Proceedings of the National Academy of Sciences},
  111(48):17164--17169, 2014.

\bibitem{boyd-adcg2015}
Nicholas Boyd, Geoffrey Schiebinger, and Benjamin Recht.
\newblock The alternating descent conditional gradient method for sparse
  inverse problems.
\newblock {\em SIAM J. Optim.}, 27(2):616--639, 2017.

\bibitem{boyd2004convex}
Stephen Boyd and Lieven Vandenberghe.
\newblock {\em Convex optimization}.
\newblock Cambridge university press, 2004.

\bibitem{bredies-inverse2013}
Kristian Bredies and Hanna~Katriina Pikkarainen.
\newblock Inverse problems in spaces of measures.
\newblock {\em ESAIM Control Optim. Calc. Var.}, 19(1):190--218, 2013.

\bibitem{cadzow-signal1988}
James~A Cadzow.
\newblock Signal enhancement-a composite property mapping algorithm.
\newblock {\em IEEE Transactions on Acoustics, Speech, and Signal Processing},
  36(1):49--62, 1988.

\bibitem{candes-superresolution2013}
Emmanuel~J. Cand\`es and Carlos Fernandez-Granda.
\newblock Super-resolution from noisy data.
\newblock {\em J. Fourier Anal. Appl.}, 19(6):1229--1254, 2013.

\bibitem{candes-towards2013}
Emmanuel~J. Cand\`es and Carlos Fernandez-Granda.
\newblock Towards a mathematical theory of super-resolution.
\newblock {\em Comm. Pure Appl. Math.}, 67(6):906--956, 2014.

\bibitem{catala-rank2017}
Paul Catala, Vincent Duval, and Gabriel Peyr\'e.
\newblock A low-rank approach to off-the-grid sparse deconvolution.
\newblock {\em PrePrint}, 2017.

\bibitem{chen-atomic1998}
Scott~Shaobing Chen, David~L. Donoho, and Michael~A. Saunders.
\newblock Atomic decomposition by basis pursuit.
\newblock {\em SIAM J. Sci. Comput.}, 20(1):33--61, 1998.

\bibitem{combettes-fb2005}
Patrick~L. Combettes and Val\'{e}rie~R. Wajs.
\newblock Signal recovery by proximal forward-backward splitting.
\newblock {\em Multiscale Model. Simul.}, 4(4):1168--1200, 2005.

\bibitem{condat-cadzow2015}
Laurent Condat and Akira Hirabayashi.
\newblock Cadzow denoising upgraded: a new projection method for the recovery
  of {D}irac pulses from noisy linear measurements.
\newblock {\em Sampl. Theory Signal Image Process.}, 14(1):17--47, 2015.

\bibitem{ferreira-2018tight}
Maxime~Ferreira Da~Costa and Wei Dai.
\newblock A tight converse to the spectral resolution limit via convex
  programming.
\newblock {\em arXiv preprint arXiv:1801.04761}, 2018.

\bibitem{daubechies-ist}
Ingrid Daubechies, Michel Defrise, and Christine De~Mol.
\newblock An iterative thresholding algorithm for linear inverse problems with
  a sparsity constraint.
\newblock {\em Comm. Pure Appl. Math.}, 57(11):1413--1457, 2004.

\bibitem{deCastro-exact2012}
Yohann de~Castro and Fabrice Gamboa.
\newblock Exact reconstruction using {B}eurling minimal extrapolation.
\newblock {\em J. Math. Anal. Appl.}, 395(1):336--354, 2012.

\bibitem{decastro-semi2015}
Yohann De~Castro, Fabrice Gamboa, Didier Henrion, and Jean-Bernard Lasserre.
\newblock Exact solutions to super resolution on semi-algebraic domains in
  higher dimensions.
\newblock {\em IEEE Trans. Inform. Theory}, 63(1):621--630, 2017.

\bibitem{prony}
G.C.F.M.R. de~Prony.
\newblock {\em Essai exp{\'e}rimental et analytique sur les lois de la
  dilatabilit{\'e} des fluides {\'e}lastiques et sur celles de la force
  expansive de la vapeur de l'eau et de la vapeur de l'alcool {\`a}
  diff{\'e}rentes temp{\'e}ratures, par R. Prony ...}
\newblock 1795.

\bibitem{demanet-recoverability2014}
Laurent Demanet and Nam Nguyen.
\newblock The recoverability limit for superresolution via sparsity.
\newblock {\em preprint arXiv:1502.01385}, 2015.

\bibitem{demyanov-1970approximate}
Vladimir~Fedorovich Demyanov and Aleksandr~Moiseevich Rubinov.
\newblock {\em Approximate methods in optimization problems}, volume~32.
\newblock Elsevier Publishing Company, 1970.

\bibitem{den-resolution1997}
AJ~Den~Dekker and A~Van~den Bos.
\newblock Resolution: a survey.
\newblock {\em JOSA A}, 14(3):547--557, 1997.

\bibitem{Denoyelle2017}
Quentin Denoyelle, Vincent Duval, and Gabriel Peyr\'{e}.
\newblock Support recovery for sparse super-resolution of positive measures.
\newblock {\em J. Fourier Anal. Appl.}, 23(5):1153--1194, 2017.

\bibitem{donoho-superresolution1992}
David~L. Donoho.
\newblock Super-resolution via sparsity constraints.
\newblock {\em SIAM J. Math. Anal., 23(5):1309-1331}, 9 1992.

\bibitem{donoho-adapting1999}
David~L. Donoho and Iain~M. Johnstone.
\newblock Adapting to unknown smoothness via wavelet shrinkage.
\newblock {\em J. Amer. Statist. Assoc.}, 90(432):1200--1224, 1995.

\bibitem{dos2016}
Marcelina~Cardoso Dos~Santos, R{\'e}gis D{\'e}turche, Cyrille V{\'e}zy, and
  Rodolphe Jaffiol.
\newblock Topography of cells revealed by variable-angle total internal
  reflection fluorescence microscopy.
\newblock {\em Biophysical journal}, 111(6):1316--1327, 2016.

\bibitem{duval-ndsc2017}
Vincent Duval.
\newblock {A characterization of the Non-Degenerate Source Condition in
  Super-Resolution}.
\newblock working paper or preprint, December 2017.

\bibitem{duval-exact2013}
Vincent Duval and Gabriel Peyr\'{e}.
\newblock Exact support recovery for sparse spikes deconvolution.
\newblock {\em Found. Comput. Math.}, 15(5):1315--1355, 2015.

\bibitem{duval-thingridsI2017}
Vincent Duval and Gabriel Peyr\'{e}.
\newblock Sparse spikes super-resolution on thin grids {I}: the {L}asso.
\newblock {\em Inverse Problems}, 33(5):055008, 29, 2017.

\bibitem{duval-thingridsII2017}
Vincent Duval and Gabriel Peyr\'{e}.
\newblock Sparse spikes super-resolution on thin grids {II}: the continuous
  basis pursuit.
\newblock {\em Inverse Problems}, 33(9):095008, 42, 2017.

\bibitem{efron-lars2004}
Bradley Efron, Trevor Hastie, Iain Johnstone, and Robert Tibshirani.
\newblock Least angle regression.
\newblock {\em Ann. Statist.}, 32(2):407--499, 2004.
\newblock With discussion, and a rejoinder by the authors.

\bibitem{eftekhari2015greed}
Armin Eftekhari and Michael~B Wakin.
\newblock Greed is super: A fast algorithm for super-resolution.
\newblock {\em arXiv preprint arXiv:1511.03385}, 2015.

\bibitem{elghaoui-safe2010}
Laurent El~Ghaoui, Vivian Viallon, and Tarek Rabbani.
\newblock Safe feature elimination in sparse supervised learning.
\newblock {\em Pac. J. Optim.}, 8(4):667--698, 2012.

\bibitem{fernandez-support2013}
C.~Fernandez{-}Granda.
\newblock Support detection in super-resolution.
\newblock {\em Proc. Proceedings of the 10th International Conference on
  Sampling Theory and Applications}, pages 145--148, 2013.

\bibitem{carlos-super2015}
Carlos Fernandez-Granda.
\newblock Super-resolution of point sources via convex programming.
\newblock {\em Inf. Inference}, 5(3):251--303, 2016.

\bibitem{figueiredo-em2003}
M\'{a}rio A.~T. Figueiredo and Robert~D. Nowak.
\newblock An {EM} algorithm for wavelet-based image restoration.
\newblock {\em IEEE Trans. Image Process.}, 12(8):906--916, 2003.

\bibitem{flinthweiss2018}
Axel Flinth and Pierre Weiss.
\newblock Exact solutions of infinite dimensional total-variation regularized
  problems.
\newblock {\em Information and Inference: A Journal of the IMA}, page iay016,
  2018.

\bibitem{frank-fw1956}
Marguerite Frank and Philip Wolfe.
\newblock An algorithm for quadratic programming.
\newblock {\em Naval Res. Logist. Quart.}, 3:95--110, 1956.

\bibitem{gustafsson-surpassing2000}
M.~G.~L. Gustafsson.
\newblock Surpassing the lateral resolution limit by a factor of two using
  structured illumination microscopy.
\newblock {\em Journal of Microscopy}, 198(2):82--87, 2000.

\bibitem{harchaoui-2015conditional}
Zaid Harchaoui, Anatoli Juditsky, and Arkadi Nemirovski.
\newblock Conditional gradient algorithms for norm-regularized smooth convex
  optimization.
\newblock {\em Mathematical Programming}, 152(1-2):75--112, 2015.

\bibitem{Hell:94}
Stefan~W. Hell and Jan Wichmann.
\newblock Breaking the diffraction resolution limit by stimulated emission:
  stimulated-emission-depletion fluorescence microscopy.
\newblock {\em Opt. Lett.}, 19(11):780--782, Jun 1994.

\bibitem{henriques2010quickpalm}
Ricardo Henriques, Mickael Lelek, Eugenio~F Fornasiero, Flavia Valtorta,
  Christophe Zimmer, and Musa~M Mhlanga.
\newblock Quickpalm: 3d real-time photoactivation nanoscopy image processing in
  imagej.
\newblock {\em Nature methods}, 7(5):339, 2010.

\bibitem{herzet-ols2014}
C\'{e}dric Herzet, Ang\'{e}lique Dr\'{e}meau, and Charles Soussen.
\newblock Relaxed recovery conditions for {OMP}/{OLS} by exploiting both
  coherence and decay.
\newblock {\em IEEE Trans. Inform. Theory}, 62(1):459--470, 2016.

\bibitem{hess-ultra2007}
Samuel Hess, Thanu P~K~Girirajan, and Michael Mason.
\newblock Ultra-high resolution imaging by fluorescence photoactivation
  localization microscopy.
\newblock {\em Biophysical journal}, 91:4258--72, 01 2007.

\bibitem{SMLM}
Seamus Holden and Daniel Sage.
\newblock Super-resolution fight club.
\newblock 10:152 EP --, 02 2016.

\bibitem{hua-pencils1990}
Yingbo Hua and Tapan~K. Sarkar.
\newblock Matrix pencil method for estimating parameters of exponentially
  damped/undamped sinusoids in noise.
\newblock {\em IEEE Trans. Acoust. Speech Signal Process.}, 38(5):814--824,
  1990.

\bibitem{huang-three2008}
Bo~Huang, Wenqin Wang, Mark Bates, and Xiaowei Zhuang.
\newblock Three-dimensional super-resolution imaging by stochastic optical
  reconstruction microscopy.
\newblock {\em Science}, 2008.

\bibitem{huang-backtracking2011}
Honglin Huang and Anamitra Makur.
\newblock Backtracking-based matching pursuit method for sparse signal
  reconstruction.
\newblock {\em IEEE Signal Processing Letters}, 18(7):391--394, 2011.

\bibitem{huang-3dastmultifp}
Jiaqing Huang, Mingzhai Sun, Kristyn Gumpper, Yuejie Chi, and Jianjie Ma.
\newblock 3d multifocus astigmatism and compressed sensing (3d macs) based
  superresolution reconstruction.
\newblock {\em Biomed. Opt. Express}, 6(3):902--917, Mar 2015.

\bibitem{huang2017super}
Jiaqing Huang, Mingzhai Sun, Jianjie Ma, and Yuejie Chi.
\newblock Super-resolution image reconstruction for high-density
  three-dimensional single-molecule microscopy.
\newblock {\em IEEE Transactions on Computational Imaging}, 3(4):763--773,
  2017.

\bibitem{jacques2008geometrical}
Laurent Jacques and Christophe De~Vleeschouwer.
\newblock A geometrical study of matching pursuit parametrization.
\newblock {\em IEEE Transactions on Signal Processing}, 56(7):2835--2848, 2008.

\bibitem{jaggi2013revisiting}
Martin Jaggi.
\newblock Revisiting frank-wolfe: Projection-free sparse convex optimization.
\newblock In {\em ICML (1)}, pages 427--435, 2013.

\bibitem{juette-three2008}
Manuel Juette, Travis J~Gould, Mark Lessard, Michael Mlodzianoski, Bhupendra
  S~Nagpure, Brian Thomas~Bennett, Samuel Hess, and Joerg Bewersdorf.
\newblock Three-dimensional sub-100 nm resolution fluorescence microscopy of
  thick samples.
\newblock {\em Nature methods}, 5:527--9, 07 2008.

\bibitem{Kailath_1990}
Thomas Kailath.
\newblock {ESPRIT}--estimation of signal parameters via rotational invariance
  techniques.
\newblock {\em Optical Engineering}, 29(4):296, 1990.

\bibitem{kirshner2013}
Hagai Kirshner, C{\'e}dric Vonesch, and Michael Unser.
\newblock Can localization microscopy benefit from approximation theory?
\newblock In {\em 2013 IEEE 10th International Symposium on Biomedical Imaging
  (ISBI)}, pages 588--591. IEEE, 2013.

\bibitem{lasserre-global2004}
Jean~Bernard Lasserre.
\newblock Global optimization with polynomials and the problem of moments.
\newblock {\em SIAM J. Optim.}, 11(3):796--817, 2000/01.

\bibitem{lasserre-moments2009}
Jean~Bernard Lasserre.
\newblock Moments, positive polynomials and their applications.
\newblock {\em Imperial College Press Optimization Series}, 1:xxii+361, 2010.

\bibitem{levitin-constrained1966}
Evgenii~Solomonovich Levitin and Boris~Teodorovich Polyak.
\newblock Constrained minimization methods.
\newblock {\em Zhurnal Vychislitel'noi Matematiki i Matematicheskoi Fiziki},
  6(5):787--823, 1966.

\bibitem{LiangLinearFB}
Jingwei Liang, Jalal Fadili, and Gabriel Peyr\'{e}.
\newblock Activity identification and local linear convergence of
  forward-backward-type methods.
\newblock {\em SIAM J. Optim.}, 27(1):408--437, 2017.

\bibitem{liao-music2014}
Wenjing Liao and Albert Fannjiang.
\newblock M{USIC} for single-snapshot spectral estimation: stability and
  super-resolution.
\newblock {\em Appl. Comput. Harmon. Anal.}, 40(1):33--67, 2016.

\bibitem{mallat-mp1994}
St\'ephane Mallat and Zhifeng Zhang.
\newblock Matching pursuit with time-frequency dictionaries.
\newblock {\em Signal Processing, IEEE Transactions on}, 41:3397 -- 3415, 01
  1994.

\bibitem{salmon-screening2017}
Mathurin Massias, Alexandre Gramfort, and Joseph Salmon.
\newblock From safe screening rules to working sets for faster lasso-type
  solvers.
\newblock {\em arXiv preprint arXiv:1703.07285}, 2017.

\bibitem{candes-stable2014}
Veniamin~I. Morgenshtern and Emmanuel~J. Cand\`es.
\newblock Super-resolution of positive sources: the discrete setup.
\newblock {\em SIAM J. Imaging Sci.}, 9(1):412--444, 2016.

\bibitem{rama-three2009}
Sri Rama Prasanna~Pavani, Michael A~Thompson, Julie S~Biteen, Samuel Lord,
  Na~Liu, Robert Twieg, Rafael Piestun, and William Moerner.
\newblock Three-dimensional, single-molecule fluorescence imaging beyond the
  diffraction limit by using a double-helix point spread function.
\newblock {\em Proceedings of the National Academy of Sciences of the United
  States of America}, 106:2995--9, 03 2009.

\bibitem{rockafellar2015convex}
Ralph~Tyrell Rockafellar.
\newblock {\em Convex analysis}.
\newblock Princeton university press, 2015.

\bibitem{rudin-real1987}
Walter Rudin.
\newblock {\em Real and Complex Analysis, 3rd Ed.}
\newblock McGraw-Hill, Inc., New York, NY, USA, 1987.

\bibitem{rust2006sub}
Michael~J Rust, Mark Bates, and Xiaowei Zhuang.
\newblock Sub-diffraction-limit imaging by stochastic optical reconstruction
  microscopy (storm).
\newblock {\em Nature methods}, 3(10):793--796, 2006.

\bibitem{sage-quantitative2015}
Daniel Sage, Hagai Kirshner, Thomas Pengo, Nico Stuurman, Junhong Min, Suliana
  Manley, and Michael Unser.
\newblock Quantitative evaluation of software packages for single-molecule
  localization microscopy.
\newblock {\em Nature methods}, 12, 06 2015.

\bibitem{Sage362517}
Daniel Sage, Thanh-An Pham, Hazen Babcock, Tomas Lukes, Thomas Pengo, Ramraj
  Velmurugan, Alex Herbert, Anurag Agarwal, Silvia Colabrese, Ann Wheeler, Anna
  Archetti, Bernd Rieger, Raimund Ober, Guy~M. Hagen, Jean-Baptiste Sibarita,
  Jonas Ries, Ricardo Henriques, Michael Unser, and Seamus Holden.
\newblock Super-resolution fight club: A broad assessment of 2d \& 3d
  single-molecule localization microscopy software.
\newblock {\em bioRxiv}, 2018.

\bibitem{schmidt-multiple1986}
Ralph Schmidt.
\newblock Multiple emitter location and signal parameter estimation.
\newblock {\em IEEE transactions on antennas and propagation}, 34(3):276--280,
  1986.

\bibitem{Soubies2016}
Emmanuel Soubies, S{\'e}bastien Schaub, Agata Radwanska, Ellen Van
  Obberghen-Schilling, Laure Blanc-F{\'e}raud, and Gilles Aubert.
\newblock A framework for multi-angle tirf microscope calibration.
\newblock In {\em Biomedical Imaging (ISBI), 2016 IEEE 13th International
  Symposium on}, pages 668--671. IEEE, 2016.

\bibitem{soussen-omp2013}
Charles Soussen, R\'{e}mi Gribonval, J\'{e}r\^{o}me Idier, and C\'{e}dric
  Herzet.
\newblock Joint {$k$}-step analysis of orthogonal matching pursuit and
  orthogonal least squares.
\newblock {\em IEEE Trans. Inform. Theory}, 59(5):3158--3174, 2013.

\bibitem{soussen-homotopy2015}
Charles Soussen, J\'{e}r\^{o}me Idier, Junbo Duan, and David Brie.
\newblock Homotopy based algorithms for {$\ell_0$}-regularized least-squares.
\newblock {\em IEEE Trans. Signal Process.}, 63(13):3301--3316, 2015.

\bibitem{tang2015resolution}
Gongguo Tang.
\newblock Resolution limits for atomic decompositions via markov-bernstein type
  inequalities.
\newblock In {\em Sampling Theory and Applications (SampTA), 2015 International
  Conference on}, pages 548--552. IEEE, 2015.

\bibitem{tang-justdiscretize2013}
Gongguo Tang, Badri~Narayan Bhaskar, and Benjamin Recht.
\newblock Sparse recovery over continuous dictionaries-just discretize.
\newblock {\em Signals, Systems and Computers, 2013 Asilomar Conference on},
  pages 1043--1047, 2013.

\bibitem{tibshirani-regression1994}
Robert Tibshirani.
\newblock Regression shrinkage and selection via the lasso.
\newblock {\em J. Roy. Statist. Soc. Ser. B}, 58(1):267--288, 1996.

\bibitem{toh-fistasvd2010}
Kim-Chuan Toh and Sangwoon Yun.
\newblock An accelerated proximal gradient algorithm for nuclear norm
  regularized linear least squares problems.
\newblock {\em Pac. J. Optim.}, 6(3):615--640, 2010.

\bibitem{tropp-omp2008}
Joel~A. Tropp and Anna~C. Gilbert.
\newblock Signal recovery from random measurements via orthogonal matching
  pursuit.
\newblock {\em IEEE Trans. Inform. Theory}, 53(12):4655--4666, 2007.

\bibitem{tseng-convergence2001}
Paul Tseng.
\newblock Convergence of a block coordinate descent method for
  nondifferentiable minimization.
\newblock {\em J. Optim. Theory Appl.}, 109(3):475--494, 2001.

\bibitem{wu-coordinate2008}
Tong~Tong Wu and Kenneth Lange.
\newblock Coordinate descent algorithms for lasso penalized regression.
\newblock {\em Ann. Appl. Stat.}, 2(1):224--244, 2008.

\bibitem{Zhang2007}
Bo~Zhang, Josiane Zerubia, and Jean-Christophe Olivo-Marin.
\newblock Gaussian approximations of fluorescence microscope point-spread
  function models.
\newblock {\em Applied optics}, 46(10):1819--1829, 2007.

\bibitem{Zheng2018}
Cheng Zheng, Guangyuan Zhao, Wenjie Liu, Youhua Chen, Zhimin Zhang, Luhong Jin,
  Yingke Xu, Cuifang Kuang, and Xu~Liu.
\newblock Three-dimensional super-resolved live cell imaging through polarized
  multi-angle tirf.
\newblock {\em Optics letters}, 43(7):1423--1426, 2018.

\end{thebibliography}

\end{document}